\documentclass[a4paper, 11pt, headings = small, abstract]{scrartcl}
\usepackage[english]{babel}

\usepackage{amscd,amsmath,amsthm,array,hhline,mathrsfs, enumerate, booktabs,fancybox,calc,textcomp,xcolor,graphicx,bbm,xspace,nicefrac,stmaryrd,url,arcs,listings}
\usepackage[theoremfont]{newpxtext}
\usepackage[vvarbb, upint, bigdelims]{newpxmath}
\usepackage[scr=boondoxupr, scrscaled = 1.05, cal = euler, calscaled =1.05]{mathalpha}
\usepackage[scaled=0.95]{inconsolata}
\usepackage{accents}
\usepackage{dsfont}
\DeclareMathAlphabet{\mathsf}{OT1}{\sfdefault}{m}{n}
\SetMathAlphabet{\mathsf}{bold}{OT1}{\sfdefault}{b}{n}
\usepackage[utf8]{inputenc}

\usepackage{enumitem}
\usepackage{etoolbox}
\usepackage[normalem]{ulem}
\usepackage{mathtools}
\usepackage{geometry}
\usepackage{float}

\usepackage{bm}
\usepackage{tikz}
\usepackage[outdir =./]{epstopdf}
\usepackage[citestyle = numeric-comp,bibstyle = numeric, firstinits=true, natbib = true, backend = bibtex,maxbibnames=99, url=false]{biblatex}
\numberwithin{equation}{section}

\usepackage{chngcntr}
\counterwithin{figure}{section}

\usepackage{xcolor}
\usepackage[colorlinks=true, allcolors=myteal]{hyperref}

\definecolor{WIMgreen}{RGB}{60 134 132}
\definecolor{UMblue}{RGB}{4 47 86}
\definecolor{myteal}{RGB}{0 123 137}
\definecolor{material_green}{RGB}{27 43 52}
\definecolor{dracula_pink}{RGB}{180 93 149}
\definecolor{dracula_blue}{RGB}{40 42 54}
\definecolor{dracula_turq}{RGB}{92 143 159}
\definecolor{dracula_orange}{RGB}{255 184 108}
\definecolor{material_petrol}{RGB}{2 119 189}
\definecolor{Purple}{RGB}{103 58 183}

\theoremstyle{plain}
\newtheorem{theorem}{Theorem}[section]
\newtheorem*{theorem*}{Theorem}
\newtheorem{proposition}[theorem]{Proposition}
\newtheorem{lemma}[theorem]{Lemma}
\newtheorem{corollary}[theorem]{Corollary}

\theoremstyle{definition}
\newtheorem{definition}[theorem]{Definition}
\theoremstyle{assumption}

\theoremstyle{remark}
\newtheorem{remark}[theorem]{Remark}

\setkomafont{sectioning}{\normalcolor\bfseries}
\setkomafont{descriptionlabel}{\normalcolor\bfseries}
\setkomafont{section}{\Large}
\setkomafont{subsection}{\large}
\setkomafont{subsubsection}{\large}
\setkomafont{paragraph}{\large}
\setkomafont{subparagraph}{\large}
\setkomafont{author}{\large}
\setkomafont{date}{\normalsize}

\def\supp{\operatorname{supp}}

\def\E{\mathbb{E}}
\def\G{\mathbb{G}}
\def\F{\mathbb{F}}

\def\N{\mathbb{N}}

\def\N{\mathbb{N}}

\def\R{\mathbb{R}}
\definecolor{darkred}{rgb}{0,0.6,0}

\def\cX{\mathcal{X}}
\def\cC{\mathcal{C}}

\newcommand{\cB}{\mathcal{B}}
\newcommand{\cD}{\mathcal{D}}

\newcommand{\cF}{\mathcal{F}}
\newcommand{\cG}{\mathcal{G}}

\newcommand{\cJ}{\mathcal{J}}

\newcommand{\cO}{\mathcal{O}}
\newcommand{\PP}{\mathbb{P}}

\renewcommand{\Re}{\operatorname{Re}}

\renewcommand{\hat}{\widehat}

\renewcommand{\tilde}{\widetilde}%

\newcommand{\overbar}[1]{\mkern 1.5mu\overline{\mkern-1.5mu#1\mkern-1.5mu}\mkern 1.5mu}

\restylefloat{table}

\newcommand*\diff{\mathop{}\!\mathrm{d}}
\newcommand{\cH}{\mathcal{H}}
\newcommand{\one}{\mathbf{1}}

\newcommand{\vertiii}[1]{{\left\vert\kern-0.25ex\left\vert\kern-0.25ex\left\vert #1
    \right\vert\kern-0.25ex\right\vert\kern-0.25ex\right\vert}}

\let\originalleft\left
\let\originalright\right
\renewcommand{\left}{\mathopen{}\mathclose\bgroup\originalleft}
\renewcommand{\right}{\aftergroup\egroup\originalright}

\newcommand{\bbGamma}{{\mathpalette\makebbGamma\relax}}
\newcommand{\makebbGamma}[2]{%
  \raisebox{\depth}{\scalebox{1}[-1]{$\mathsurround=0pt#1\mathds{L}$}}%
}

\makeatletter
\newcommand{\specificthanks}[1]{\@fnsymbol{#1}}
\makeatother

\allowdisplaybreaks

\makeatother
\title{\fontsize{16}{19} \selectfont Stability of overshoots of Markov additive processes}
\author{Leif Döring\thanks{University of Mannheim, Institute of Mathematics, B6 26, 68159 Mannheim, Germany. \newline Email: \href{mailto:doering@uni-mannheim.de}{doering@uni-mannheim.de}/\href{mailto:ltrottne@mail.uni-mannheim.de}{trottner@math.au.dk}} \and Lukas Trottner\textsuperscript{\specificthanks{1},}\thanks{Supported by the Research Training Group ”Statistical Modeling of Complex Systems” funded by the German Science Foundation.}}
\date{\vspace{-20pt}}

\bibliography{lit_overshoot}

\begin{document}
\maketitle
\normalsize
\begin{abstract}
We prove precise stability results for overshoots of Markov additive processes (MAPs) with finite modulating space. Our approach is based on the Markovian nature of overshoots of MAPs whose mixing and ergodic properties are investigated in terms of the characteristics of the MAP. On our way we extend fluctuation theory of MAPs, contributing among others to the understanding of the Wiener--Hopf factorization for MAPs by generalizing Vigon's équations amicales inversés known for L\'evy processes. Using the Lamperti transformation the results can be applied to self-similar Markov processes. Among many possible applications, we study the mixing behavior of stable  processes sampled at first hitting times as a concrete example.
\end{abstract}

\section{Introduction}
\subsection{Background and aims of the article}
Overshoots of a L\'evy process $\xi$, defined by
$$\cO_x = \xi_{T_x} - x,\quad x\geq 0,$$
on $\{T_x < \infty\}$, where $T_x := \inf\{t \geq 0: \xi_t > x\}$, are classical objects in the study of L\'evy processes. Their asymptotic analysis is essentially rooted in renewal theory for random walks and has gained a lot of interest in the past two decades starting with the observation in \cite{BertoinHarnSteutel1999} that classical limit theorems for the residual time chain of renewal processes have a natural analogue in weak convergence of overshoots of subordinators to a non-trivial limiting distribution. Besides applications and extensions in ruin theory for insurance risk processes driven by L\'evy processes (see \cite{kluppelberg2004,park2008,griffin2016}), this observation was used to  explain the entrance behavior of positive self-similar Markov processes (pssMps) at the origin. Using the Lamperti transformation for transient pssMps one can show that a pssMp can be started from the origin if and only if the overshoots of the underlying L\'evy process converge weakly as the overshoot level $x$ diverges to $+\infty$ (see \cite{BertoinSavov2011,chaumont2012}). This was generalized in \cite{dereich2017} to the question of how to start real self-similar Markov processes (rssMps)  from the origin. Methods for rssMps are similar to those for pssMps replacing the L\'evy processes $\xi$ in the Lamperti transformation by \textit{Markov additive processes} $(\xi,J)$, MAPs in the following, with finite modulating space $\{-1,1\}$. The corresponding transformation is usually called Lamperti--Kiu transform. MAPs $(\xi,J)$ are also called Markov modulated L\'evy processes, due to the \textit{ordinator} $\xi$ behaving as a L\'evy process in between jumps of a \textit{modulating chain} $J$, with the L\'evy triplet of $\xi$ being determined by the current state of $J$. The limiting behavior of overshoots of MAPs, defined by
$$(\cO_x,\cJ_x) = \big(\xi_{T_x} - x, J_{T_x}\big), \quad x \geq 0,$$
on $\{T_x < \infty\}$, where $T_x := \inf\{t \geq 0: \xi_t > x\}$, then plays the same role for the entrance law at $0$ of rssMps, as do overshoots of L\'evy processes for pssMps.

The aim of this article is to explore in detail mixing and ergodicity of overshoots of MAPs. We study the convergence in total variation norm, including conditions for polynomial and exponential rates of convergence. Based on fluctuation theory of MAPs developed in \cite{dereich2017} we will use the Meyn and Tweedie approach to stability of continuous time Markov processes (see for instance \cite{MeynTweedie1993,MeynTweedie1993b,MeynTweedie1993c,Tweedie1994}) to demonstrate that overshoot convergence can be much more finely analyzed once we take the perspective on overshoots as a Markov process, where the subsequent spatial levels that are passed by the ordinator $\xi$ serve as time index for the overshoot process $(\cO,\cJ) = (\cO_t,\cJ_t)_{t \geq 0}$. This idea is inspired by the observation that for the overshoot process of a L\'evy subordinator $\sigma$, inverse local time at $0$ is given by $\sigma$ itself \cite{bertoin2000}. For this special case, this opens the door to powerful results of excursion theory for general Markov processes and allows, among others, to derive explicit formulas for the invariant measure and resolvent of the overshoot process of a L\'evy subordinator in terms of its triplet \cite{getoor1979,blumenthal1992}. We generalize these findings to the MAP situation and consequently make use of the analytical tractability of overshoots to analyze their ergodic behavior. For the particular case of L\'evy processes, the results can be interpreted as a natural continuous time generalization of results on ergodicity and exponential convergence of the residual time chain belonging to a renewal process, which can be found in the standard references on stability of discrete time Markov chains, Meyn and Tweedie \cite{MeynTweedie2009} and Nummelin \cite{nummelin1984}. Extensions of renewal theory for random walks to discrete time MAPs (often called \textit{Markov random walks}) were treated in \cite{cinlar1969, jacod71, kesten1974, lalley1984, alsmeyer1994} among others.

Our fine analysis of overshoot stability of MAPs is not only inspired by a theoretical desire to understand their asymptotics, but also by a practical need to develop statistical and numerical procedures to get hold of the \textit{ascending ladder height process} $(H^+,J^+)$ of a given MAP $(\xi,J)$. This process is one of the cornerstones of fluctuation theory of MAPs and is theoretically accessible by means of the Wiener--Hopf factorization. However, its explicit analytical characteristics are in general unknown, with a notable exception being the factorization of the MAP associated to an $\alpha$-stable L\'evy process via the Lamperti--Kiu transform, which was found in \cite{kyprianou2016}. Due to its intimate connection with the running supremum of the MAP, observing $(\xi,J)$ at first hitting times offers all information needed to determine $(H^+,J^+)$ in numerical or statistical procedures. For a recent account of fluctuation theory of Markov random walks we refer to \cite{alsmeyer2018}.

The results of the present article have applications in optimal control problems based on MAPs, see e.g.\ the recent article \cite{christensen2020} for the more particular case of a Lévy driven impulse control problem. There, the generator of the ascending ladder height process is decisive for determining optimal threshold levels of a desired reflection strategy. Thus, under uncertainty concerning the underlying Lévy process, efficient statistical estimation of the ascending ladder height process is needed. Such data-driven reflection strategies are investigated in \cite{christensen21} based on the stability results from this article. Moreover, parametric estimation becomes feasible for the Lévy system of MAPs -- which encodes the jumps of a MAP in analogy to the Lévy measure of a Lévy process --  with explicit overshoot distributions based on the MAP observed at first hitting times $(T_{n\Delta})_{n \in \N_0}$ for some step size $\Delta > 0$. Such observation scheme can be described as \textit{stochastic low frequency scheme} as opposed to deterministic low and high frequency schemes usually encountered in parametric inference of stochastic processes (see \cite{levymatters4} for an overview in the context of Lévy processes) or the stochastic high-frequency scheme analyzed in \cite{rosenbaum2011} for Lévy processes.  Furthermore, nonparametric statistical estimation procedures for the ascending ladder height characteristics can be developed based on our observation that under some natural conditions, the overshoot process is exponentially $\beta$-\textit{mixing}. This property, describing rigorously asymptotic independence of the past and the future of a Markov process, is earmarked in \cite{dexheimer20} as a central building block to nonparametric statistical analysis of non-reversible ergodic Markov processes. Hence, our results indicate how to include MAPs (which are non-ergodic) in an ergodic statistical setting by considering the space-time transform introduced in form of overshoots. 

Due to recent applications of MAPs we also expect applications of our mixing estimates in other fields of probability theory such as planar maps (see for instance \cite{bertoin18}). We highlight this point by  making use of the the Lamperti--Kiu transform to translate the mixing behavior of MAPs into mixing bounds for self-similar Markov processes sampled at first hitting times. Further applications to non-parametric statistical estimation for MAPs, L\'evy processes and equivalently self-similar Markov processes will be subject to future research.

\subsection{Organization of the paper and main result}
We start in Section \ref{sec: map} with formally introducing Markov additive processes and summarizing some results belonging to their fluctuation theory as given in \cite{dereich2017}. We then proceed in Section \ref{sec: stability} with the stability analysis of MAP overshoots, starting with the rigorous description of their Markovian nature and then studying important concepts from the theory of stability for Markov processes such as Harris recurrence, invariant measures and petite sets. For the reader unfamiliar with these concepts, we have devoted Appendix \ref{sec: markov stability} to a brief summary of stability of Markov processes in the sense of Meyn and Tweedie, additionally clarifying some results in the literature and developing a new technique for deriving invariant measures of Markov processes based on a limiting argument involving the resolvent of the process in Proposition \ref{invariant measure resolvent}. Moreover, some general terms for Markov processes, such as Borel right processes, the Feller property and resolvents are summarized in Appendix \ref{sec: markov stability} without further explanations in the main body of the text. With this setup we come to our primary goal, the ergodicity analysis of overshoots.
Our main results in this respect, taking also account of the developments in Section \ref{sec: vigon} described below, can be informally summarized as follows.

\begin{theorem*}
Suppose that the MAP $(\xi,J)$ is upward regular, $J$ is irreducible and the ascending ladder height MAP $(H^+,J^+)$ has a finite first moment. Under mild assumptions on the L\'evy system of $(\xi,J)$, $(\cO_t,\cJ_t)_{t \geq 0}$ converges in total variation to a unique stationary distribution, which encodes the characteristics of the ascending ladder height MAP. If moreover the jump measures associated to the MAP's L\'evy system possess a common (exponential) moment, then the convergence takes place at (exponential) polynomial speed and overshoots are (exponentially) polynomially $\beta$-mixing.
\end{theorem*}

This will be made precise in a sequence of theorems in Section \ref{sec: stability}. In Theorem \ref{theo: ergodic} we establish conditions on either the creeping probabilities of the subordinators associated to the ascending ladder height MAP or its L\'evy system that guarantee total variation convergence of overshoots. Theorem \ref{theo: exp ergodicity} and Theorem \ref{theo: mixing} build on this result, giving exponential/polynomial ergodicity and the exponential/polynomial $\beta$-mixing property, respectively.

Section \ref{sec: vigon} is devoted to finding conditions on the L\'evy system of the parent MAP, which imply the required assumptions on $(H^+,J^+)$ for the ergodic results of the previous section, thus enhancing significantly our understanding of asymptotics of MAP overshoots. The tool we develop for this purpose is an extension of Vigon`s équations amicales inversés for Lévy processes given in \cite{vigon2002} to MAPs. These equations analytically relate the Lévy systems of $(\xi,J)$ and $(H^+,J^+)$, which makes inference of distributional properties of the ascending ladder height process based on the characteristics of the parent MAP possible.

Finally, in Section \ref{sec: self similar} we apply our $\beta$-mixing result for MAPs to real self-similar Markov processes sampled at symmetric first hitting times by exploiting the Lamperti--Kiu transform, which bridges these two classes of processes. As an even more specific application, we then consider the mixing behavior of $\alpha$-stable L\'evy processes and ergodicity of overshoots of the associated Lamperti-stable MAP. 

\subsection{Basic notation}
For a given space $\cX$ we will denote by $\cB(\cX)$ its Borel $\sigma$-algebra and by $\cB_+(\cX)$ and $\cB_b(\cX)$ the space of positive, resp.\ bounded real-valued functions on $\cX$. If $\cX$ is locally compact, then $\cC_0(\cX)$ denotes the space of continuous, real-valued functions on $\cX$ vanishing at infinity. If $\cX = \R$ and $\mu$ is a measure on $(\R, \cB(\R))$, we let $\overbar{\mu}(y) \coloneq \mu((y,\infty))$, $y \in \R$, be its tail. $\mathrm{Leb}(\diff{x})$ denotes the Lebesgue measure on $\R$ and $\mathrm{Leb}_+(\diff{x})$ is its restriction to $\R_+ = [0,\infty)$.

\section{Markov additive processes and their fluctuation theory} \label{sec: map}
We start with introducing Markov additive processes with finite modulating space. For the general theory of Markov additive processes the reader may consult the landmark papers of \c{C}inlar \cite{cinlar1972,cinlar1974}, a good start for the particular case of finite modulating space is \cite[Chapter XI]{asmussen2003}, and a focus on fluctuation theory is given in \cite{dereich2017}. 
Let $[n] = \{1,\ldots,n\}$ be a finite set and $(\R \times [n])_{\vartheta}$ be the Alexandrov one-point compactification of $\R \times [n]$ with some isolated state $\vartheta = (\infty, \varpi)$. Throughout we will always extend a function $f\in \cB(\R \times [n])$ to a function in $\cB((\R \times [n])_\vartheta)$ by setting $f(\vartheta) = 0$, which will make notation more convenient. A (killed) Markov additive process (MAP) $(\xi,J)$ with finite modulating space $[n]$ is defined as a Feller process with state space $\R \times [n]$ and cemetery state $\vartheta$, having a possibly finite lifetime $\zeta$ and underlying stochastic base $(\Omega,\cF,\F = (\cF_t)_{t\geq 0}, (\PP^{x,i})_{(x,i) \in (\R \times [n])_\vartheta})$ and which moreover has the characteristic property that given $s,t \geq 0,$ $(x,i) \in \R \times [n]$ and $f \in \cB_b((\R \times [n])_\vartheta)$ it holds that
$$\E^{x,i}\big[f(\xi_{t+s} - \xi_t, J_{t+s}) \one_{\{t < \zeta\}} \vert \cF_t\big] = \E^{0,J_t}[f(\xi_s,J_s)] \one_{\{t < \zeta\}}, \quad \PP^{x,i}\text{-a.s.}$$
In other words, conditionally on $\{J_t = i\}$ and no killing before time $t \geq 0$, the pair $(\xi_{t+s}- \xi_t, J_{t+s})_{s\geq 0}$ is independent of the past and has the same distribution as $(\xi_s, J_s)_{s\geq 0}$ under $\PP^{0,i}$, which is an equivalent definition for MAPs with finite modulating space often encountered in the literature such as \cite{dereich2017}. A straightforward consequence of this property is conditional spatial homogeneity of the process, i.e.\
$$\E^{x,i}[f(\xi, J)] = \E^{0,i}[f(\xi + x, J)]$$
holds for any bounded measurable $f$ on the Skorokhod space $\cD(\R \times [n])$ of càdlàg functions mapping from $\R_+ = [0,\infty)$ to $\R \times [n]$ equipped with its Borel $\sigma$-algebra (here and for the rest of the paper we implicitly assume that $(\xi,J)$ has exclusively càdlàg paths, which can be easily achieved by either constructing the process as the canonical coordinate process on the Skorokhod space or by a reduction of the probability space and the facts that, by definition, Feller  processes have càdlàg paths almost surely and $\cF$ is complete). 
Moreover, $(J_t)_{t \geq 0}$ is a continuous time Markov chain, whose transition function is independent of the initial distribution of $\xi$. Conditional independence of increments and spatial homogeneity of the ordinator $\xi$ already teases an intimate relation of MAPs and L\'evy processes. In fact, any MAP can be decomposed into an independent sequence of L\'evy processes, whose characteristic triplet depends on the current state of the modulating Markov chain $J$.

More precisely, we suppose that the measurable space $(\Omega, \cF)$ is rich enough to support a probability measure $\PP$ such that $\PP^{x,i} = \PP(\cdot \vert \xi_0 = x, J_0 = i)$, i.e.\ the probabilities underlying the Markov process $(\xi,J)$ are given as regular conditional probabilites of $\PP$. Then, Proposition 2 in \cite{dereich2017} (see also \cite[Proposition 2.5]{ivanovs2007} or \cite[Theorem 2.23]{cinlar1972}) gives the following characterization of a MAP, showing that in between jumps of $J$, $\xi$ behaves as a L\'evy process with characteristic triplet determined by the current state of $J$ and every jump of $J$ potentially triggers an additional jump of $\xi$.

\begin{proposition}\label{char map levy}
  A process $(\xi,J)$ is an unkilled MAP if and only if there exist sequences of
  \begin{itemize}
    \item (killed) L\'evy processes $(\xi^{n,i})_{n \in  \N_0},$ i.i.d.\ under $\PP$ for fixed $i \in [n]$,
    \item real random variables $(\Delta^n_{i,j})_{n\in \N}$, i.i.d.\ under $\PP$ for fixed and distinct $i,j \in [n]$,
  \end{itemize}
  independent of $J$ and of each other under $\PP$, such that if $\sigma_n$ is the $n$-th jump time of $J$, then under $\PP^{x,i}$, $\xi$ can be written almost surely as
  $$\xi_t = \left\{\begin{array}{ll} x + \xi_t^{0,i}, \quad &t \in [0, \sigma_1),\\ \xi_{\sigma_{n}-} + \Delta^n_{J_{\sigma_n-}, J_{\sigma_n}} + \xi_{t-\sigma_n}^{n,J_{\sigma_n}}, \quad &t \in [\sigma_n,\sigma_{n+1}) , t < \zeta,\\ \xi_t = \infty, \quad &t \geq \zeta, \end{array}\right.$$
  where the lifetime $\zeta$ is the first time one of the appearing L\'evy processes is killed:
  $$\zeta = \inf\big\{ t > 0: \exists n\in \N_0, \sigma_n \leq t \text { such that } \xi^{n, J_{\sigma_n}} \text{ is killed at time } t - \sigma_n\big\}.$$
\end{proposition}

In this paper, we will only deal with MAPs $(\xi,J)$ with \textit{infinite lifetime}, i.e.\ $\zeta = \infty$, $\PP^{x,i}$-a.s.\ for all $(x,i) \in \R \times [n]$. However, killing is relevant for fluctuation theory of MAPs as described below. For technical reasons we also exclude cases when $\xi$ has lattice support. Let us define $(\xi^{(i)})_{i \in [n]}$ as L\'evy processes with characteristic triplets $(a_i,b_i, \Pi_i)$ that have the same law as $(\xi^{0,i})_{i \in [n]}$ and $(\Delta_{i,j})_{i,j \in [n]}$ as random variables sharing the same law as the corresponding $(\Delta_{i,j}^{1})_{i,j \in [n]}$, with $\Delta_{i,i} \coloneq 0$ for all $i \in [n]$. Moreover, let $F_{i,j}$ be the law of $\Delta_{i,j}$. Then, $(\xi,J)$ can be uniquely characterized by the L\'evy--Khintchine exponents $\Psi_i(\theta) = \log \E[\exp(\mathrm{i}\theta \xi^{(i)}_1)], i \in [n]$, the transition rate matrix $\bm{Q} = (q_{i,j})_{i,j \in [n]}$ of $J$ and the Fourier transforms of $\Delta_{i,j}$ denoted by $G_{i,j}(\theta) = \E[\exp(\mathrm{i} \Delta_{i,j})], i,j \in [n]$. For convenience we assume $\Delta_{i,j} = 0$ whenever $q_{i,j} = 0$, which is without loss of generality because Proposition \ref{char map levy} shows that these transitional jumps never occur. If we now define the characteristic matrix exponent
$$\bm{\Psi}(\theta) \coloneq \mathrm{diag}\left(\Psi_1(\theta),\ldots, \Psi_n(\theta)\right) + \bm{Q} \odot \bm{G}(\theta),$$
as an analogue to the L\'evy--Khintchine exponent of a L\'evy process, then
$$\E^{0,i}\Big[\mathrm{e}^{\mathrm{i}\theta \xi_t}; \, J_t = j \Big] = \big(\mathrm{e}^{t \bm{\Psi}(\theta)}\big)_{i,j}, \quad i,j \in [n], \theta \in \R.$$
Here, $\bm{G}(\theta) = (G_{i,j}(\theta))_{i,j \in [n]}$ and $\odot$ denotes the Hadamard product, i.e.\ pointwise multiplication of matrices of the same dimension. Note that since $\Delta_{i,i} = 0$ we have $G_{i,i}(\theta) = 1$ for all $i \in [n]$ and hence $(\bm{Q} \odot \bm{G}(\theta))_{i,i} = -q_{i,i}$. Let us also define the family of potential measures $(U_{i,j})_{i,j \in [n]}$ given by
$$U_{i,j}(\diff{x}) = \E^{0,i}\Big[\int_0^\infty \one_{\{\xi_t\in \diff{x}, J_t = j\}}\diff{t}\Big] = \int_0^\infty \PP^{0,i}(\xi_t \in \diff{x}, J_t = j)\diff{t}, \quad x\in \R, i,j\in [n],$$
i.e., $U_{i,j}(A)$ measures the time $\xi$ spends in $A$ when started in $i$, while the modulator $J$ is in state $j$.
Another important concept in the theory of (general state space) Markov additive processes is the existence of a \textit{L\'evy system}, see \c{C}inlar \cite{cinlar1974}, which generalizes the notion of a L\'evy measure and becomes explicit for MAPs with finite modulating space thanks to the path decomposition given in Proposition \ref{char map levy}. We say that $(\bm{\Pi}, A)$, where $\bm{\Pi}$ is a kernel on $([n],\mathcal{B}(\R\times[n]))$ satisfying
$$\bm{\Pi}(i,\{(0,i)\}) = 0, \quad \int_{\R} \big(1 \wedge \lvert y \rvert^2\big)\, \bm{\Pi}(i, \diff{y} \times \{i\}) < \infty, \quad i \in [n],$$
and $A$ is an increasing continuous additive functional of $(\xi,J)$ such that for any $f \in \mathcal{B}_+([n] \times \R \times[n]))$ and $(x,i) \in \R \times [n]$,
\begin{equation}\label{eq: levy system}
\E^{0,i}\Big[\sum_{s \leq t} f(J_{s-},\Delta \xi_s, J_s) \one_{\{ \Delta \xi_s \neq 0 \text{ or } J_{s-}\neq J_s\}} \Big] = \E^{0,i}\Big[\int_0^t A_s \int_{\R \times [n]} \bm{\Pi}(J_s,\diff{x}, \diff{y})\, f(J_s,x,y)\Big],
\end{equation}
is a L\'evy system for $(\xi,J)$. Using Proposition \ref{char map levy} and results on expectations of functionals of Poisson random measures, see e.g.\ Theorem 2.7 in \cite{kyprianou2014}, one can demonstrate that $A_t = t \wedge \zeta$ and
$$\bm{\Pi}(i,\diff{y} \times \{j\}) = \one_{\{i=j\}} \Pi_i + \one_{\{i\neq j\}} q_{i,j} F_{i,j}, \quad i,j \in [n],$$
and thus for any $i \in [n]$,
\begin{equation}\label{eq: compensation map}
\begin{split}
\E^{0,i}\Big[\sum_{s \leq t} f(J_{s-},\Delta \xi_s,J_s) \one_{\{ \Delta \xi_s \neq 0 \text{ or } J_{s-}\neq J_s\}} \Big] &= \sum_{k=1}^n\Big(\E^{0,i}\Big[\int_0^t \int_{\R\setminus\{0\}} f(k,x,k) \one_{\{J_{s} = k\}}\, \Pi_k(\diff{x}) \diff{s}\Big]\\
&\qquad + \sum_{j \neq k} q_{k,j} \E^{0,i}\Big[\int_0^t \int_{\R} f(k,x,j) \one_{\{J_{s} =k\}}\, F_{k,j}(\diff{x}) \diff{s}\Big]\Big)\\
&= \sum_{k=1}^n \int_0^t \PP^{0,i}(J_s = k) \diff{s} \,\Big( \int_{\R\setminus\{0\}} f(k,x,k)\, \Pi_k(\diff{x})\\
&\qquad +\sum_{j \neq k}q_{k,j} \int_{\R} f(k,x,j)\, F_{k,j}(\diff{x})\Big).
\end{split}
\end{equation}
Since $A$ is simply the uniform motion, we will also refer to just $\bm{\Pi}$ as the L\'evy system for the remainder of this article. As remarked in \cite{kyprianou2020}, this can be generalized to the following identity for any predictable process $(Z_t)_{t \geq 0}$ and $g\in \cB_+([n] \times \R \times \R \times [n])$:
\begin{equation}\label{eq: compensation map2}
\begin{split}
&\E^{0,i}\Big[\sum_{s \leq t} Z_s g(J_{s-},\xi_{s-},\xi_s,J_s) \one_{\{ \Delta \xi_s \neq 0 \text{ or } J_{s-}\neq J_s\}} \Big] \\
&\quad = \sum_{k=1}^n\Big(\E^{0,i}\Big[\int_0^t \diff{s}\, Z_s \one_{\{J_s = k\}} \int_{\R\setminus\{0\}}  \Pi_k(\diff{x})\, g(k,\xi_s, \xi_s + x,k)\Big]\\
&\qquad \quad + \sum_{j \neq k} q_{k,j} \E^{0,i}\Big[\int_0^t \diff{s}\, Z_s \one_{\{J_s = k\}} \int_{\R}   F_{k,j}(\diff{x})\, g(k,\xi_s, \xi_s + x,j)\Big]\Big).
\end{split}
\end{equation}

Let us now dive into fluctuation theory of MAPs, which in the form suited to our needs was developed in \cite{dereich2017}. An essential tool for our upcoming analysis of the overshoots is the \textit{ascending ladder height MAP} $(H_t^+,J_t^+)_{t \geq 0}$, which is defined as follows (see the appendix of \cite{dereich2017} for more details). Let $(\mathsf{L}^{(i)}_t)_{t \geq 0}$ be a version of local time at the point $(0,i)$ for the strong Markov process $(\overbar{\xi}_t -\xi_t, J_t) _{t \geq 0}$, where $\overbar{\xi}_t \coloneq \sup_{s \leq t} \xi_s$. Define then $\mathsf{L}_t \coloneq \sum_{i=1}^n \mathsf{L}^{(i)}_t$, which is a continuous additive functional of $(\overbar{\xi}_t -\xi_t, J_t) _{t \geq 0}$, increasing almost surely on the set of times when $\xi$ attains a new maximum.

With this at hand we define the ladder height process $(H^+,J^+)$ by the time change
$$\big(H^+_t,J^+_t\big) = \left\{ \begin{array}{ll} \big(\xi_{\mathsf{L}^{-1}_t}, J_{\mathsf{L}^{-1}_t}\big), \quad & 0\leq t < \mathsf{L}_\infty,\\ \vartheta = (\infty,\varpi),\quad &t \geq \mathsf{L}_\infty, \end{array}\right.$$
where $\mathsf{L}^{-1}_t  \coloneq \inf\{s \geq 0: \mathsf{L}_s > t\}$ is the right-continuous inverse of $\mathsf{L}$. It can be shown that $(H^+,J^+)$ is a Markov additive subordinator with lifetime $\mathsf{L}_\infty$, i.e.\ a Markov additive process such that the ordinator $H^+$ has increasing paths almost surely before killing. Moreover, $(\mathsf{L}^{-1}_t)_{0 \leq t < \infty}$ almost surely equals the ordered set of times, when $\xi$ reaches a maximum and hence the closure of the range of $H^+$ up to its lifetime is identical to that of the supremum process $\overbar{\xi}$ almost surely. Denote by $H^{+,(i)}$ the L\'evy subordinators appearing in the decomposition of $(H^+,J^+)$ in the spirit of Proposition \ref{char map levy}. The respective drifts and L\'evy measures are denoted by $d_i^+$ and $\Pi^+_i$, the intensity matrix of $J^+$ by $\mathbf{Q}^+ = (q^+_{i,j})_{i,j \in [n]}$ and the killing rates of $H^{+,(i)}$ by $\dagger^+_i$, i.e., when $\dagger^+_i > 0$, the lifetime $\zeta^+_i$ of $H^{+,(i)}$ is exponentially distributed with mean $1\slash \dagger^+_i$ and otherwise, for $\dagger^+_i = 0$, $\zeta^+_i = \infty$ almost surely. Note that the MAP subordinator $(H^+,J^+)$ is then uniquely characterized by its Laplace exponent, given as follows:
\begin{equation} \label{char exp map}
\bm{\Phi}^+(\theta) \coloneq \mathrm{diag}\left(\Phi^{+}_1(\theta),\ldots, \Phi^{+}_n(\theta)\right) - \bm{Q}^+ \odot \bm{G}^+(\theta), \quad \theta \geq 0,
\end{equation}
where $\Phi^+_i$ is the Laplace exponent of $H^{+,(i)}$ and $\bm{G}^+(\theta) = (G_{i,j}^+(\theta))_{i,j \in [n]} = (\E[\exp(-\theta \Delta_{i,j}^+)])_{i,j \in [n]}$. It then holds that
$$\E^{0,i}\big[\exp(-\theta H^+_t);\, J^+_t = j\big] = \big(\mathrm{e}^{- \bm{\Phi}^+(\theta)t}\big)_{i,j}, \quad t \geq 0, \theta \geq 0, i,j \in [n].$$
Let us also denote the family of potential measures of $(H^+,J^+)$ by $(U^+_{i,j})_{i,j \in [n]}.$

In analogy to the case for L\'evy processes we also need the ascending ladder height process of the \textit{dual} of the MAP $(\xi,J)$, i.e.\ a MAP which has the same law as the time reversed MAP $(\xi,J)$. As remarked in \cite{dereich2017} the construction of the dual MAP is slightly more elaborate compared to the L\'evy case, where the dual process is simply the negative of the original L\'evy process, because we have to take care of time reversion of the ordinator $J$. Suppose that $J$ is irreducible -- and hence ergodic thanks to its finite state space -- and denote its invariant distribution by $\bm{\pi} = (\pi(i))_{i \in [n]}$. Moreover, let
$$\hat{q}_{i,j} = \frac{\pi(j)}{\pi(i)} q_{j,i}, \quad i,j \in [n],$$
which are the intensities of the time reversed modulating Markov chain $J$ and let $\bm{\hat{Q}} = (\hat{q}_{i,j})_{i,j\in [n]}$. Now let $(\hat{\PP}^{x,i})_{(x,i)\in \R \times [n]}$ be a family of probability measures such that $(\xi,J)$ has characteristic matrix exponent given by
$$\bm{\hat{\Psi}}(\theta) = \big(\hat{\E}^{0,i}\big[\exp(\mathrm{i}\theta\xi_1);\, J_1 = j\big]\big)_{i,j\in [n]} = \mathrm{diag}(\psi_1(-\theta),\ldots, \psi_n(-\theta)) + \bm{\hat{Q}}\odot \bm{G}(-\theta)^\top, \quad \theta \in \R.$$
Then indeed, under $\PP^{0,\bm{\pi}} \coloneq \sum_{i=1}^n \pi(i)\PP^{0,i}$, it holds that the time reversed process $(\xi_{(t-s)-}-\xi_t, J_{(t-s)-})_{0\leq s\leq t}$ is equal in law to $(\xi_s,J_s)_{s\leq t}$ under $\hat{\PP}^{0,\bm{\pi}}$, see Lemma 21 in \cite{dereich2017}. Let $\bm{\Delta_\pi} \coloneq \mathrm{diag}(\bm{\pi})$ and denote the matrix Laplace exponent of the ascending ladder height process of the dual process of $(\xi,J)$ by $\bm{\hat{\Phi}}{}^+$ and also the objects belonging to its L\'evy system in the obvious way.\footnote{A word of caution at this point: $\bm{\hat{\Phi}}{}^+$ is \textit{not} the matrix exponent of the dual of the ascending ladder height MAP $(H^+,J^+)$. To not confuse the reader we will therefore withhold the temptation to denote the ascending ladder height process of the dual of $(\xi,J)$ by $(\hat{H}{}^+,\hat{J}{}^+)$.} The key result for fluctuation theory of MAPs is the (spatial) Wiener--Hopf factorization given in Theorem 26 of \cite{dereich2017}, which states that up to pre-multiplication by a positive diagonal matrix corresponding to the scaling of local time at the supremum,
\begin{equation} \label{eq: wiener-hopf}
-\bm{\Psi}(\theta) = \bm{\Delta_\pi}^{-1} \bm{\hat{\Phi}}{}^+(\mathrm{i\theta})^\top\bm{\Delta_\pi}\bm{\Phi}^+(-\mathrm{i}\theta) = \bm{\Delta_\pi}^{-1} \bm{\hat{\Psi}}{}^+(-\theta)^\top\bm{\Delta_\pi}\bm{\Psi}^+(\theta), \quad \theta \in \R,
\end{equation}
and thus gives a factorization of the characteristic matrix exponent $\bm{\Psi}$ of $(\xi,J)$ in terms of the characteristic exponents $\bm{\Psi}^+$ and $\bm{\hat{\Psi}}{}^+$ of the ascending ladder height processes of $(\xi,J)$ and its dual, respectively. This identity is the key for understanding the interplay between the parent MAP $\xi$ and the ladder height processes, which we will further explore in Section \ref{sec: vigon}.

\section{Stability analysis of overshoots of MAPs} \label{sec: stability}
In this section, we assume that the lifetime $\zeta$ of $(\xi,J)$ is equal to $\infty$ on all of $\Omega$. For $t \geq 0$ define the ordinator`s $\xi$ first hitting time $T_t$ of the set $(t,\infty)$ by
$$T_t \coloneq \inf\{s \geq 0: \xi_s > t\}.$$
Note that by right-continuous paths of the process and right-continuity of the filtration $(\cF_t)_{t \geq 0}$ underlying $(\xi,J)$ this is a stopping time for the MAP. Set also
$$\overbar{\xi}_\infty \coloneq \sup_{0 \leq t < \infty} \xi_t.$$
We now define the process $(\cO_t,\cJ_t)_{t \geq 0}$ by
$$(\cO_t,\cJ_t) = \left\{ \begin{array}{ll} \big(\xi_{T_t} - t, J_{T_t}), \quad &\text{if } t < \overbar{\xi}_\infty,\\ \vartheta, \quad &\text{if } t \geq \overbar{\xi}_\infty,\end{array}\right. \quad t\geq 0,$$
i.e.\ if the level $t$ is smaller than the supremum of the process over its entire lifetime, then $\cO_t$ corresponds to the overshoot of $\xi$ over $t$ and $\cJ_t$ is equal to the state of the modulator at first passage of $t$, whereas for $t \geq \overbar{\xi}_\infty$ the process is sent to the cemetery state $\vartheta$.
An essential observation for our analysis is that $(\cO_t,\cJ_t)_{t \geq 0}$ is indistinguishable with respect to the family of probability measures $(\PP^{x,i})_{(x,i)\in (\R_+\times [n])_\vartheta}$ from the process $(\cO^+_t,\cJ^+_t)_{t \geq 0}$ corresponding to the ascending ladder MAP $(H^+,J^+)$, and hence is given by
$$(\cO^+_t, \cJ^+_t) = \left\{ \begin{array}{ll} \big(H^+_{T^+_t} - t, J^+_{T^+_t}), \quad &\text{if } t < \overbar{H}{}^+_\infty,\\ \vartheta, \quad &\text{if } t \geq \overbar{H}{}^+_\infty,\end{array}\right.\quad t \geq 0,$$
where $(T^+_t)_{t \geq 0}$ is the first passage process of $H^+$, which by increasing paths of $H^+$ is equal to its right-continuous inverse. Indistinguishability of the processes follows immediately from the fact that on $[0, \mathsf{L}_\infty)$, the range of the increasing process $(\mathsf L^{-1}_t)_{t \geq 0}$ almost surely equals the set of times when $\xi$ reaches a maximum. Using this relationship, \eqref{eq: compensation map2} and arguing as in the classical proof for the law of the undershoot/overshoot distribution for L\'evy processes (see \cite[Theorem 5.6]{kyprianou2014}), we obtain the following formula for the marginal distribution of the overshoot process
\begin{equation} \label{eq: overshoot law}
\begin{split}
\PP^{x,i}(\cO_t \in \diff{y}  , \cJ_t = j) &= \PP^{0,i}(\cO^+_{t-x} \in \diff{y} , \cJ^+_t = j)\\
&= \int_{[0,t-x)} \, \Pi^+_j(u+\diff{y}) \, U^+_{i,j}(t - x - \diff{u})\\
&\hspace{6pt} + \sum_{k \neq j} q^+_{k,j} \int_{[0,t-x)}  \, F^+_{k,j}(u+\diff{y}) \, U^+_{i,k}(t - x - \diff{u}), \quad i,j \in [n], x \in [0, t), y > 0,
\end{split}
\end{equation}
and
\begin{equation} \label{eq: sawtooth}
\E^{x,i}[f(\cO_t,\cJ_t)] = f(x-t,i), \quad x \in [t,\infty), i \in [n],
\end{equation}
provided that $\PP^{0,i}(T^+_0 = 0) = 1$. Assumption \ref{ass: regularity} introduced below will ensure this property. Equation \eqref{eq: sawtooth} describes the characteristic behavior of the overshoot process away from $0$ in the sense that if $\cO_t(\omega) = y > 0$ we have $\cO_s(\omega) = y - (s-t)$ for $s \in [t, t+y]$, i.e.\ the origin is approached at unit speed. This characteristic path structure of the overshoot process is visualized in Figure \ref{fig: overshoot plot} for the case of a compound Poisson subordinator $\sigma$ with positive drift, and is the reason why for such L\'evy subordinators the overshoot process is also known as sawtooth process, cf.\ Chapter II.3 in \cite{blumenthal1992}. We will therefore also refer to it as the sawtooth structure for MAP overshoots.

\begin{figure}[ht]
	\centering
	\fontsize{10}{12}
	\begin{tikzpicture}
	\node[circle, scale = 0.3, color= myteal, draw] (x1) at (1,0.5) {} ;
	\node[circle, scale = 0.3, color= myteal, fill = myteal, draw] (x2) at (1,1.25) {} ;
	\node[circle, scale = 0.3, color= myteal, draw] (x3) at (2.6,2.05) {} ;
	\node[circle, scale = 0.3, color= myteal, fill = myteal, draw] (x4) at (2.6,3.2) {} ;
	\node[circle, scale = 0.3, color= myteal, draw] (x5) at (3.2,3.5) {} ;
	\node[circle, scale = 0.3, color= myteal, fill = myteal, draw] (x6) at (3.2,4) {} ;
	\node[circle, scale = 0.3, color= myteal, draw] (x7) at (5.2,5) {} ;
	\node[circle, scale = 0.3, color= myteal, fill = myteal, draw] (x8) at (5.2,7) {} ;
	\node[circle, scale = 0.3,color = Purple, draw] (y1) at (0,0.5) {} ;
	\node[circle, scale = 0.3, fill = Purple, color = Purple,  draw] (y2) at (0.75,0.5) {} ;
	\coordinate (y3) at (0,1.25) {};
	\node[circle, scale = 0.3, color = Purple, draw] (y4) at (0,2.05) {} ;
	\node[circle, scale = 0.3, fill = Purple, color= Purple, draw] (y5) at (1.05,2.05) {} ;
	\coordinate (y6) at (0,3.2) {} ;
	\node[circle, scale = 0.3, color = Purple, draw] (y7) at (0,3.5) {} ;
	\node[circle, scale = 0.3, fill = Purple, color = Purple, draw] (y8) at (0.5,3.5) {} ;
	\coordinate (y9) at (0,4) {} ;
	\node[circle, scale = 0.3, color = Purple, draw] (y10) at (0,5) {} ;
	\node[circle, scale = 0.3, fill = Purple, color = Purple, draw] (y11) at (2,5) {} ;
	\coordinate (y12) at (0,7) {} ;
	\draw[->] (0,0) -- (7.2,0) node at (7.0,-0.3) {\textcolor{myteal}{$s$}, \textcolor{Purple}{$\mathcal{O}^\sigma_t$}};
	\draw[->] (0,0) -- (0,8) node at (-0.5,7.8) {\textcolor{myteal}{$\sigma_s$}, \textcolor{Purple}{$t$}};
	\draw[myteal, thick] (0,0) -- (x1);
	\draw[myteal, dotted] (x1) -- (1,0) node at (1.0,-0.3) {$T^\sigma_{t_1}$};
	\draw[dashed, color = Purple] (x1) -- (x2) node at (1.9,0.875) {\textcolor{Purple}{\fontsize{10}{12} $\mathcal{O}^\sigma_{t_1}$} $=$ \textcolor{myteal}{\fontsize{10}{12} $\Delta \sigma_{T^\sigma_{t_1}}$}};
	\draw[myteal, thick] (x2) -- (x3);
	\draw[myteal, dotted] (x3) -- (2.6,0) node at (2.6,-0.3) {$T^\sigma_{t_2}$};
	\draw[dashed, color = Purple] (x3) -- (x4) node at (3.5,2.625) {\textcolor{Purple}{\fontsize{10}{12} $\mathcal{O}^\sigma_{t_2}$} $=$ \textcolor{myteal}{\fontsize{10}{12} $\Delta \sigma_{T^\sigma_{t_2}}$}};
	\draw[myteal, thick] (x4) -- (x5);
	\draw[myteal, dotted] (x5) -- (3.2,0) node at (3.2,-0.3) {$T^\sigma_{t_3}$};
	\draw[dashed, color = Purple] (x5) -- (x6) node at (4.1,3.75) {\textcolor{Purple}{\fontsize{10}{12} $\mathcal{O}^\sigma_{t_3}$} $=$ \textcolor{myteal}{\fontsize{10}{12} $\Delta \sigma_{T^\sigma_{t_3}}$}};
	\draw[myteal, thick] (x6) -- (x7);
	\draw node at (4,5) {\textcolor{myteal}{$(\sigma_s)_{s \geq 0}$}};
	\draw[myteal, dotted] (x7) -- (5.2,0) node at (5.2,-0.3) {$T^\sigma_{t_4}$};
	\draw[dashed, color = Purple] (x7) -- (x8) node at (6.1,6) {\textcolor{Purple}{\fontsize{10}{12} $\mathcal{O}^\sigma_{t_4}$}$=$ \textcolor{myteal}{\fontsize{10}{12} $\Delta \sigma_{T^\sigma_{t_4}}$}};
	\draw[myteal, thick] (x8) -- (7.2,8);
	\draw[Purple,thick] (0,0) -- (y1) node at (-0.85,0.43) {\textcolor{myteal}{$\sigma_{T_{t_1}^\sigma-}$} \textcolor{black}{$=$} \textcolor{Purple}{$t_1$}};
	\draw[dashed, color = Purple] (y1) -- (y2);
	\draw[Purple,thick] (y2) -- (y3);
	\draw[Purple,thick] (y3) -- (y4) node at (-0.85,1.98) {\textcolor{myteal}{$\sigma_{T_{t_2}^\sigma-}$} \textcolor{black}{$=\,$}\textcolor{Purple}{$t_2$}};
	\draw[dashed, color = Purple] (y4) -- (y5);
	\draw[Purple,thick] (y5) -- (y6);
	\draw[Purple,thick] (y6) -- (y7) node at (-0.85,3.43) {\textcolor{myteal}{$\sigma_{T_{t_3}^\sigma-}$} \textcolor{black}{$=\,$}\textcolor{Purple}{$t_3$}};
	\draw[dashed, color = Purple] (y7) -- (y8);
	\draw[Purple,thick] (y8) -- (y9);
	\draw[Purple,thick] (y9) -- (y10) node at (-0.85,4.93) {\textcolor{myteal}{$\sigma_{T_{t_4}^\sigma-}$} \textcolor{black}{$=\,$}\textcolor{Purple}{$t_4$}};
	\draw[dashed, color = Purple] (y10) -- (y11);
	\draw[Purple,thick] (y11) -- (y12);
	\draw[Purple,thick] (y12) -- (0,8);
	\draw[Purple] node at (2.6,6) {\textcolor{Purple}{$(\mathcal{O}^\sigma_t = \sigma_{T^\sigma_t} - t)_{t \geq 0}$}};
	\end{tikzpicture}
	\caption{Path of a compound Poisson subordinator with drift, $\sigma$, and associated overshoot process $\cO^\sigma$} \label{fig: overshoot plot}
\end{figure}
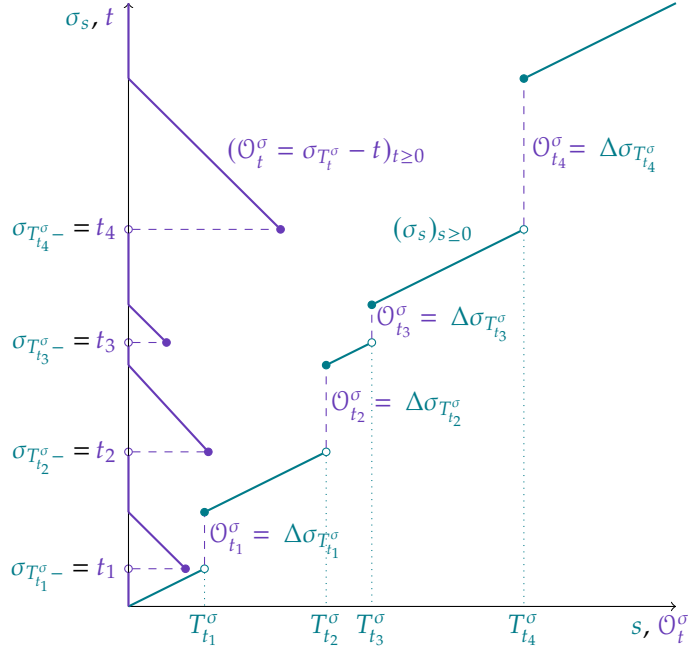

Let $\cG_t \coloneq \cF_{T_t}$ for $t \geq 0$ and define the filtration $\G \coloneq (\cG_t)_{t \geq 0}$. The following technical results hold.

\begin{lemma} \label{lemma: filtration}
$\G$ is right-continuous and complete.
\end{lemma}
\begin{proof}
First note that $\cF_{T_t} = \cF_{T_t+}$, with
$$\cF_{T_t+} \coloneq \big\{\Lambda \in \cF: \Lambda \cap \{T_t < s\} \in \cF_s \mbox{ for all } s \geq 0\big\}$$
since the latter can be shown to be equal to
$$\big\{\Lambda \in \cF: \Lambda \cap \{T_t \leq s\} \in \cF_{s+} \mbox{ for all } s \geq0\big\}$$
which in turn equals $\cF_{T_t}$ thanks to right-continuity of $\F$. Letting $\Lambda \in \cG_{t+} = \bigcap_{n \in \N} \cF_{T_{t + 1 \slash n}}$ we obtain by right-continuity of $t \mapsto T_t$ that for any $s \geq 0$
$$\Lambda \cap \{T_t < s\} = \bigcup_{n \in \N} \Lambda \cap \big\{T_{t+ \frac{1}{n}} < s\big\} \in \cG_s,$$
since any set in the right-hand union belongs to $\cG_s$ thanks to $\cF_{T_{t + 1\slash n}} = \cF_{T_{(t + 1 \slash n)}+}.$ It follows that $\Lambda \in \cF_{T_t+} = \cF_{T_t} = \cG_t,$ which proves right-continuity of $\G$. Completeness of $\mathbb{G}$ follows from the fact the $\PP^{x,i}$-augmentation of $\cF_{T_t}$ is equal to $\cF_{T_t}$ itself, since $\F$ is $\PP^{x,i}$-augmented already, see also p.36\ of \cite{blumenthal1968}.
\end{proof}

\begin{corollary}
For any $0< s \leq \infty$ the running supremum $\overbar{\xi}_s$ is a stopping time with respect to $\G$. In particular, the lifetime $\overbar{\xi}_\infty$ of $(\cO_t,\cJ_t)_{t \geq 0}$ is a $\G$-stopping time.
\end{corollary}
\begin{proof}
Let $s \in (0,\infty].$ For any $t \geq 0$
$$\Big\{\sup_{u \in [0,s)} \xi_u > t\Big\} = \{T_t < s\} \in \cF_{T_t} = \cG_t,$$
which implies that $\sup_{u \in [0,s)} \xi_u$ is a $\mathbb{G}$ stopping time. Moreover, for $s \in (0,\infty)$, quasi-left-continuity of $(\xi,J)$ implies that $\xi_{s-} = \xi_s$, $\PP^\ast$-a.s., and therefore $\overbar{\xi}_s = \sup_{u \in [0,s)} \xi_u$, $\PP^\ast$-a.s.. Since $\mathbb{G}$ is complete by Lemma \ref{lemma: filtration} and $\sup_{u \in [0,s)} \xi_u$ is a $\mathbb{G}$-stopping time, we conclude that $\overbar{\xi}_s$ is a $\mathbb{G}$-stopping time as well.
\end{proof}

We now show that under a technical assumption, the overshoot process given by the quintuple $(\Omega,\cF,\G,(\cO_t,\cJ_t)_{t \geq 0}, (\PP^{x,i})_{(x,i)\in (\R_+ \times [n])_\vartheta})$ determines a Feller process and therefore also a Borel right process. The technical assumption under which we will be working throughout the rest of the paper without further mention, is the following.
\begin{enumerate}[leftmargin=*,label=($\mathscr{A}0$),ref=($\mathscr{A}0$)]
\item The MAP $(\xi,J)$ is upward regular, i.e.\ for any $i \in [n]$ it holds that $\PP^{0,i}(T_0 = 0) = 1$. \label{ass: regularity}
\end{enumerate}
By definition, $(\xi,J)$ is upward regular if, independently of the starting point of the modulator $J$, $\xi$ started from $0$ immediately hits the upper half line. By the path decomposition given in Proposition \ref{char map levy}, this is the case if and only if the underlying L\'evy processes $\xi^{(i)}$ are regular upward for any $i \in [n]$. Upward regularity for L\'evy processes is completely understood, see the full characterization given in Theorem 6.5 of \cite{kyprianou2014}, and hence upward regularity of the MAP can be characterized by properties of its underlying L\'evy processes. Moreover, by the general theory on local times of Markov processes, see e.g.\ Chapter 4 in Bertoin \cite{bertoin1996} or the landmark paper Blumenthal and Getoor \cite{blumenthal1964}, it follows that upward regularity implies that for each $i \in [n]$, the local time $\mathsf{L}^{(i)}$ of $(\overbar{\xi}-\xi,J)$ at $(0,i)$ is almost surely continuous and hence $\mathsf{L} = \sum_{i=1}^n \mathsf{L}^{(i)}$ is almost surely continuous as well. Hence, the right-continuous inverse $(\mathsf{L}^{-1}_t)_{t \geq 0}$, corresponding to the set of times when a new maximum of $\xi$ is reached, is strictly increasing on $[0, \mathsf{L}_\infty)$ almost surely and it follows that $H^+$ is strictly increasing up to its lifetime. This property is essential for $(\cO,\cJ)$ being a Feller process, as the proof of the following proposition shows.

\begin{proposition} \label{overshoot feller}
$(\cO,\cJ)$ is a càdlàg Feller process with lifetime $\overbar{\xi}_\infty$.
\end{proposition}
\begin{proof}
Càdlàg paths of the process are a direct consequence of càdlàg paths of $(\xi,J)$ and the fact that $t \mapsto T_t$ is right-continuous on $[0,\infty)$ and increasing on $[0,\overbar{\xi}_\infty)$. Let now $f \in \cB_b((\R_+ \times [n])_\vartheta)$ and $(x,i) \in (\R_+ \times [n])_\vartheta.$ Recalling that $\overbar{\xi}_\infty$ is a $\G$-stopping time and using $T_{t+s} = T_{t} + T_{t+s} \circ \theta_{T_t},$ on $\{T_t < \infty\}$, where $(\theta_t)_{t \geq 0}$ are the transition opertors of $(\xi,J)$, it follows that $\PP^{x,i}$-a.s.\
\begin{align*}
  \E^{x,i}[f(\cO_{t+s}, \cJ_{t+s}) \vert \cG_t] &= \E^{x,i}\left[f\big(\xi_{T_{t+s}} -(t+s), J_{T_{t+s}}\big)\circ \theta_{T_t} \vert \cF_{T_t}\right] \one_{\{t < \overbar{\xi}_\infty\}} + f(\vartheta)\one_{\{t \geq \overbar{\xi}_\infty\}} \\
  &= \E^{\xi_{T_t}, J_{T_t}}\left[f\big(\xi_{T_{t+s}} - (t+s),J_{T_{t+s}}\big)\right]\one_{\{t < \overbar{\xi}_\infty\}} + f(\vartheta)\one_{\{t \geq \overbar{\xi}_\infty\}}\\
  &= \E^{\xi_{T_t} - t, J_{T_t}}\left[f\big(\xi_{T_s} - s,J_{T_s}\big)\right]\one_{\{t < \overbar{\xi}_\infty\}} + f(\vartheta)\one_{\{t \geq \overbar{\xi}_\infty\}}\\
  &= \E^{\cO_t, \cJ_t}\left[f\big(\cO_s,\cJ_s\big)\right].
\end{align*}
Here, we used the strong Markov property of $(\xi,J)$ for the second and spatial homogeneity of $\xi$ for the third equality. This proves the Markov property of $(\cO,\cJ)$. Moreover, for $x > 0$ and $i \in [n]$ we have $\PP^{x,i}(T_0 = 0) = 1$ and by upward regularity of $\xi$ we also have $\PP^{0,i}(T_0 = 0) = 1$. Thus, $\PP^{x,i}(\cO_0,\cJ_0) = (x,i)$ for any $(x,i) \in (\R_+ \times [n])_\vartheta$, i.e.\ the process is a normal Markov process and its lifetime is given by $\overbar{\xi}_\infty$ by construction. Let $(\mathcal{P}_t)_{t \geq 0}$ be its sub-Markov transition semigroup, i.e.\
$$\mathcal{P}_t f(x,i) = \E^{x,i}[f(\cO_t,\cJ_t);\ t < \overbar{\xi}_\infty], \quad (x,i) \in (\R_+ \times [n])_\vartheta, f \in \cB_b((\R_+ \times [n])_{\vartheta}).$$
Let us check the Feller property. Let $f \in \cC_0(\R_+ \times [n])$. Since $[n]$ is finite and recalling our convention that $f(\vartheta) = 0$, it suffices to show for fixed $i \in [n]$ that
$x \mapsto \mathcal{P}_t f(x,i) = \E^{x,i}[f(\cO_t,\cJ_t)]$ is continuous to prove that $(x,i) \mapsto \E^{x,i}[f(\cO_t,\cJ_t)]$ is continuous. If $x > t$ this is obvious. For $x \leq t$ let first $y \uparrow x$. By right-continuity of $t \mapsto (\cO_t,\cJ_t)$, continuity and boundedness of $f$, dominated convergence and conditional spatial homogeneity of $(\xi,J)$, it follows that
$$\lim_{y \uparrow x} \E^{y,i}[f(\cO_t,\cJ_t)] = \lim_{y \uparrow x} \E^{0,i}[f(\cO_{t-y},\cJ_{t-y})] = \E^{0,i}[f(\cO_{t-x},\cJ_{t-x})] = \E^{x,i}[f(\cO_{t},\cJ_{t})],$$
showing left-continuity of $x \mapsto \E^{x,i}[f(\cO_{t},\cJ_{t})].$ To show right-continuity, note that for $y \downarrow x$ it holds that $T^+_{t-y}$ increases to $\inf\{s \geq 0: H^+_s \geq t - x\}$ on $\{T^+_{t-x} < \infty\}$ and since $H^+$ is strictly increasing up to its lifetime by upward regularity of $\xi$, it follows that the latter hitting time is almost surely equal to $T^+_{t-x}.$ Since $(H^+,J^+)$ as a Feller process is quasi-left-continuous, it therefore follows that on $\{T^+_{t-x} < \infty\}$,
$$\lim_{y \downarrow x} \Big(H^+_{T^+_{t-y}}, J^+_{T^+_{t-y}}\Big) = \Big(H^+_{T^+_{t-x}}, J^+_{T^+_{t-x}}\Big), \quad \PP^{0,i}\text{-a.s.}$$
By indistinguishability of $(\cO^+,\cJ^+)$ and $(\cO,\cJ)$ we therefore obtain
$$\lim_{y \downarrow x} \E^{y,i}[f(\cO_t,\cJ_t)] = \lim_{y \downarrow x} \E^{0,i}[f(\cO^+_{t-y},\cJ^+_{t-y})] = \E^{0,i}[f(\cO^+_{t-x},\cJ^+_{t-x})] = \E^{x,i}[f(\cO_{t},\cJ_{t})],$$
proving also right-continuity of $(x,i) \mapsto \mathcal{P}_tf(x,i)$. Since moreover $[n]$ is compact and for fixed $i \in [n]$,
$$\lim_{x \to \infty} \mathcal{P}_tf(x,i) = \lim_{x \to \infty} f(x-t,i) = 0$$
thanks to $f \in \cC_0(\R_+ \times [n])$, we conclude that $\mathcal{P}_t\cC_0(\R_+ \times [n]) \subset \cC_0(\R_+ \times [n])$. Finally, for fixed $(x,i) \in \R_+ \times [n]$ (again applying to upward regularity in case $x = 0$) it follows from $T_t \to 0$ a.s.\ as $t \downarrow 0$ and dominated convergence, that $\mathcal{P}_tf(x,i) \to \mathcal{P}_0f(x,i) = f(x,i)$. This is enough to show that  $(\mathcal{P}_t)_{t \geq 0}$ is a Feller semigroup, as discussed in Appendix \ref{sec: markov stability}. Noting that $\G$ is right continuous and complete by Lemma \ref{lemma: filtration}, the proof is complete.
\end{proof}

Having established the Markovian nature of the overshoot process, we now proceed by investigating its stability properties and long-time behavior. We must therefore restrict to the case, when the overshoot process is almost surely unkilled, which is the case if and only if $\sup_{0 \leq s < \infty} \xi_s = \infty$, $\PP^{0,i}$-a.s.\ for all $i \in [n]$. As for L\'evy processes, there is a dichotomy concerning the long-time behavior of the ordinator $\xi$, namely that exactly one of the following cases can occur:
\begin{enumerate}[label=(\alph*),ref=(\alph*)]
  \item for any $(x,i) \in \R \times [n]$, $\limsup_{t \to \infty} \xi_t = \infty$, $\PP^{x,i}$-almost surely, and in this case either $\liminf_{t \to \infty} \xi_t = - \infty$ or $\lim_{t \to \infty} \xi_t = \infty$, $\PP^{x,i}$-a.s.;
  \item for any $(x,i) \in \R \times [n]$, $\lim_{t \to \infty} \xi_t = -\infty$, $\PP^{x,i}$-almost surely.
\end{enumerate}
When $J$ is irreducible and the MAP's ordinator possesses an exponential moment, which of these cases occurs for a given MAP is determined by a Perron--Frobenius type eigenvalue of the MAP's Laplace exponent, see Asmussen \cite[Proposition XII.2.10]{asmussen2003}. We will therefore henceforth work under the following additional assumption, which guarantees that $(\cO,\cJ)$ is an unkilled Borel right Markov process and therefore gives us access to the theory of stability for Markov processes teased in Appendix \ref{sec: markov stability}.
\begin{enumerate}[leftmargin=*,label=($\mathscr{A}1$),ref=($\mathscr{A}1$)]
\item For any $(x,i) \in (\R \times [n])$ it $\PP^{x,i}$-almost surely holds $\limsup_{t \to \infty} \xi_t = \infty.$ \label{ass: long time}
\end{enumerate}
Let us give the following definition.
\begin{definition}
Let $\bm{A} = (a_{i,j})_{i,j = 1,\ldots, n} \in \R^{n \times n}$ be a matrix with $a_{i,j} \geq 0$ for any $i \neq j$. We say that $\bm{A}$ is \textit{irreducible}, if for any $i \neq j$ there exists $(a_{i_k,i_{k+1}})_{k=0,\ldots,m-1}$ for some $m \in  \N$ with $i_0 = i, i_m = j$ such that $\prod_{k=0}^{m-1} a_{i_k,i_{k+1}} > 0.$ A matrix $\tilde{\bm{A}} = (\tilde{a}_{i,j})_{i,j = 1,\ldots,n}$ such that $\mathrm{diag}(\bm{A}) = \mathrm{diag}(\tilde{\bm{A}})$ and $\tilde{a}_{i,j} \in \{a_{i,j},0\}$ for any $i \neq j$ is said to be a \textit{minimal irreducible version} of an irreducible matrix $\bm{A}$, if any matrix obtained from $\tilde{A}$ by setting some off-diagonal element to $0$ is not irreducible anymore. 
\end{definition}
If we visualize a matrix $\bm{A}$ as in the definition above as a directed graph with vertices $V = \{1,\ldots,n\}$ representing the on-diagonal elements of $\bm{A}$ and edges $E = \{(i,j): a_{i,j} > 0\}$ representing the non-zero off-diagonal elements of $\bm{A}$, irreducibility of $\bm{A}$ is equivalent to strong connectedness of the graph of $\bm{A}$. The graph of a minimal irreducible version $\tilde{\bm{A}}$ of an irreducible matrix $\bm{A}$ is therefore a minimal strongly connected subgraph of the graph of $\bm{A}$ with $\tilde{V} = V$ and $\tilde{E} \subset E$. Also note that a continuous time Markov chain is irreducible if and only if its $Q$-matrix is irreducible. 

As a minimal requirement for stability we need to ensure irreducibility of the Markov process $(\cO,\cJ)$. We therefore introduce the following assumption.
\begin{enumerate}[leftmargin=*,label=($\mathscr{A}2$),ref=($\mathscr{A}2$)]
\item The modulator $J^+$ of the ascending ladder MAP is irreducible, i.e., $\bm{Q}^+$ is an irreducible matrix.\label{ass: irreducible}
\end{enumerate}
For general MAPs irreducibility of $J$ does not necessarily entail irreducibility of $J^+$, with the latter property being equivalent to the property that $\xi$ can reach a new maximum in any phase of $J$. E.g., if one of the L\'evy components $\xi^{(i)}$ is a negative subordinator and $\Delta_{j,i} < 0$ for any $j \in [n]$, $J^+$ is not irreducible since $\xi$ never reaches a new maximum when its phase is $i$. However, the following result shows that irreducibility of $J^+$ is given for a wide range of MAPs with irreducible modulator $J$. To give one particular example covered by Proposition \ref{prop: irr} below, suppose that for any $j \in [n]$ the L\'evy component $\xi^{(j)}$ is neither a negative subordinator nor spectrally negative with bounded variation,  or, when this fails for some $j \in [n]$ this is compensated for by some unbounded transitional jump of $\xi$ when $J$ switches to $j$. Then, $J^+$ is irreducible whenever $J$ is irreducible and Assumption \ref{ass: long time} is in place. We emphasize that upward regularity \ref{ass: regularity} is not needed for the statement of Proposition \ref{prop: irr}. Recall that for any measure $\nu$ on $(\R, \mathcal{B}(\R))$, the support $\supp(\nu)$ is defined as the set of points $x \in \R$ such that for any open neighborhood $U_x$ of $x$ it holds $\nu(U_x) > 0$.

\begin{proposition} \label{prop: irr}
Suppose that $J$ is irreducible and \ref{ass: long time} holds.
\begin{enumerate}[label = (\roman*), ref= (\roman*)]
\item Introduce the following conditions for $j \in [n]$:
\begin{enumerate}[leftmargin=*,start = 8,label=($\mathscr{\Alph*}(j))$,ref=($\mathscr{\Alph*}(j)$)]
\item $\xi^{(j)}$ is of unbounded variation or $\mathrm{supp}(\Pi_j) \cap (0,\infty) \neq \varnothing$; \label{irr ass1}
\item there exists $k \neq j$ such that $\mathrm{supp}(q_{k,j}F_{k,j})$ is unbounded from above. \label{irr ass2}
\end{enumerate}
Let $\Lambda_1 \coloneqq \{j \in [n]: \text{\ref{irr ass1} or \ref{irr ass2} holds}\}$ and 
$$\Lambda_2 \coloneqq \{j \in [n]\setminus \Lambda_1: \exists k \in \Lambda_1 \text{ s.t.\ } \supp(q_{k,j}F_{k,j}) \cap (0,\infty) \neq \varnothing\}.$$
Then, $J^+$ is irreducible if $\Lambda_1 \cup \Lambda_2 = [n]$.
\item Let $\tilde{\bm{Q}}$ be a minimal irreducible version of $\bm{Q}$. If  
\begin{align*}
&\{(i,j) \in [n]^2\setminus\{(i,i): i \in [n]\}: \tilde{q}_{i,j} > 0\} \\
&\quad \subset \{(i,j) \in [n]^2\setminus\{(i,i): i \in [n]\}: \text{\ref{irr ass1} holds or } \supp(q_{i,j}F_{i,j}) \cap (0,\infty) \neq \varnothing\},
\end{align*}
then $J^+$ is irreducible.
\end{enumerate}
\end{proposition}
\begin{proof}
\begin{enumerate}[leftmargin = *,label = (\roman*), ref= (\roman*)]
\item Fix $i,j \in [n]$ with $i \neq j$. We have to show that $\PP^{0,i}(\tau^+(j) < \infty) > 0$, where $\tau^+(j) \coloneqq \inf\{t \geq 0: J^+_t = j\}$. Recall that $(\sigma_n)_{n \in \N}$ denote the jump times of the modulating chain $J$. Let $n \in \N$ such that $\PP^{0,i}(J_{\sigma_n} = j) > 0$, which exists by irreducibility of $J$. Let $G_t \coloneq \sup\{0 \leq s < t: \xi_s = \overbar{\xi}_s\}$ be the last time before $t > 0$ at which $\xi$ attains its supremum. By construction of local time at the supremum $\mathsf{L}$, the range of $(\mathsf{L}^{-1}_t)_{t \geq 0}$ almost surely equals the set of times, when $\xi$ reaches a maximum.  Thus, we have 
\begin{equation}\label{eq: irr prop1}
\begin{split}
\PP^{0,i}(\tau^+(j) < \infty) &\geq \PP^{0,i}(G_{\sigma_{n+1}} \geq \sigma_n, J_{\sigma_n} = j)\\
&\geq \max\big\{\PP^{0,i}(\xi_{\sigma_n} \geq \overbar{\xi}_{\sigma_n-}, J_{\sigma_n} = j), \PP^{0,i}(G_{\sigma_{n+1}} > \sigma_n, J_{\sigma_n} = j) \big\}.
\end{split}
\end{equation}
Suppose first that \ref{irr ass1} holds. By the path decomposition of $(\xi,J)$ from Proposition \ref{char map levy}, we obtain 
\begin{equation}\label{eq: irr prop2}
\begin{split}
\PP^{0,i}(G_{\sigma_{n+1}} > \sigma_n, J_{\sigma_n} = j) &= \PP^{0,i}\big(\{J_{\sigma_n} = j\} \cap \{\exists t  \in (0,\sigma_{n+1} - \sigma_n): \xi_{t + \sigma_n} - \xi_{\sigma_n} \geq \overbar{\xi}_{\sigma_n} - \xi_{\sigma_n}\} \big)\\
&= \E^{0,i}\big[\PP^{0,j}(\exists t \in (0, \sigma_1): \xi_t \geq x)\vert_{x = \overbar{\xi}_{\sigma_n} - \xi_{\sigma_n}}; \, J_{\sigma_n} = j \big]\\
&\geq \E^{0,i}\big[\PP^{0,j}(\xi_{\sigma_1 \slash 2} \geq x)\vert_{x = \overbar{\xi}_{\sigma_n} - \xi_{\sigma_n}}; \, J_{\sigma_n} = j \big]\\
&= \E^{0,i}\Big[\Big(\int_0^\infty -2q_{j,j} \mathrm{e}^{2q_{j,j} t} \PP(\xi^{(j)}_{t} \geq x) \diff{t} \Big)\Big\vert_{x = \overbar{\xi}_{\sigma_n} - \xi_{\sigma_n}}; \, J_{\sigma_n} = j \Big]\\
& > 0.
\end{split}
\end{equation}
To argue that the last inequality holds, note that since \ref{irr ass1} was assumed, Theorem 24.7 in \cite{sato1999} yields that for any $t > 0$, $\supp(\PP(\xi^{(j)}_t \in \cdot))$ is not bounded from above. Thus, $\PP(\xi^{(j)}_{t} \geq x) > 0$ for any $x \in \R$ and hence 
$$\Big(\int_0^\infty -2q_{j,j} \mathrm{e}^{2q_{j,j} t} \PP(\xi^{(j)}_{t} \geq x) \diff{t} \Big)\Big\vert_{x = \overbar{\xi}_{\sigma_1} - \xi_{\sigma_1}} > 0, \quad \PP^{0,i}\text{-a.s.}.$$
Combining this with $\PP^{0,i}(J_{\sigma_n} = j) > 0$ by our choice of $n \in \N$, the inequality follows.

Suppose now that \ref{irr ass2} holds, i.e., $\supp(F_{k,j})$ is unbounded from above for some $k \neq j$ s.t.\ $q_{k,j} >  0$. Let $m \in \N$ such that $\PP^{0,i}(J_{\sigma_{m-1}} = k ,J_{\sigma_m} = j) > 0$, which exists by irreducibility of $J$ and $q_{k,j} > 0$. Then, again by Proposition \ref{char map levy},
$$\PP^{0,i}(\xi_{\sigma_m} \geq \overbar{\xi}_{\sigma_{m-}}, J_{\sigma_m} = j) = \sum_{k \neq j} \E^{0,i}\big[\PP(\Delta_{k,j} \geq x)]\vert_{x = \overbar{\xi}_{\sigma_m-}- \xi_{\sigma_{m}-}}; \, J_{\sigma_{m-1}} = k, J_{\sigma_m} = j\big] > 0,$$
where the inequality follows from 
$$\PP(\Delta_{k,j} \geq x)\vert_{x = \overbar{\xi}_{\sigma_m-}- \xi_{\sigma_{m}-}} > 0, \quad \PP^{0,i}\text{-a.s.},$$
thanks to assumed unboundedness of the support of $F_{k,j}$. We therefore conclude with \eqref{eq: irr prop1} that $\PP^{0,i}(\tau^+(j) < \infty) > 0$ for $j \in \Lambda_1$. Suppose now that $j \in \Lambda_2$, i.e., there exists $k \in \Lambda_1$ s.t.\ $\supp(q_{k,j}F_{k,j}) \cap (0,\infty) \neq \varnothing.$ Then, by Lemma \ref{lem: irr}, $q^+_{k,j} > 0$ and since $k \in \Lambda_1$, it follows from above that $\PP^{0,i}(\tau^+(k) < \infty) > 0$. Combining these observations yields again $\PP^{0,i}(\tau^+(j) < \infty) > 0$. Thus, the assumption $\Lambda_1 \cup \Lambda_2 = [n]$ implies $\PP^{0,i}(\tau^+(j) < \infty) > 0$ for any $j \neq i$, as desired.
\item Let $(i,j) \in [n]^2$ with $i \neq j$ s.t.\ $\tilde{q}_{i,j} > 0$. Suppose first that \ref{irr ass1} holds. Then, $q^+_{i,j} > 0$ holds if we can show that $\PP^{0,i}(G_{\sigma_2} > \sigma_1, J_{\sigma_1} = j) > 0$. This is an immediate consequence of \eqref{eq: irr prop2} with $n=1$ since $q_{i,j} = \tilde{q}_{i,j} > 0$ implies $\PP^{0,i}(J_{\sigma_1} = j) = -q_{i,j}\slash q_{i,i} > 0$. Suppose now that $\supp(q_{i,j} F_{i,j}) \cap (0,\infty) \neq \varnothing$. Then, again by Lemma \ref{lem: irr}, $q^+_{i,j} > 0$ as well. Thus, the assumption yields
$$\{(i,j) \in [n]^2\setminus\{(i,i): i \in [n]\}: \tilde{q}_{i,j} > 0\} \subset \{(i,j) \in [n]^2\setminus\{(i,i): i \in [n]\}: q^+_{i,j} > 0\},$$
and irreducibility of $\bm{Q}^+$ follows from irreducibility of $\bm{Q}$.
\end{enumerate}
\end{proof}
Assume for the rest of this section that \ref{ass: irreducible} is satisfied and denote by $\bm{\pi}^+ = (\pi^+(1),\ldots,\pi^+(n))$ the invariant distribution of $J^+$. Our main goal is to understand the asymptotic behavior of overshoots. As a natural extension of the well-known limiting distributional behavior of overshoots of L\'evy processes, cf.\ \cite{BertoinHarnSteutel1999}, it is shown in Theorem 28 of \cite{dereich2017} that under assumptions \ref{ass: long time} and \ref{ass: irreducible} the overshoot process converges weakly to the limiting distribution
$$\rho(\diff{y},\{i\}) \coloneq \frac{1}{\E^{0,\bm{\pi}^+}[H_1^+]} \bigg(\pi^+(i) d^+_i \delta_0(\diff{y}) + \one_{(0,\infty)}(y)\Big(\pi^+(i)\overbar{\Pi}{}^+_i(y) + \sum_{j \neq i} \pi^+(j)q^+_{j,i} \overbar{F}{}^+_{j,i}(y)\Big)\diff{y}\bigg),$$
$(y,i) \in \R_+ \times[n]$, if and only if $\E^{0,\bm{\pi}^+}[H_1^+] < \infty$.\footnote{Here we made a correction to \cite{dereich2017}, since in the authors' statement the limiting distribution of the parents modulator $J$, $\bm{\pi}$, appears instead of $\bm{\pi}^+$. As argued before, irreducibility of $J$ does not necessarily imply irreducibility of $J^+$ and even when $J^+$ is irreducible, $\bm{\pi}$ and $\bm{\pi}^+$ are not the same, see \cite[Proposition 2.19]{kyprianou2020}. Our analysis will show however that stationarity of the ascending ladder height's modulator and its stationary distribution are decisive for tight overshoots.} The Feller property of the overshoot process guarantees that in this case $\rho$ is also an invariant measure. We will show that deleting the scaling factor $\E^{0,\bm{\pi}^+}[H_1^+]^{-1}$ yields the essentially unique invariant measure of the overshoot process and hence a stationary distribution coinciding with $\rho$ exists iff overshoots are tight. Moreover, we will dig deeper into the mode of convergence, establishing conditions ensuring convergence in the total variation norm and exponential or polynomial speed of convergence, which also gives new results for the special case of L\'evy process overshoots.

An analytical tool of central importance to us is the resolvent of the overshoot process. Let $(\mathcal{P}_t)_{t \geq 0}$ be the transition function of $(\cO,\cJ)$ defined by $\mathcal{P}_tf(x,i) = \E^{x,i}[f(\cO_t, \cJ_t)]$ for any $f \in \cB_b(\R_+ \times [n]) \cup \cB_+(\R_+ \times [n])$ and $(\mathcal{U_\lambda})_{\lambda > 0}$ be the associated resolvent given by
$$\mathcal{U}_\lambda f(x,i) = \int_0^\infty \mathrm{e}^{-\lambda t} \mathcal{P}_tf(x,i) \diff{t},$$
for any $\lambda > 0$. Our proof for the explicit formula of the resolvent is close in spirit to the proof for the overshoot process of a L\'evy subordinator in Blumenthal \cite{blumenthal1992}, which in turn is a special case of a general result by It\={o} for Markov processes possessing a local time at a specific point of the state space, see \cite[Theorem 2.5.5]{ito2015}. The detailed proof is quite long and can be found in Appendix \ref{app: resolvent}.

\begin{theorem}\label{resolvent overshoot}
For any $f \in \cB_+(\R_+ \times [n]) \cup \cB_b(\R_+ \times [n])$ and $x \in \R_+$ it holds that
\begin{equation} \label{eq: resolvent formula}
  (\mathcal{U}_\lambda f(x,i))^\top_{i =1, \ldots, n} = (Q_\lambda f(x,i))^\top_{i=1,\ldots n} + \mathrm{e}^{-\lambda x}\bm{\Phi}^+(\lambda)^{-1} \cdot \bm{\psi}(f,\lambda),
\end{equation}
where
$$\bm{\psi}(f,\lambda) = \left(d_i^+ f(0,i) + \int_0^\infty Q_\lambda f(x,i) \,\Pi_i^+(\diff{x}) + \sum_{j \neq i} q^+_{i,j}\E[Q_\lambda f(\Delta^+_{i,j},j)]\right)^\top_{i=1,\ldots,n}$$
and
$$Q_\lambda f(x,i) = \int_0^x \mathrm{e}^{-\lambda t} f(x-t,i) \diff{t},\quad (x,i) \in \R_+ \times [n].$$
\end{theorem}

The resolvent formula has far reaching consequences for understanding the behavior of the MAP at first passage.
A first neat observation is the strong Feller property of the resolvent operator, which implies that $(\cO,\cJ)$ is a $T$-process.

\begin{corollary}\label{corol: t-process}
For any $\lambda > 0$ the resolvent $\mathcal{U}_\lambda$ has the strong Feller property. In particular the overshoot process $(\cO,\cJ)$ is a $T$-process.
\end{corollary}
\begin{proof}
Let $\lambda > 0$ and let $f \in \cB_b(\R_+ \times [n])$. Since we can write $Q_\lambda f(x,i) = \mathrm{e}^{-\lambda x} \int_0^x \mathrm{e}^{\lambda t} f(t,i) \diff{t}$, it follows that $(x,i) \mapsto Q_\lambda f(x,i)$ is continuous and hence $(x,i) \mapsto \mathcal{U}_\lambda f(x,i)$ is clearly continuous. Moreover, $\mathcal{U}_\lambda f$ is bounded and thus, $\mathcal{U}_\lambda \cB_b(\R_+ \times [n]) \subset \mathcal{C}_b(\R_+ \times [n])$ follows, i.e.\ $\mathcal{U}_\lambda$ has the strong Feller property. Hence, the resolvent kernel $\mathcal{R}_\lambda \coloneq \lambda \mathcal{U}_\lambda$ is a continuous  component for itself, implying  that $(\cO,\cJ)$ is a $T$-process.
\end{proof}

We will also use the resolvent formula combined with Proposition \ref{invariant measure resolvent} to determine an invariant measure for the overshoot process. To show its essential uniqueness, we need to establish Harris recurrence first, which is taken care  of in the following proposition.

\begin{proposition} \label{prop: harris}
The overshoot process $(\cO,\cJ)$ is Harris recurrent.
\end{proposition}
\begin{proof}
Let $j\in [n]$ be arbitrarily chosen and let $\mu \coloneqq \delta_0 \otimes \delta_j$. Fix $(x,i) \in \R_+ \times [n]$ and let $B \in \cB(\R_+ \times [n])$ such that $\mu(B) > 0$, i.e.\ $\{0\} \times \{j\} \in B$. Since $J^+$ is irreducible and $t \mapsto T_t^+$ is continuous and increases to $\infty$ as $t \to \infty$, it follows that $\PP^{x,i}(\mathfrak{t}^+(j) < \infty) > 0$, where $\mathfrak{t}^+(j) \coloneq \inf\{t> 0: \cJ_t^+ = j\}$ is the first hitting  time of $\{j\}$ of $\cJ^+$. Let $T_\Lambda = \inf\{t \geq 0: (\cO_t, \cJ_t) \in \Lambda\}$ be the first hitting time of a set $\Lambda \in \R_+ \times [n]$ by $(\cO,\cJ)$ and denote by $T^+_\Lambda$ the first hitting time of $(\cO^+,\cJ^+)$. By the sawtooth structure of $\cO^+$ we have $T^+_{\{0\}\times\{j\}} = x$, $\PP^{x,j}$-a.s.. Since $\mathfrak{t}^+(j) \leq T^+_{\{0\}\times \{j\}}$ it therefore follows by the strong Markov property of $(\cO^+,\cJ^+)$ that
\begin{align*}
\PP^{x,i}(T_B < \infty) \geq \PP^{x,i}\big(T_{\{0\} \times \{j\}} < \infty\big) &= \PP^{x,i}\big(T^+_{\{0\} \times \{j\}} < \infty\big)\\
&= \E^{x,i}\Big[\PP^{\cO^+_{\mathfrak{t}^+(j)},\cJ^+_{\mathfrak{t}^+(j)}}\big(T^+_{\{0\} \times \{j\}} < \infty\big) \one_{\{\mathfrak{t}^+(j) < \infty\}}\Big]\\
&= \PP^{x,i}(\mathfrak{t}^+(j) < \infty) > 0,
\end{align*}
where we used for the last equality that $\cJ^+_{\mathfrak{t}^+(j)} = j$ and $\cO^+_{\mathfrak{t}^+(j)} < \infty$ almost surely. It now follows from Proposition 2.1 in \cite{MeynTweedie1993} that $(\cO,\cJ)$ is irreducible with irreducibility measure
$$\mathcal{R}^\mu_1(\diff{y}) \coloneq \int_{\R_+ \times [n]} \mathcal{R}_1(x,\diff{y}) \, \mu(\diff{x}) = \mathcal{R}_1((0,j),\diff{y}), \quad y \in \R_+\times[n].$$ Moreover, $(\cO,\cJ)$ is a $T$-process by Corollary \ref{corol: t-process}. Hence, if we can argue that the process is non-evanescent, i.e. that there exists a compact set $K$ such that $(\cO,\cJ)$ returns to $K$ at arbitrarily large times, it will follow from Theorem 3.2 in \cite{MeynTweedie1993} that $(\cO,\cJ)$ is Harris recurrent. But non-evanescence is a direct consequence of the sawtooth structure of the overshoot process, since for the compact set $K \coloneq \{0\} \times [n]$ we have for any $(x,i) \in \R_+ \times [n]$ and $t > 0$
$$\PP^{x,i}(\inf\{s \geq t: (\cO_s,\cJ_s) \in \{0\}\times [n]\} < \infty) = \E^{x,i}[\PP^{\cO_{t},\cJ_t}(T_{\{0\} \times [n]} < \infty)] = 1,$$
where we used that $T_{\{0\} \times [n]} = y$, $\PP^{y,j}$-a.s.\ for any $(y,j) \in \R_+ \times [n]$ and $\cO_{t} < \infty$ almost surely. Hence, $(\cO,\cJ)$ is non-evanescent and the assertion follows.
\end{proof}

As a consequence of irreducibility implied by Harris recurrence and $(\cO,\cJ)$ being a $T$-process, we obtain that every compact set is petite, which will be useful for our proof of exponential convergence of the overshoot process later on.
\begin{corollary} \label{coro: petite}
Every compact set is petite for the overshoot process.
\end{corollary}
\begin{proof}
This is an immediate consequence of Theorem 5.1 in \cite{Tweedie1994} since $(\cO,\cJ)$ is a Harris recurrent $T$-process under the given assumptions and Harris recurrence implies irreducibility.
\end{proof}

Let us now determine the essential unique invariant measure of $(\cO,\cJ)$ and also derive a necessary and sufficient condition for the existence of a unique stationary distribution, which is the same condition needed for weak convergence of overshoots.

\begin{theorem} \label{theo: invariant}
The overshoot process $(\cO,\cJ)$ has an essentially unique invariant measure given by
\begin{equation}\label{eq: invariant measure overshoots}
\chi(\diff{y},\{i\}) = \pi^+(i) d^+_i \delta_0(\diff{y}) + \one_{(0,\infty)}(y)\Big(\pi^+(i)\overbar{\Pi}{}^+_i(y) + \sum_{j \neq i} \pi^+(j)q^+_{j,i} \overbar{F}{}^+_{j,i}(y)\Big)\diff{y}, \quad (y,i) \in \R_+ \times [n].
\end{equation}
In particular, a stationary distribution for $(\cO,\cJ)$ exists if and only if
$$\E^{0,\bm{\pi}^+}[H_1^+] \coloneq \sum_{i=1}^n \pi^+(i) \E^{0,i}[H_1^+] < \infty.$$
\end{theorem}
\begin{proof}
Define  $\bm{\alpha}(\lambda) \coloneq \bm{\pi}^+ \cdot \bm{\Phi}^+(\lambda)$ and
$$\alpha_\lambda \coloneq \sum_{i=1}^n \alpha_i(\lambda)\, \delta_{\{0\} \times \{i\}}.$$
Then, $\alpha_\lambda$ is a positive measure since $\Phi^+_i(\lambda) \geq 0$ and $G^+_{i,j}(\lambda) \in [0,1]$ for any $i,j = 1,\ldots,n$ imply that for any $\lambda > 0$ and $i = 1, \ldots,n$,
$$\alpha_i(\lambda) = \sum_{j=1}^n \pi^+(j)\bm{\Phi}^+(\lambda)_{j,i} \geq -\sum_{j=1}^n \pi^+(j) q^+_{j,i} = 0.$$
Since $\lim_{\lambda \downarrow 0} \bm{\Phi}^+(\lambda) = \bm{Q}^+$ and $\bm{\pi}^+$ is the stationary distribution of $J^+$ we have
$$\lim_{\lambda \downarrow 0} \bm{\pi}^+ \cdot \bm{\Phi}^+(\lambda) = \bm{\pi}^+ \cdot \bm{Q}^+ = \bm{0}_{1 \times n},$$
such that $\lim_{\lambda \downarrow 0} \alpha_\lambda(\R_+ \times [n]) = 0$ follows.
Recall from Appendix \ref{sec: markov stability} the notation $\mathcal{U}_\lambda^{\alpha_\lambda}(\diff{x}) \coloneqq \int_{\R_+ \times [n]} \mathcal{U}_\lambda(y, \diff{x})\, \alpha_\lambda(\diff{y})$. Plugging into the resolvent formula from Theorem \ref{resolvent overshoot} yields for any $f \in \cB_b(\R_+ \times [n]) \cap \cB_+(\R_+ \times [n])$ that
\begin{equation} \label{eq:scale resolvent}
\begin{split}
\mathcal{U}_\lambda^{\alpha_\lambda}(f) &= \bm{\alpha}(\lambda) \cdot (\mathcal{U}_\lambda f(0,i))_{i =1,\ldots,n}^\top\\
&= \bm{\pi}^+ \cdot \left(d_i^+ f(0,i) + \int_0^\infty Q_\lambda f(y,i)\, \Pi_i^+(\diff{y}) + \sum_{j \neq i} q^+_{i,j} Q_\lambda f(y,j)\, F^+_{i,j}(\diff{y}) \right)^\top_{i=1,\ldots n}.
\end{split}
\end{equation}
By monotone convergence and an integration by parts it follows that for any measure $\mu$ on $\R_+$
\begin{align*}
\lim_{\lambda \downarrow 0} \int_0^\infty Q_\lambda f(y,i) \, \mu(\diff{y}) &= \int_0^\infty \int_0^y f(y-t,i) \diff{t} \,\mu(\diff{y})\\
&= \int_0^\infty \int_0^y f(t,i) \diff{t} \, \mu(\diff{y})\\
&= \int_0^\infty \overbar{\mu}(y) f(y,i) \diff{y},
\end{align*}
where $\overbar{\mu}(y) \coloneq \mu(y,\infty)$. Thus, we obtain from \eqref{eq:scale resolvent} that
\begin{align*}
\lim_{\lambda \downarrow 0} \mathcal{U}_\lambda^{\alpha_\lambda}(f) &= \bm{\pi}^+ \cdot \left(d_i^+ f(0,i) + \int_0^\infty  f(y,i) \overbar{\Pi}{}^+_i(y) \diff{y} + \sum_{j \neq i} q^+_{i,j} f(y,j) \overbar{F}{}^+_{i,j}(y) \diff{y}\right)_{i=1,\ldots,n}^\top\\
&= \sum_{i = 1}^n \pi^+(i) \left(d_i^+ f(0,i) + \int_0^\infty f(y,i) \overbar{\Pi}{}^+_i(y) \diff{y} + \sum_{j \neq i}q^+_{i,j} \int_0^\infty f(y,i) \overbar{F}{}^+_{i,j}(y)\diff{y}\right)\\
&= \sum_{i = 1}^n \left(\pi^+(i) \left(d_i^+ f(0,i) + \int_0^\infty f(y,i) \overbar{\Pi}{}^+_i(y) \diff{y}\right) + \sum_{j \neq i} \pi^+(j)q^+_{j,i} \int_0^\infty f(y,i) \overbar{F}{}^+_{j,i}(y)\diff{y}\right)\\
&= \int_{\R_+ \times [n]} f(y,z) \, \chi(\diff{y} \times \diff{z}),
\end{align*}
where for the second to last equality we used that
\begin{align*}
  \sum_{i=1}^n \pi^+(i) \sum_{j \neq i} q^+_{i,j}\int_0^\infty f(y,j)\overbar{F}{}^+_{i,j}  \diff{y} &= \sum_{j=1}^n \sum_{i \neq j} q^+_{i,j}\pi^+(i) \int_0^\infty f(y,j) \overbar{F}{}^+_{i,j}(y) \diff{y}\\
  &= \sum_{i=1}^n \sum_{j \neq i} q^+_{j,i}\pi^+(j) \int_0^\infty f(y,i) \overbar{F}{}^+_{j,i}(y) \diff{y}.
\end{align*}
From Proposition \ref{invariant measure resolvent} it now follows that $\chi$ is indeed an invariant measure for $(\cO,\cJ)$. By irreducibility of $J^+$, $(\cO,\cJ)$ is a Harris recurrent Feller process according to Propositions \ref{overshoot feller} and \ref{prop: harris} and hence Theorem 2.5 in \cite{AzemaDufloRevuz1969} yields that $\chi$ is essentially unique.

Finally, using the Laplace exponent of $(H^+,J^+)$ we obtain
\begin{align*}
\big(\E^{0,i}[H^+_1 \one_{\{J^+_1 = j\}}] \big)_{i,j=1,\ldots,n} &=  \frac{\uppartial}{\uppartial \lambda} \bm{\Phi}^+(\lambda)\big\vert_{\lambda = 0} \\
&= \mathrm{diag}\big(\big(\E[H_1^{+,(i)}] \big)\big)_{i\in [n]} + \bm{Q}^+ \odot \big(\E[\Delta^+_{i,j}] \big)_{i,j=1,\ldots,n}\\
&= \mathrm{diag}\Big(\Big(d_i^+ + \int_0^\infty \overbar{\Pi}{}^+_i(x) \diff{x} \Big)\Big)_{i\in [n]} + \bm{Q}^+ \odot \Big(\int_0^\infty \overbar{F}{}^+_{i,j}(x) \diff{x} \Big)_{i,j=1,\ldots,n},
\end{align*}
and hence
$$\E^{0,i}\big[H_1^+\big] = d_i^+ + \int_0^\infty \overbar{\Pi}{}^+_i(x) \diff{x} + \sum_{j \neq i} q^+_{i,j} \int_0^\infty \overbar{F}{}^+_{i,j}(x) \diff{x},  \quad i \in [n],$$
which shows that
$$\chi(\R_+ \times [n]) = \E^{0,\bm{\pi}^+}\big[H_1^+\big].$$
Thus, $\chi$ can be normalized to an invariant distribution if and only if $\E^{0,\bm{\pi}^+}[H_1^+] < \infty$.
\end{proof}
\begin{remark}
The finite mean condition for the ascending ladder height process is exactly the same condition, which is necessary and sufficient for stationary overshoots of MAPs in the sense of weak convergence. As shown in Theorem 35 of \cite{dereich2017} as an extension of Theorem 8 in \cite{DoneyMaller2002} for L\'evy processes, this condition is equivalent to $\E^{0,i}[\lvert \xi_1 \rvert] < \infty$ and either $\lim_{t \to \infty} \xi_t = \infty$, $\PP^{0,i}$-a.s., or $\limsup_{t \to \infty} \xi_t = -\liminf_{t \to \infty} \xi_t = \infty$, $\PP^{0,i}$-a.s., together with
\begin{equation}\label{eq: ladder height finite mean}
\int_\kappa^\infty \frac{x \sum_{i=1}^n\bm{\Pi}(i,[x,\infty) \times [n])}{1 + \int_0^x \int_y^\infty  \sum_{i=1}^n\bm{\Pi}(i,(-\infty,-z] \times [n]) \diff{z}\diff{y}} \diff{x} < \infty,
\end{equation}
for some $\kappa > 0$.
\end{remark}

Classical results on the interplay between Harris recurrence and invariant measures for Markov process (cf.\ Appendix \ref{sec: markov stability}) now also yields that $\chi$ is a maximal Harris meaure.
\begin{corollary}\label{coroll: maximal harris}
The invariant measure $\chi$ given in \eqref{eq: invariant measure overshoots} is a maximal Harris measure.
\end{corollary}
\begin{remark}
This could have also been shown directly by an alternative proof of Proposition \ref{prop: harris} based on Kaspi and Mandelbaum's characterization of Harris recurrence in terms of almost sure finiteness of first hitting times \eqref{eq: kaspi} and the characteristic property \eqref{eq: compensation map} of the L\'evy system belonging to $(H^+,J^+)$.
\end{remark}

Having established the existence of a unique invariant distribution, we now proceed to investigate ergodicity of overshoots. To this end, we need to find criteria ensuring the existence of an irreducible skeleton chain. One of these criteria will be a strictly positive creeping probability of the MAP and we lift a sufficient criterion for this to happen from the well-known L\'evy process situation.

\begin{lemma} \label{lemma: creeping}
Suppose that $d_i^+ > 0$ for some $i \in [n]$. Then, for any $t > 0$ we have
  $$\PP^{0,i}\big(\xi_{T_t} = t, J_{T_t} = i \big) > 0.$$
\end{lemma}
\begin{proof}
Let $\sigma^+_1$ be the first jump time of $J^+$. If $q^+_{i,i} = 0$, then under $\PP^{0,i}$, $H^+$ is a L\'evy subordinator with positive drift and therefore has positive creeping probability by Theorem 5.9 in \cite{kyprianou2014}, implying the claim. Suppose now $-q^+_{i,i} > 0$. Then, using the representation from Proposition \ref{char map levy} we have
\begin{align*}
  \PP^{0,i}(\cO_t = 0, \cJ_t = i) &= \PP^{0,i}(\cO^+_t = 0, \cJ^+_t = i)\\
  &\geq \PP^{0,i}\Big(H^{+,0,i}_{T^{+,0,i}_t} = t, T^{+,0,i}_t < \sigma^+_1\Big)\\
  &= \int_0^\infty \PP\Big(H^{+,(i)}_{T^{+,(i)}_t} = t, T^{+,(i)}_t < y\Big)\, \PP^{0,i}(\sigma_{1}^+ \in \diff{y})\\
  &= -q^+_{i,i} \int_0^\infty \mathrm{e}^{q^+_{i,i}y} \PP\Big(H^{+,(i)}_{T^{+,(i)}_t} = t, T^{+,(i)}_t < y\Big) \diff{y},
\end{align*}
where we used independence of $H^{+,0,i}$ and $J^+$ for the third equality. Since again by Theorem 5.9 in \cite{kyprianou2014}, $d^+_i > 0$ gives that $\PP(H^{+,(i)}_{T^{+,(i)}_t} = t) > 0$ for all $t \geq 0$ and
$$\lim_{y \to \infty }\PP\Big(H^{+,(i)}_{T^{+,(i)}_t} = t, T^{+,(i)}_t < y\Big) = \PP\Big(H^{+,(i)}_{T^{+,(i)}_t} = t\Big),$$
it follows that there is $z > 0$ such that $\PP\big(H^{+,(i)}_{T^{+,(i)}_t} = t, T^{+,(i)}_t < y\big) > 0$ for all $y \geq z$ and hence, from above it follows that
$$\PP^{0,i}(\cO_t = 0, \cJ_t = i) \geq -q^+_{i,i} \int_z^\infty \mathrm{e}^{q^+_{i,i}y} \PP\Big(H^{+,(i)}_{T^{+,(i)}_t} = t, T^{+,(i)}_t < y\Big) \diff{y} > 0.$$
\end{proof}
\begin{remark}
The irreducibility assumption \ref{ass: irreducible} is not required for this statement.
\end{remark}
Let us now state properties of the ascending ladder height process that imply existence of an irreducible skeleton of $(\cO,\cJ)$.

\begin{proposition} \label{prop: aperiod}
If
\begin{enumerate}[label=(\roman*),ref=(\roman*)]
  \item $d_i^+ > 0$ for some $i \in [n]$, then $(\cO,\cJ)$ is aperiodic and any $\Delta$-skeleton is  irreducible.
  \item for some $j \in [n]$ it holds $\mathrm{Leb}\vert_{(0,\infty)} \ll \Pi^+_j\vert_{(0,\infty)}$, then any $\Delta$-skeleton $(\cO^\Delta,\cJ^\Delta)$ is $\mathrm{Leb}_+ \otimes \delta_j$-irreducible. \label{prop: aperiod2}
  \item for some $j \in [n]$ there exists an interval $(a,b) \subset \R_+$ such that $\mathrm{Leb}\vert_{(a,b)} \ll \Pi^+_j\vert_{(a,b)}$ and for any $i \in [n]$ and $x> 0$ it holds that $U^+_{i,j}([0,x)) > 0$, then for any $\Delta \in (0, (a+b)\slash 2)$, the $\Delta$-skeleton $(\cO^\Delta,\cJ^\Delta)$ is $\mathrm{Leb}_+(\cdot \cap (a,(a+b)\slash2))\otimes \delta_j$-irreducible. \label{prop: aperiod3} 
  \item for some $(j,k) \in [n]^2$ with $k \neq j$ it holds $\mathrm{Leb}\vert_{(0,\infty)} \ll F^+_{k,j}\vert_{(0,\infty)}$ and $q^+_{k,j}> 0$, then any $\Delta$-skeleton $(\cO^\Delta,\cJ^\Delta)$ is $\mathrm{Leb}_+ \otimes \delta_j$-irreducible. \label{prop: aperiod4}
  \item for some $(j,k) \in [n]^2$ with $k \neq j$ it holds $q^+_{k,j} > 0$, there exists an interval $(a,b) \subset \R_+$ such that $\mathrm{Leb}\vert_{(a,b)} \ll F^+_{k,j}\vert_{(a,b)}$ and for any $i \in [n]$ and $x> 0$ it holds that $U^+_{i,k}([0,x)) > 0$, then for any $\Delta \in (0, (a+b)\slash 2)$, the $\Delta$-skeleton $(\cO^\Delta,\cJ^\Delta)$ is $\mathrm{Leb}_+(\cdot \cap (a,(a+b)\slash2))\otimes \delta_j$-irreducible. \label{prop: aperiod5}
\end{enumerate}
\end{proposition}
\begin{proof}\quad
\begin{enumerate}[label=(\roman*),ref=(\roman*)]
\item The singleton set $C = \{0\} \times \{i\}$ is trivially small (just choose $\nu_a = \mathcal P_t((0,i),\cdot)$ for $a =\delta_t$ and some $t > 0$.). Further, $C \in \mathcal{B}^+(\R_+ \times [n])$ since Corollary \ref{coroll: maximal harris}  tells us that the invariant measure $\chi$ is an irreducibility measure for $(\cO,\cJ)$ and thanks to $d_i^+>0$, we have $\chi(C) > 0$. Lemma \ref{lemma: creeping} gives that
$$\PP^{0,i}((\cO_t,\cJ_t) \in C) = \PP^{0,i}(\cO_t = 0,\cJ_t = i) > 0$$
for all $t \geq 0$, which implies that $(\cO,\cJ)$ is aperiodic with defining singleton set $C = \{0\} \times \{i\}$, which by Lemma \ref{lemma: aperiod} also implies that any $\Delta$-skeleton is irreducible.
\item Let $B = B_1 \times B_2 \in \cB(\R_+ \times [n])$ such that $\mathrm{Leb}_+ \otimes \delta_j(B) > 0$. Without loss of generality we may assume that $0 \notin B_1$. Since $J^+$ is irreducible it holds $\PP^{0,i}(J^+_t = j) > 0$ for any $t > 0$ and $i \in [n]$ and hence by monotone convergence,
$$\lim_{x \to \infty} U^+_{i,j}([0,x)) = \int_0^\infty \PP^{0,i}(J^+_t = j) \diff{t} > 0,$$
which yields that there exists $\overbar{x} > 0$ such that $U^+_{i,j}([0,x)) > 0$ for all $x \geq \overbar{x}$ and $i \in [n]$. For given $x \geq 0$ let $t > x + \overbar{x}$. Then, by the overshoot formula and Fubini it follows that for any $i \in [n]$ we have
\begin{equation}\label{eq: aperiod1}
  \begin{split}
\PP^{x,i}(\cO_t \in B_1, \cJ_t \in B_2) &\geq \int_{[0,t-x)} \int_{B_1}  \, \Pi^+_j(y+\diff{u}) \, U^+_{i,j}(t - x - \diff{y})\\
 &= \int_{[0,t-x)} \, \Pi^+_j(B_1 + t-x-y) \, U^+_{i,j}(\diff{y}).
\end{split}
\end{equation}
Since by translation invariance of the Lebesgue measure it holds $\mathrm{Leb}(B_1+z) > 0$ for any $z \geq 0$ and $\mathrm{Leb}\vert_{(0,\infty)} \ll \Pi^+_j\vert_{(0,\infty)}$ by assumption, it follows that for any $y \in [0,t-x)$ we have $\Pi^+_j(B_1 + t - x - y)  > 0$. By our choice of $t$ it also holds that $U^+_{i,j}([0,t-x)) > 0$, thus \eqref{eq: aperiod1} yields that $\PP^{x,i}(\cO_t \in B_1, \cJ_t \in B_2) > 0$. Hence, given $\Delta > 0$, choosing $n_x \in \N$ large enough such that $n_x\Delta > x + \overbar{x}$, it follows that $\PP^{x,i}((\cO_{n_x\Delta},\cJ_{n_x\Delta}) \in B) > 0$ for any $i \in [n]$, which shows that any $\Delta$-skeleton is $\mathrm{Leb}_+ \otimes \delta_j$-irreducible.
\item Choose $B = B_1 \times B_2 \in \cB(\R_+ \times [n])$ such that $\mathrm{Leb}_+(\cdot \cap (a,b)) \otimes \delta_j(B) > 0$. Again we may assume that $0 \notin B_1$. Let $(x,i) \in \R_+ \times [n]$ and $t \in (x, x +(b-a)\slash2).$ Since for any $z \geq 0$ it holds that
$$(B_1+z) \cap (a,b) = (B_1 \cap (a-z, b-z)) + z$$
it follows for $z \in (0, (b-a)\slash 2)$ by translation invariance of the Lebesgue measure that
$$\mathrm{Leb}((B_1+z)\cap (a,b)) = \mathrm{Leb}(B_1 \cap (a-z, b-z)) \geq \mathrm{Leb}(B_1 \cap (a, (a+b)\slash 2)) > 0.$$
By our choice of $t \in (x, x+ (b-a)\slash 2)$ it holds that $0 < t - x- y < (b-a)\slash 2$ for all $y \in (0,t-x)$ and therefore $\mathrm{Leb}((B_1 + t - x -y)\cap (a,b)) > 0$, which by our assumption $\mathrm{Leb} \vert_{(a,b)} \ll \Pi^+_j\vert_{(a,b)}$ implies that $\Pi^+_j(B_1 + t - x -y) > 0$. Since $U^+_{i,j}([0,t-x)) > 0$ by assumption it now follows from \eqref{eq: aperiod1} that $\PP^{x,i}((\mathcal{O}_t, \cJ_t) \in B) > 0$. Hence, given $\Delta \in (0,(b-a)\slash 2)$, if we choose $k \in \N$ such that $k\Delta \in (x,x+(b-a)\slash2)$ it follows that $\PP^{x,i}((\cO_{k \Delta}, \cJ_{k \Delta}) \in B) > 0$ and therefore $\sum_{k=1}^\infty \PP^{x,i}((\cO_{k \Delta}, \cJ_{k\Delta}) \in B) > 0$. Since $(x,i) \in \R_+ \times [n]$ was chosen arbitrarily we conclude that the $\Delta$-skeleton is irreducible with irreducibility measure $\mathrm{Leb}_+(\cdot \cap (a,(a+b)\slash 2)) \otimes \delta_j$.
\end{enumerate}
Parts \ref{prop: aperiod4} and \ref{prop: aperiod5} can be demonstrated exactly as parts \ref{prop: aperiod2} and \ref{prop: aperiod3} when instead of \eqref{eq: aperiod1} we use that for $B = B_1 \times B_2 \in \mathcal{B}(\R_+ \times [n])$ with $j \in B_2$, $(x,i) \in \R_+ \times [n]$ and $t > x$ it holds
$$\PP^{x,i}(\cO_t \in B_1, \cJ_t \in B_2) \geq q^+_{k,j} \int_{[0,t-x)} \, F^+_{k,j}(B_1 + t-x-y) \, U^+_{i,k}(\diff{y}).$$
\end{proof}
\begin{remark}
The condition in part \ref{prop: aperiod3} and \ref{prop: aperiod5} that $U^+_{i,j}([0,x)) > 0$ for all $i \neq j$ is non-redundant in general. If, e.g., $F^+_{i,j}([0,x)) = 0$ for some $i \neq j$, then $U^+_{i,j}([0,x)) = 0$.
\end{remark}

These conditions in combination with Harris recurrence now allow us to determine when $(\cO,\cJ)$ is ergodic.

\begin{theorem} \label{theo: ergodic}
Suppose that $\E^{0,\bm{\pi}^+}[H_1^+] < \infty$. Then, under any of the conditions of Proposition \ref{prop: aperiod}, it holds that $(\cO,\cJ)$ is ergodic, i.e.\
$$\forall (x,i) \in \R_+ \times [n]: \quad \lim_{t \to \infty} \lVert \PP^{x,i}((\cO_t,\cJ_t) \in \cdot) - \rho\rVert_{\mathrm{TV}} = 0,$$
where for $(x,i) \in \R_+ \times [n]$,
\begin{equation} \label{eq: stationary dist}
\rho(\diff{x}, \{i\}) \coloneq \frac{1}{\E^{0,\bm{\pi}^+}[H_1^+]} \Big(\pi^+(i) d^+_i \delta_0(\diff{y}) + \one_{(0,\infty)}(y)\Big(\pi^+(i)\overbar{\Pi}{}^+_i(y) + \sum_{j \neq i} \pi^+(j)q^+_{j,i} \overbar{F}{}^+_{j,i}(y)\Big)\diff{y}\Big),
\end{equation}
is the stationary distribution of $(\cO,\cJ)$.
\end{theorem}
\begin{proof}
As a consequence of Proposition \ref{overshoot feller}, Proposition \ref{prop: harris} and Theorem \ref{theo: invariant}, it follows that under any of the conditions of Proposition \ref{prop: aperiod}, $(\cO,\cJ)$ is a positive Harris recurrent Borel right Markov process with unique stationary distribution given in \eqref{eq: stationary dist} such that some $\Delta$-skeleton is irreducible. Thus, Theorem 6.1 in \cite{MeynTweedie1993} yields the assertion.
\end{proof}

A direct implication of ergodicity is that a continuous time version of the von Neumann--Birkhoff ergodic theorem holds, see the discussion in \cite{sandric2017}.

\begin{corollary}
Given the assumptions from Theorem \ref{theo: ergodic}, it holds for any $f \in L^p(\R_+ \times [n],\rho)$ and $(x,i) \in \R_+ \times [n]$ that
$$\lim_{T \to \infty} \frac{1}{T} \int_0^T f(\cO_t,\cJ_t) \diff{t} = \rho(f), \quad \PP^{x,i}\text{-a.s. and in } L^p(\PP^\rho). $$
\end{corollary}

Once we have derived an analogue of Vigon's équations amicales inversés in Section \ref{sec: vigon}, we will be able to express conditions on the L\'evy system $\bm{\Pi}$ of $(\xi,J)$ that guarantee one of the conditions on the L\'evy system $\bm{\Pi}^+$ of $(H^+,J^+)$ required for ergodicity. For the moment we content ourselves with studying the drifts $d_i^+$ of the subordinators associated to the ascending ladder height process.

\begin{lemma}
If $J$ is irreducible, then for any $i  \in [n]$ and an appropriate scaling of local time, the diffusion parameter $b_i$ of $\xi^{(i)}$ is given by
$$b_i^2 = 2d_i^+\hat{d}{}^+_i.$$
\end{lemma}
\begin{proof}
Let $i \in [n]$. Considering the diagonal of $\bm{\Psi}$, the spatial Wiener--Hopf factorization \eqref{eq: wiener-hopf} yields for every $\theta \in \R$
\begin{align*}
&\mathrm{i}a_i\theta  - \frac{b_i^2}{2}\theta^2 + \int_{\R}\big(\mathrm{e}^{\mathrm{i}\theta x} -1 - \mathrm{i}\theta x \one_{[-1,1]}(x) \big)\, \Pi_i(\diff{x}) + q_{i,i}\\
&\quad= \Big(\hat{q}{}^+_{i,i}-\hat{\dagger}{}^+_i - \mathrm{i}\hat{d}{}^+_i\theta  + \int_0^\infty\big(\mathrm{e}^{-\mathrm{i}\theta x} -1 \big)\, \hat{\Pi}{}^+_i(\diff{x})\Big) \cdot \Big(q^+_{i,i}-\dagger^+_i + \mathrm{i}d^+_i\theta  + \int_0^\infty\big(\mathrm{e}^{\mathrm{i}\theta x} - 1 \big)\, \Pi^+_i(\diff{x})\Big)\\
&\qquad + \sum_{k\neq i} \frac{\pi(k)}{\pi(i)}\hat{q}{}^+_{k,i}q^+_{k,i}\hat{G}{}^+_{k,i}(-\theta)G^+_{k,i}(\theta).
\end{align*}
Since
$$\lim_{\lvert \theta \rvert \to \infty} \frac{1}{\theta^2} \int_{\R} \big(\mathrm{e}^{\mathrm{i}\theta x} - 1 -\mathrm{i}\theta x\one_{[-1,1]}(x)\big) \, \Pi_i(\diff{x}) = 0,$$
and
$$\lim_{\lvert \theta \rvert \to \infty} \frac{1}{\lvert \theta \rvert} \int_0^\infty (\mathrm{e}^{\mathrm{i}\theta x} - 1) \, \Pi^+_i(\diff{x}) = 0, \quad \lim_{\lvert \theta \rvert \to \infty} \frac{1}{\lvert \theta \rvert} \int_0^\infty (\mathrm{e}^{-\mathrm{i}\theta x} - 1) \, \hat{\Pi}{}^+_i(\diff{x}) = 0,$$
and moreover $\lvert \hat{G}{}^+_{k,i}(-\theta)G^+_{k,i}(\theta)\rvert \leq 1$, dividing both sides of the equation by $\theta^2$ and letting $\theta \to \infty$ yields $b_i^2 = 2d_i^+\hat{d}{}^+_i.$
\end{proof}

Thus, $b_i > 0$ if and only if $d_i^+ \wedge \hat{d}{}^+_i > 0$ and therefore Theorem \ref{theo: ergodic} shows that for any MAP with tight overshoots and some L\'evy component $\xi^{(i)}$ with non-zero diffusion component, convergence to the stationary overshoot distribution takes place in total variation.
\smallskip

As a next step we show that under appropriate moment conditions on the L\'evy processes and transitional jumps underlying the ascending ladder height MAP, overshoots converge with polynomial rate and in case of existence of exponential moments even exponentially fast. Thus, the speed of convergence is reflected in the tail behavior of the jump measures associated to the L\'evy system $\bm{\Pi}^+$, with light tails giving exponential decay and moderately heavy tails resulting in polynomial decay. For the proof we yet again make use of the resolvent formula \eqref{eq: resolvent formula} to find Lyapunov functions needed for the resolvent drift criteria \eqref{eq: exp ergodic2} and \eqref{eq:drift subgeo}.

\begin{theorem} \label{theo: exp ergodicity}
Suppose that one of the conditions of Proposition \ref{prop: aperiod} is satisfied.
\begin{enumerate}[label=(\roman*),ref=(\roman*)]
\item  Suppose there exists $\lambda > 0$ such that the exponential $\lambda$-moment exists for all $H^{+,(i)}$, $i \in [n]$, and for all $\Delta^+_{i,j}$, $i \neq j$, such that $q^+_{i,j}  \neq 0$. Then, for the choice $V_\lambda(x,i) = \exp(\lambda x)$, $(x,i) \in \R_+ \times [n]$, $(\cO,\cJ)$ is $\mathcal{R}_\alpha V_\lambda$-uniformly ergodic for any $\alpha > 0$, i.e.\,
\begin{equation} \label{eq:fconv}
\sup_{\lvert f \rvert \leq \mathcal{R}_\alpha V_\lambda} \big\lvert \E^{x,i}[f(\cO_t,\cJ_t)] - \rho(f) \big\rvert \leq C(\alpha)\mathcal{R}_{\alpha} V_\lambda(x,i) \mathrm{e}^{-\kappa(\alpha) t}, \quad (x,i) \in \R_+ \times [n],
\end{equation}
for some constants $C(\alpha),\kappa(\alpha) > 0$. Moreover for any $\varepsilon \in (0,\lambda)$, it holds that 
\begin{equation}\label{eq: tv exp}
\lVert \PP^{x,i}((\cO_t,\cJ_t) \in \cdot) - \rho \rVert_{\mathrm{TV}} \leq \mathfrak C(\alpha,\varepsilon)\mathcal{R}_\alpha V_\lambda(x,i) \mathrm{e}^{-\alpha(\lambda - \varepsilon) t/(\alpha + \lambda)}, \quad (x,i) \in \R_+ \times [n],
\end{equation}
for some constant $\mathfrak C(\alpha,\varepsilon)>0$.
\label{theo: exp ergodicity 1}
\item Suppose that for some $\lambda > 1$ the $\lambda$-moment exists for all $H^{+,(i)}$, $i \in [n]$, and for all $\Delta^+_{i,j}$, $i \neq j$, such that $q^+_{i,j}  \neq 0$. Then, there exists $\tilde{C} > 0$ such that 
$$\lVert \PP^{x,i}((\cO_t,\cJ_t) \in \cdot) - \rho \rVert_{\mathrm{TV}} \leq \tilde{C} \mathcal{R}_\lambda \tilde{V}_\lambda(x,i) (1+t)^{1-\lambda}, \quad (x,i) \in \R_+ \times [n],$$
where  $\tilde{V}_\lambda(x,i) = \mathrm{e}^{\lambda x} \one_{[0,1)}(x) + x^\lambda \one_{[1,\infty)}(x)$. \label{theo: exp ergodicity 2}
\end{enumerate}
\end{theorem}
\begin{proof}
\begin{enumerate}[label=(\roman*),ref=(\roman*)]
\item For a matrix $A \in \R^{n \times n}$, let $\lVert A \rVert_\infty \coloneq \max_{i=1,\ldots n} \sum_{j=1}^n \lvert a_{ij} \rvert$ be its matrix norm induced by the $\sup$-norm. Let $Q_\alpha$ be the operator from the statement of Theorem \ref{resolvent overshoot}. Then,
$$\alpha Q_\alpha V_\lambda(x,i) = \alpha \int_0^x \mathrm{e}^{-\alpha t} V_\lambda(x-t,i) \diff{t} = \frac{\alpha}{\alpha + \lambda}\big(\mathrm{e}^{\lambda x} - \mathrm{e}^{-\alpha x}\big), \quad (x,i) \in \R_+ \times [n].$$
Since $\mathrm{e}^{\lambda x} - \mathrm{e}^{-\alpha x} = O(x)$ as $x \downarrow 0$ and $\Pi_i^+$ are L\'evy subordinator measures, it follows that
$$\int_0^1 \alpha Q_{\alpha} V_\lambda(x,i) \, \Pi_i^+(\diff{x}) < \infty.$$
Moreover, by assumption $H^{+,(i)}$ has an exponential $\lambda$-moment, which according to Theorem 3.6 of \cite{kyprianou2014} is equivalent to $\int_1^\infty \exp(\lambda x) \,\Pi_i^+(\diff{x}) < \infty$, implying that
$$\int_1^\infty \alpha Q_{\alpha} V_\lambda(x,i) \, \Pi_i^+(\diff{x}) < \infty$$
as well and thus
$$\int_0^\infty \alpha Q_\alpha V_\lambda(x,i) \, \Pi_i^+(\diff{x}) < \infty$$
for all $i \in [n]$. Since additionally $\E[\exp(\lambda \Delta^+_{i,j})] < \infty$ for any $i,j \in [n]$ such that $i \neq j$ and $q^+_{i,j} > 0$, it follows that if we define
\begin{align*}
b(\alpha) &\coloneqq \alpha \lVert \bm{\Phi}^+(\alpha)^{-1} \rVert_\infty \sum_{i=1}^n \left(d_i^+ + \int_0^\infty Q_\alpha V_\lambda(x,i)\, \Pi_i^+(\diff{x}) + \sum_{j \neq i} q^+_{i,j} \E[Q_\alpha V_\lambda(\Delta^+_{i,j},j)] \right)\\
&\leq \lVert \bm{\Phi}^+(\alpha)^{-1} \rVert_\infty \sum_{i=1}^n \left(\alpha d_i^+ + \frac{\alpha}{\alpha + \lambda}\int_0^\infty \big(\mathrm{e}^{\lambda x} - \mathrm{e}^{-\alpha x} \big)\, \Pi_i^+(\diff{x}) + \sum_{j \neq i} \frac{\alpha q^+_{i,j}}{\alpha + \lambda} \E\big[\exp\big(\lambda \Delta^+_{i,j}\big)\big] \right),
\end{align*}
we have $b < \infty$. Using \eqref{eq: resolvent formula} it therefore follows for any $i \in [n]$ that
\begin{equation} \label{eq: exp1}
\begin{split}
\mathcal{R}_\alpha V_\lambda(x,i) = \alpha \mathcal{U}_\alpha V_\lambda(x,i) &\leq \frac{\alpha}{\alpha + \lambda}\big(\mathrm{e}^{\lambda x} - \mathrm{e}^{-\alpha x}\big) + b(\alpha)\\
&< \frac{\alpha}{\alpha + \lambda} V_\lambda(x,i) + b(\alpha), \quad (x,i) \in \R_+ \times [n],
\end{split}
\end{equation}
which shows that \eqref{eq: exp ergodic2} holds for $\beta_0 = \alpha/(\alpha + \lambda) \in (0,1)$ and $b(\alpha) < \infty$ as above. Under the given assumptions, $(\cO,\cJ)$ is Harris recurrent and there exists an irreducible skeleton chain by Proposition \ref{prop: harris} and Proposition \ref{prop: aperiod}, hence $(\cO,\cJ)$ is irreducible and aperiodic. Moreover, $V_\lambda$ is unbounded off petite sets since $V_\lambda$ is increasing and continuous and hence for any $z > 0$, the set $\{(x,i) \in \R_+ \times [n]: V_\lambda(x,i) \leq z\}$ is compact and hence petite, according to Corollary \ref{coro: petite}. Thus, \eqref{eq: exp ergodic2} being satisfied for our choice of $V_\lambda$,  Theorem 5.2 in \cite{DownMeynTweedie1995} implies that  $(\cO,\cJ)$ is $\mathcal{R}_\alpha V_\lambda$-uniformly ergodic.

To establish the more explicit rate of convergence for the total variation norm in \eqref{eq: tv exp}, note that \eqref{eq: exp1} combined with \eqref{eq: erg equiv} shows that for the petite set $C(\varepsilon) = \{V_\lambda \leq (\alpha + \lambda) b(\alpha)/\varepsilon\}$, $\varepsilon \in (0,\lambda)$ and $\phi_{\alpha,\varepsilon}(z) = (\lambda - \varepsilon)z/(\alpha + \lambda)$ we have 
\begin{equation*}
  \mathcal{R}_\alpha V_\lambda \leq \frac{\alpha + \varepsilon}{\alpha + \lambda} V_\lambda + b(\alpha)\one_{C(\varepsilon)} = V_\lambda - \phi_{\alpha,\varepsilon} \circ V_\lambda + b(\alpha) \one_{C(\varepsilon)},
\end{equation*}
and thus, the claim follows easily from \eqref{eq:subgeo rate}. 
\item Since $\tilde{V}_\lambda(x,i) = V_\lambda(x,i)$ for $x \in [0,1)$, $i \in [n]$, it follows from above that 
$$\int_0^1 \lambda Q_\lambda \tilde{V}_\lambda(x,i) \, \Pi^+_i(\diff{x}) < \infty.$$
Moreover, for $x \geq 1$ we have $\lambda Q_\lambda \tilde{V}_\lambda(x,i) \leq x^\lambda$ and thus by our moment assumptions on $H^{+,(i)}$ and $\Delta^+_{i,j}$ 
$$\int_1^\infty \lambda Q_\lambda \tilde{V}_\lambda(x,i) \, \Pi^+_i(\diff{x}) < \infty, \quad \E[\lambda Q_\lambda \tilde{V}_\lambda(\Delta^+_{i,j},j)] < \infty.$$
This shows that 
$$\tilde{b} \coloneq \lambda \lVert \bm{\Phi}^+(\lambda)^{-1} \rVert_\infty \sum_{i=1}^n \left(d_i^+ + \int_0^\infty Q_\lambda \tilde{V}_\lambda(x,i)\, \Pi_i^+(\diff{x}) + \sum_{j \neq i} q^+_{i,j} \E[Q_\lambda \tilde{V}_\lambda(\Delta^+_{i,j},j)] \right) < \infty.$$
Observe now that integrating by parts twice yields that for $x \geq 1$ and $i \in [n]$,
$$\lambda Q_\lambda \tilde{V}_\lambda(x,i) \leq \tilde{V}_\lambda(x,i) - x^{\lambda -1} + \mathrm{e}^{-\lambda x}\Big(\mathrm{e}^\lambda+ \int_1^x (\lambda-1) \mathrm{e}^{\lambda t} t^{\lambda -2} \diff{t} \Big)$$
and for $x \in [0,1)$, 
$$\lambda Q_\lambda \tilde{V}_\lambda(x,i) \leq \mathrm{e}^{\lambda x}.$$
Thus, for all $(x,i) \in \R_+ \times [n]$, we have 
$$\lambda Q_\lambda \tilde{V}_\lambda(x,i) \leq \tilde{V}_\lambda(x,i) - (\tilde{V}_\lambda(x,i))^{\frac{\lambda-1}{\lambda}} + \mathrm{e}^{-\lambda x}\Big(\mathrm{e}^\lambda+ \int_1^x (\lambda -1) \mathrm{e}^{\lambda t} t^{\lambda -2} \diff{t} \Big) + \mathrm{e}^{\lambda-1}\one_{[0,1]}(x)$$
and hence by the resolvent formula and the definiton of $\tilde{b}$, 
\begin{equation}\label{eq:subgeo1}
\mathcal{R}_\lambda \tilde{V}_\lambda(x,i) \leq \tilde{V}_\lambda(x,i) - (\tilde{V}_\lambda(x,i))^{\frac{\lambda-1}{\lambda}} + \mathrm{e}^{-\lambda x}\Big(\tilde{b} + \mathrm{e}^\lambda+ \int_1^x (\lambda-1) \mathrm{e}^{\lambda t} t^{\lambda -2} \diff{t} \Big)+ \mathrm{e}^{\lambda-1}\one_{[0,1]}(x).
\end{equation}
Let $x^\ast > 1$ be large enough such that for all $x > x^\ast$
$$\psi_\lambda(x) \coloneq \mathrm{e}^{-\lambda x}\Big(\tilde{b} + \mathrm{e}^{\lambda} + \int_1^x (\lambda-1) \mathrm{e}^{\lambda t} t^{\lambda -2} \diff{t}  \Big) \leq \frac{1}{2} x^{\lambda -1}.$$
By the same arguments as in the previous part, the compact set $C \coloneq [0,x^\ast] \times [n]$ is petite and it follows from \eqref{eq:subgeo1} that 
\begin{equation}\label{eq:subgeo2}
\mathcal{R}_\lambda \tilde{V}_\lambda \leq \tilde{V}_\lambda - \phi \circ \tilde{V}_\lambda + \tilde{c} \one_C,
\end{equation}
where $\tilde{c} \coloneq \mathrm{e}^{\lambda -1} + \max_{x \in [0,x^\ast]} \psi_\lambda(x) < \infty$ and $\phi(z) = \tfrac{1}{2} z^{1-1\slash \lambda}$, $z \geq 1$, is concave, differentiable and increasing. Hence, \eqref{eq:drift subgeo} is satisfied. The assertion now follows from \eqref{eq:subgeo rate} upon noting that 
$$H_{\lambda\phi}(t) = \int_1^t (1/(\lambda\phi)(s)) \diff{s} = 2 (t^{1\slash \lambda}-1), \quad H_{\lambda\phi}^{-1}(t) = \Big(1+ \frac{t}{2}\Big)^\lambda,$$
and therefore the rate of convergence $\Xi(t)$ defined in Appendix \ref{sec: markov stability} is given by 
$$\Xi(t) = 1\slash (\lambda\phi \circ H_{\lambda\phi}^{-1})(t) = \frac{2}{\lambda}\Big(1+ \frac{t}{2}\Big)^{1-\lambda}.$$
\end{enumerate}
\end{proof}
\begin{remark}
Our analysis of the mixing behavior of self-similar Markov processes later on profits from the exact exponential total variation rate, since the Lamperti--Kiu transform turns the exponential rate into a polynomial one.
\end{remark}

As a consequence, we can infer exponential and polynomial $\beta$-mixing rates for the overshoot process.

Recall the definition of the $\beta$-mixing coefficient from \eqref{eq: beta mix coeff2}.
\begin{theorem} \label{theo: mixing}
Suppose that one of the conditions of Proposition \ref{prop: aperiod} is satisfied.
\begin{enumerate}[label=(\roman*),ref=(\roman*)]
\item Suppose that the exponential moment assumption from Theorem \ref{theo: exp ergodicity}.\ref{theo: exp ergodicity 1} is satisfied and let $\eta$ be a probability measure on $(\R_+ \times [n],\cB(\R_+ \times [n]))$ such that $\eta(\cdot,[n])$ has an exponential $\lambda$-moment. Then, for any $\delta \in (0,1)$, $(\cO,\cJ)$ started in $\eta$ is exponentially $\beta$-mixing with the $\beta$-mixing coefficient $\beta(\eta,\cdot)$ satisfying
$$\beta(\eta,t) \leq 2\varrho(\eta,\lambda,\delta) \mathrm{e}^{-\lambda t/(1+\delta)},$$
for
$$\varrho(\eta,\lambda,\delta) \coloneqq C(\lambda,\delta) \sup_{t \geq 0} \E^\eta\big[\mathcal{R}_{2\lambda/\delta}V_\lambda(\cO_t,\cJ_t) \big] < \infty,$$
for some constant $C(\lambda,\delta) > 0$ and $V_\lambda(x,i) = \exp(\lambda x)$, $(x,i) \in \R_+ \times [n]$.\label{theo: mixing 1}
\item Suppose that the $\lambda$-moment assumption from Theorem \ref{theo: exp ergodicity}.\ref{theo: exp ergodicity 2} is satisfied for some $\lambda > 2$. Then, $(\cO_t,\cJ_t)_{t \geq 0}$ started in its invariant distribution is $\beta$-mixing with rate
$$\beta(\rho,t) \lesssim (1+t)^{2- \lambda}, \quad t \geq 0.$$\label{theo: mixing 2}
\end{enumerate}
\end{theorem}
\begin{proof}
\begin{enumerate}[label=(\roman*),ref=(\roman*),leftmargin=0cm,itemindent=.5cm]
\item By Theorem \ref{theo: exp ergodicity}, for any $\alpha > 0$ and $\varepsilon\in (0,\lambda)$ there exists $\mathfrak C(\alpha,\varepsilon)> 0$  s.t.\
$$\lVert \PP^{x,i}(\cO_t \in \cdot) - \rho\rVert_{\mathrm{TV}} \leq \mathfrak C(\alpha,\varepsilon) \mathcal{R}_{\alpha} V_{\lambda }(x,i) \mathrm{e}^{- \alpha(\lambda - \varepsilon)t/(\alpha + \lambda)}, \quad (x,i) \in \R_+ \times [n],$$
which for the choice $\varepsilon = \lambda\delta/(2(1+\delta))$ and $\alpha = 2\lambda /\delta$ becomes 
$$\lVert \PP^{x,i}(\cO_t \in \cdot) - \rho\rVert_{\mathrm{TV}} \leq \mathfrak C(2\lambda /\delta,\lambda\delta/(2(1+\delta))) \mathcal{R}_{2\lambda/\delta} V_{\lambda }(x,i) \mathrm{e}^{- \lambda t/(1+\delta)}, \quad (x,i) \in \R_+ \times [n].$$
Hence, the assertion will follow from Lemma 3.9 in Masuda \cite{masuda2007} if we can establish that $\varrho(\eta,\lambda,\delta) < \infty.$ Either by a direct calculation or setting $t = 0$ in \eqref{eq:fconv} it follows that $\rho(\mathcal{R}_{2\lambda/\delta}V_\lambda) < \infty$. Thus, using triangle inequality and \eqref{eq:fconv} we obtain for any $t \geq 0$
$$\E^{x,i}[\mathcal{R}_{2\lambda/\delta} V_\lambda(\mathcal{O}_t, \mathcal{J}_t)] \leq \rho(\mathcal{R}_{2\lambda/\delta} V_\lambda) + \lVert \PP^{x,i}(\cO_t,\cJ_t) - \rho\rVert_{\mathcal{R}_{2\lambda / \delta} V_\lambda} \leq \rho(\mathcal{R}_{2\lambda/\delta} V_\lambda) + C(2\lambda/\delta)\mathcal{R}_{2\lambda/\delta} V_\lambda(x,i).$$
With \eqref{eq: exp1} it now follows that 
$$\sup_{t \geq 0}\E^\eta[\mathcal{R}_{2\lambda/\delta} V_\lambda(\cO_t,\cJ_t)] \leq \rho(\mathcal{R}_{2\lambda/\delta} V_\lambda) + C(2\lambda/\delta)\big(b(2\lambda/\delta) + 2\eta(V_{\lambda})/(2+\delta) \big) < \infty$$
by assumption on $\eta$. This proves the result.

\item By stationarity, it holds that 
$$\beta(\rho,t) = \int_{\R_+ \times [n]} \lVert \mathcal{P}_{t}((x,z),\cdot) - \rho \rVert_{\mathrm{TV}} \, \rho(\diff{x}\times \diff{z}) = \sum_{i = 1}^n \int_{\R_+} \lVert \mathcal{P}_{t}((x,i),\cdot) - \rho \rVert_{\mathrm{TV}} \, \rho(\diff{x},\{i\}).$$
Since the $(\lambda-1)$th moments of $H^{+,(i)}_1$ for all $i \in [n]$ and $\Delta^+_{i,j}$ for all $i,j \in [n]$ such that $q^+_{i,j} \neq 0$ exist, it follows from Theorem \ref{theo: exp ergodicity}.\ref{theo: exp ergodicity 2} that 
$$\beta(\rho,t) \leq \tilde{C}(1+t)^{2-\lambda}\sum_{i=1}^n \int_{\R_+} \mathcal{R}_{\lambda-1} \tilde{V}_{\lambda - 1}(x,i) \, \rho(\diff{x,\{i\}}),$$
and hence, to prove the assertion it is enough to show that the integrals on the right-hand side are finite.
From the drift inequality \eqref{eq:subgeo2} established in the proof of Theorem \ref{theo: exp ergodicity}.\ref{theo: exp ergodicity 2} we obtain that for any $i \in [n]$,
$$\int_{\R_+ \times [n]} \mathcal{R}_{\lambda -1} \tilde{V}_{\lambda-1}(x,i) \, \rho(\diff{x},\{i\}) \leq \int_0^1 \mathrm{e}^{(\lambda - 1) x} \rho(\diff{x},i) + \tilde{c} \rho(C) + \int_1^\infty x^{\lambda -1} \, \rho(\diff x, \{i\}).$$
Since by our moment assumptions
\begin{equation} \label{eq:polymixing}
\begin{split}
\int_1^\infty x^{\lambda -1}\, \rho(\diff x, \{i\}) &= \frac{1}{\E^{0,\bm{\pi}^+}[H_1^+]} \int_1^\infty x^{\lambda -1}\Big(\pi^+(i)\overbar{\Pi}{}^+_i(x) + \sum_{j \neq i} \pi^+(j)q^+_{j,i} \overbar{F}{}^+_{j,i}(x)\Big)\diff{x},\\
&\leq \frac{1}{\E^{0,\bm{\pi}^+}[H_1^+]} \Big(\pi^+(i)\int_1^\infty x^{\lambda}\,\Pi^+_i(\diff{x}) + \sum_{j \neq i} \pi^+(j)q^+_{j,i}\int_1^\infty  x^\lambda\, F^+_{j,i}(\diff{x})\Big)\\
&< \infty,
\end{split}
\end{equation}
the assertion follows.
\end{enumerate}
\end{proof}
As a direct corollary we obtain the exponential resp.\ polynomial $\beta$-mixing behavior of MAPs sampled at first hitting times provided that creeping is possible or the L\'evy system has some minor regularity and moreover the respective moment conditions on the MAP are satisfied. Let
$$\mathcal{K}_t \coloneq \sigma\big(\big(\xi_{T_s}, J_{T_s}\big), s \leq t\big), \quad \overbar{\mathcal{K}}_t \coloneq \sigma\big(\big(\xi_{T_s}, J_{T_s}\big), s \geq t\big), \quad t \geq 0$$
be the $\sigma$-algebras generated by the MAP sampled at first hitting times up to level $t$ and from level $t$ onwards, respectively.

\begin{corollary}\label{coroll: mixing}
Suppose that the assumptions of Theorem \ref{theo: exp ergodicity}.\ref{theo: exp ergodicity 1} are satisfied and let $\eta$ be a probability measure on $(\R_+ \times [n],\cB(\R_+ \times [n]))$ such that $\eta(\cdot,[n])$ has an exponential $\lambda$-moment. Then, for any $\delta \in (0,1)$,
$$\sup_{ t > 0} \beta_{\PP^{\eta}}\big(\mathcal{K}_t, \overbar{\mathcal{K}}_{t+s} \big) \leq 2\varrho(\eta,\lambda,\delta) \mathrm{e}^{-\lambda s/(1+\delta)}, \quad s > 0,$$
where $\varrho(\eta,\lambda,\delta) > 0$ is the constant from Theorem \ref{theo: mixing}. If instead the assumptions from \ref{theo: exp ergodicity}.\ref{theo: exp ergodicity 2} are satisfied with $\lambda > 2$, then 
$$\sup_{t>0}\beta_{\PP^{\rho}}\big(\mathcal{K}_t, \overbar{\mathcal{K}}_{t+s} \big) \lesssim (1+s)^{2-\lambda}, \quad s > 0.$$
\end{corollary}

\section{Équations amicales invers\'es for MAPs} \label{sec: vigon}
With the help of the spatial Wiener--Hopf factorization for MAPs we can generalize Vigon's équation amicale inversé for L\'evy processes to a characterization of the L\'evy system of the ascending ladder height MAP in terms of the L\'evy system of the parent MAP and the potential measures of the ascending ladder height process of the dual MAP. This is crucial for our results since this relation will allow to impose conditions on the parent MAP instead of the ascending ladder height MAP that imply the overshoot convergence results from the previous section. To this end, we first need to recall some concepts from distribution theory and introduce more notation.

Let $\mathcal{S}(\R)$ be the Schwartz space of rapidly decreasing smooth functions on $\R$ and consider its dual space $\mathcal{S}^\prime(\R)$, the space of tempered distributions. For $\mu \in \mathcal{S}^\prime(\R)$ the $k$-th derivative $\mu^{(k)} \in \mathcal{S}^\prime(\R)$ is defined by
$$\big\langle \mu^{(k)},\phi\big\rangle =(-1)^k \big\langle\mu,\phi^{(k)}\big\rangle, \quad \phi \in \mathcal{S}(\R), k \in \N.$$
If $\mu$ is induced by a function $\psi \in \cB(\R)$ via
$$\langle\mu,\phi\rangle = \int_{\R} \psi(x) \phi(x) \diff{x}, \quad \phi \in \mathcal{S}(\R),$$
we just write $\mu = \psi$, provided that the above integrals are well defined. Similarly, if $\mu$ is a measure on $(\R,\cB(\R))$ such that $\int \phi \diff{\mu}$ is well-defined for any $\phi \in \mathcal{S}(\R)$, we identify the distribution induced by $\phi \mapsto \int \phi \diff{\mu}$ with $\mu$.

For a L\'evy measure $\nu$ integrating $x \mapsto \lvert x \rvert$ on $[-1,1]$, let $\bbGamma\nu$ be the tempered distribution defined via
$$\big\langle\bbGamma\nu,\phi\big\rangle \coloneq \int_{\R} (\phi(x)  -  \phi(0))\, \nu(\diff{x}),\quad \phi\in\mathcal{S}(\R),$$
and for a general L\'evy measure $\nu$ let $\bbGamma^2 \nu$ be the tempered distribution defined via
$$\big\langle\bbGamma^2\nu,\phi\big\rangle \coloneq \int_{\R} (\phi(x)  -  \phi(0)  - \phi^\prime(0)x \one_{[-1,1]}(x))\, \nu(\diff{x}),\quad \phi\in\mathcal{S}(\R).$$
Recall that for a tempered distribution $\mu \in \mathcal{S}^\prime(\R)$ the Fourier transform $\mathscr{F} \mu$ is defined by
$$\langle\mathscr{F}\mu, \phi\rangle \coloneq \langle\mu, \mathscr{F} \phi\rangle = \Big\langle\mu,\int_{\R} \mathrm{e}^{\mathrm{i} x \cdot} \phi(x) \diff{x} \Big\rangle, \quad \phi \in \mathcal{S}(\R),$$
and that the Fourier transform is a bijective, continuous mapping on $\mathcal{S}^\prime(\R)$. If $\delta$ is the Dirac delta distribution and letting $\psi_2(x) = x^2$, $x \in \R$, it is immediate that
$$\mathscr{F} \delta = \mathrm{id}, \quad \mathscr{F} \delta^\prime = -\mathrm{i}\cdot\mathrm{id}, \quad \mathscr{F} \delta^{\prime\prime} = -\psi_2.$$
Hence, for a L\'evy subordinator with characteristic Fourier exponent  $\kappa$ , L\'evy measure $\nu$, drift $d \geq 0$ and killing rate $q \geq 0$ we obtain
$$\big\langle\mathscr{F} \big(-q \delta -d \delta^\prime + \bbGamma \nu\big),\phi\big\rangle = \int_{\R} \Big(-q + \mathrm{i}d\theta + \int_{\R} \big(\mathrm{e}^{\mathrm{i}\theta x} -1\big)\, \nu(\diff{x})\Big) \phi(\theta) \diff{\theta} = \int_{\R} \kappa(\theta) \phi(\theta) \diff{\theta},$$
and therefore it holds that
$$\mathscr{F}^{-1} \kappa = -q\delta -d \delta^\prime + \bbGamma \nu,$$
i.e.\ if $\mathcal{A}^\ast$ denotes the infinitesimal generator of the subordinator's dual, then 
$$\mathcal{A}^\ast f = (\mathscr{F}^{-1}\kappa) \ast f, \quad f \in \mathcal{S}(\R).$$
Similarly, we get for the characteristic exponent $\Psi$ of a L\'evy process with generating triplet $(a,\sigma^2,\nu)$ and killing rate $q$ that
$$\mathscr{F}^{-1} \Psi = -q\delta -a \delta^\prime + \frac{1}{2}\sigma^2 \delta^{\prime \prime} + \bbGamma^2 \nu.$$
We start with a simple lemma. Let
$$\sigma(A) \coloneq \sup\{\Re(\lambda) : \lambda \text{ eigenvalue of } A\},$$
be the spectral bound of a quadratic complex matrix $A$.

\begin{lemma}\label{lemma: spectral bound}
For any (non-trivial) MAP with characteristic matrix exponent $\bm{\Psi}$ and $\theta \in \R$, it holds that $\sigma(\bm{\Psi}(\theta)) \leq 0$ and for any $\lambda > 0$, $\lambda \mathbb{I}_n - \bm{\Psi}(\theta)$ is invertible.
\end{lemma}
\begin{proof}
Let $\lambda > 0$ be arbitrary and $\mathbf{e}_\lambda$ be an independent exponential time with mean $1\slash \lambda$ and define for $x \in \R$, $i,j \in [n]$,
\begin{align*}
^\lambda{U}_{i,j}(\diff{x}) = \E^{0,i}\Big[\int_0^{\mathbf{e}_\lambda} \one_{\{\xi_t \in \diff{x}, J_t = j\}}\diff{t}  \Big] &= \int_0^\infty \PP^{0,i}(\xi_t \in \diff{x}, J_t = j, t < \mathbf{e}_\lambda) \diff{t}\\
&= \int_0^\infty \mathrm{e}^{-\lambda t} \PP^{0,i}(\xi_t \in \diff{x}, J_t = j) \diff{t},
\end{align*}
i.e.\, $^\lambda{U}_{i,j}$ is the (finite) occupation measure of the MAP started in $(0,i)$,  while the modulator $J$ is in state $j$, killed at an independent exponential time. Clearly,
$$\big\{\mathscr{F} {^\lambda U}_{i,j}\big\}(\theta) = \Big(\int_0^\infty \mathrm{e}^{t(\bm{\Psi}(\theta)- \lambda \mathbb{I}_n)} \diff{t} \Big)_{i,j},$$
where for a matrix valued function $f\colon \R \to \R^{n \times n},$ such that $f_{i,j}$ is integrable, $\int_{\R} f(t) \diff{t} \coloneqq (\int_{\R} f_{i,j}(t) \diff{t})_{i,j=1,\ldots,n}$. Hence, if we let $^\lambda U \coloneq (^\lambda{U}_{i,j})_{i,j\in [n]}$, it follows that
$$\big\{\mathscr{F} \,{^\lambda U}\big\}(\theta) = \int_0^\infty \mathrm{e}^{t (\bm{\Psi}(\theta)- \lambda \mathbb{I}_n)} \diff{t}.$$
Noting that
\begin{equation}\label{eq: spectral bound}
(\lambda \mathbb{I}_n - \bm{\Psi}(\theta)) \int_0^T \mathrm{e}^{t (\bm{\Psi}(\theta)- \lambda \mathbb{I}_n)}\diff{t} = \mathbb{I}_n - \mathrm{e}^{T(\bm{\Psi}(\theta) - \lambda \mathbb{I}_n)},
\end{equation}
and that the left-hand side converges to
$$(\lambda \mathbb{I}_n - \bm{\Psi}(\theta)) \cdot \big\{\mathscr{F}\, {^\lambda U}\big\}(\theta),$$
as $T \to \infty$, it follows that the matrix exponential $\mathrm{e}^{T(\bm{\Psi}(\theta) - \lambda \mathbb{I}_n)}$ must converge as well as $T \to \infty$. E.g.\ from Theorem 4.12 of \cite{batkai2017}, this can only be the case if $\sigma(\bm{\Psi}(\theta) - \lambda \mathbb{I}_n) \leq 0$. But since $\lambda > 0$ was chosen arbitrarily, it follows that for any $\lambda > 0$, actually $\sigma(\bm{\Psi}(\theta) - \lambda \mathbb{I}_n) < 0$, implying $\sigma(\bm{\Psi}(\theta)) \leq 0$. Again by Theorem 4.12 of \cite{batkai2017}, this implies that
$$\lim_{T \to \infty} \mathrm{e}^{T(\bm{\Psi}(\theta) - \lambda \mathbb{I}_n)} = \mathbf{0}_{n \times n}.$$
Thus, \eqref{eq: spectral bound} yields that $\lambda \mathbb{I}_n - \bm{\Psi}(\theta)$ is invertible with inverse
\begin{equation}\label{eq: inverse char}
(\lambda \mathbb{I}_n - \bm{\Psi}(\theta))^{-1} = \big\{\mathscr{F}\, {^\lambda U}\big\}(\theta).
\end{equation}
\end{proof}
\begin{remark}
This result generalizes part of Theorem 1 in \cite{ivanovs2010} in the sense that, if we let $\Upsilon(z) = (\E^{0,i}[\mathrm{exp}(z\xi_1); J_1 = j])_{i,j\in [n]}$ for $z \in \mathbb{C}$ whenever it is defined, $z \mapsto \det(\Upsilon(z)- \lambda \mathbb{I}_n)$ has no zeros on the imaginary axis, without having to assume anything on the jump structure of $(\xi,J)$ or irreducibility of $J$.
\end{remark}

Let us assume for the rest of this section that
\begin{enumerate}[leftmargin=*,label=($\mathscr{A}3$),ref=($\mathscr{A}3$)]
\item the modulator $J$ of the MAP $(\xi,J)$ is irreducible, i.e.\ $\bm{Q}$ is an irreducible matrix.
\end{enumerate}

\begin{theorem}[Équations amicales inversés for MAPs]\label{theo: vigon}
For an appropriate scaling of local time at the supremum it holds for any $i,j \in [n]$, $i \neq j$ and $x>0$ that
\begin{align}
\Pi_i^+(\diff{x}) &= \int_{0}^\infty \Pi_i(y + \diff{x})\, \hat{U}^+_{i,i}(\diff{y}) + \sum_{k \neq i} \frac{\pi(k)}{\pi(i)} q_{k,i} \int_{0}^\infty F_{k,i}(y + \diff{x})\, \hat{U}^+_{k,i}(\diff{y}),\label{eq: vigon inverse1}\\
q^+_{i,j} F^+_{i,j}(\diff{x}) &= \frac{\pi(j)}{\pi(i)} \int_{0}^\infty \Pi_j(y + \diff{x})\, \hat{U}^+_{j,i}(\diff{y}) + \sum_{k \neq j} \frac{\pi(k)}{\pi(i)} q_{k,j} \int_{0}^\infty F_{k,j}(y + \diff{x})\, \hat{U}^+_{k,i}(\diff{y}).\label{eq: vigon inverse2}
\end{align}
and 
\begin{align}
\hat{\Pi}{}^+_i(\diff{x}) &= \int_{0}^\infty \Pi_i(-y - \diff{x})\, U^+_{i,i}(\diff{y}) + \sum_{k \neq i} q_{i,k} \int_{0}^\infty F_{i,k}(-y -\diff{x})\, U^+_{k,i}(\diff{y}),\label{eq: vigon inverse3}\\
\hat{q}{}^+_{i,j} \hat{F}{}^+_{i,j}(\diff{x}) &= \frac{\pi(j)}{\pi(i)}\Big( \int_{0}^\infty \Pi_j(-y - \diff{x})\, U^+_{j,i}(\diff{y}) + \sum_{k \neq j}  q_{j,k} \int_{0}^\infty F_{j,k}(-y - \diff{x})\, U^+_{k,i}(\diff{y})\Big).\label{eq: vigon inverse4}
\end{align}
\end{theorem}
\begin{remark}
If we let $\bm{\Pi}(\diff{x}) \coloneqq (\bm{\Pi}(i, \diff{x} \times \{j\}))_{i,j = 1,\ldots,n}$, $\bm{\Pi}^+(\diff{x}) \coloneqq (\bm{\Pi}^+(i, \diff{x} \times \{j\}))_{i,j = 1,\ldots,n}$ and $\bm{U}{}^+(\diff{x}) \coloneqq (U^+_{i,j}(\diff{x}))_{i,j=1,\ldots,n}$ (with the analogous definitions for the ascending ladder height process of the dual MAP), then we may compactly express the équations amicales inversés (up to premultiplication of some diagonal matrix corresponding to the scaling of local time at the supremum) for $x > 0$ as 
\begin{align*}
\bm{\Pi}^+(\diff{x}) &= \int_0^\infty \bm{\Delta}_{\bm{\pi}}^{-1}  \hat{\bm{U}}{}^+(\diff{y})^\top \bm{\Delta}_{\bm{\pi}}\, \bm{\Pi}(y + \diff{x}),\\
\hat{\bm{\Pi}}{}^+(\diff{x}) &= \int_0^\infty \bm{\Delta}_{\bm{\pi}}^{-1}  \big(\bm{\Pi}(-y - \diff{x})\,\bm{U}{}^+(\diff{y})\big)^\top \bm{\Delta}_{\bm{\pi}},
\end{align*}
where $\int_0^\infty \bm{A}(\diff{y}) \, \bm{B}(y+\diff{x}) \coloneqq (\sum_{k=1}^n \int_0^\infty a_{i,k}(\diff y) \, b_{k,j}(y + \diff{x}))_{i,j = 1,\ldots, n}$ for measure matrices $\bm{A}(\diff{y}) = (a_{i,j}(\diff{y}))_{i,j = 1, \ldots, n}$ and $\bm{B}(\diff{y}) = (b_{i,j}(\diff{y}))_{i,j = 1,\ldots,n}$.
\end{remark}
\begin{proof}[Proof of Theorem \ref{theo: vigon}]
Analogously to Vigon's \cite{vigon2002} idea, we use inverse Fourier transformations of the quantities involved in the spatial Wiener--Hopf factorization for MAPs to prove the desired equalitites. To this end, recall from \eqref{eq: wiener-hopf} that for an appropriate scaling of local time at the supremum, it holds that
\begin{equation}\label{eq: vigon1}
\bm{\Psi}(\theta) = -\bm{\Delta_\pi}^{-1} \bm{\hat{\Psi}}{}^+(-\theta)^\top\bm{\Delta_\pi} \bm{\Psi}^+(\theta), \quad \theta \in \R.
\end{equation}
Rearranging yields for any $\lambda > 0$,
\begin{equation} \label{eq: vigon2}
\bm{\Psi}^+(\theta) =  -\bm{\Delta_\pi}^{-1} \Big(\Big(\bm{\hat{\Psi}}{}^+(-\theta) - \lambda \mathbb{I}_n\Big)^{-1}\Big)^\top \bm{\Delta_\pi} \big(\bm{\Psi}(\theta) + \lambda \bm{\Psi}^+(\theta)\big), \quad \theta \in \R,
\end{equation}
where invertibility of $\bm{\hat{\Psi}}{}^+(-\theta) - \lambda \mathbb{I}_n$ is shown in Lemma \ref{lemma: spectral bound}. By the form of the characteristic matrix exponent of a MAP it follows by taking inverse Fourier transformation of the distribution induced by the left-hand side that
$$\mathscr{F}^{-1} \bm{\Psi}^+_{i,j} = \one_{\{i = j\}}\big((q^+_{i,i}-\dagger_i^+)\delta - d_i^+ \delta^\prime + \bbGamma \Pi_i^+ \big) + \one_{\{i \neq j\}} q^+_{i,j} F^+_{i,j}.$$
Note that by \eqref{eq: inverse char}
$$\Big(\lambda \mathbb{I}_n - \bm{\hat{\Psi}}{}^+(-\cdot)\Big)^{-1}_{i,j} = \mathscr{F}\, ^\lambda\tilde{U}^+_{i,j},$$
where for an independent exponentially distributed random variable $\mathbf{e}_\lambda$ with mean $1\slash \lambda$ we define
$$^\lambda\tilde{U}^+_{i,j}(\diff{x}) \coloneqq \hat{\E}^{0,i} \Big[\int_0^{\mathbf{e}_\lambda} \one_{\{-H^+_t \in \diff{x}, J^+_t = j\}} \diff{t}\Big], \quad x \in \R.$$
With this observation, our previous discussion of inverse Fourier transforms of L\'evy characteristic exponents and the property that if we regard two tempered distributions whose Fourier transforms are induced by some measurable functions, the Fourier transform of the convolution of those distributions becomes the tempered distribution induced by the product of the functions, it follows that the inverse Fourier transformation of the distribution induced by the right-hand side of \eqref{eq: vigon2} may be written as
\begin{align*}
-\mathscr{F}^{-1}&\Big(\bm{\Delta_\pi}^{-1} \Big(\Big(\bm{\hat{\Psi}}{}^+(-\cdot) - \lambda \mathbb{I}_n\Big)^{-1}\Big)^\top \bm{\Delta_\pi} \big(\bm{\Psi} + \lambda \bm{\Psi}^+\big) \Big)_{i,j}\\
&= -\sum_{k=1}^n \frac{\pi(k)}{\pi(i)}\, \mathscr{F}^{-1} \Big(\Big(\bm{\Psi} + \lambda \bm{\Psi}^+ \Big)_{k,j}\Big(\bm{\hat{\Psi}}^+(-\cdot) - \lambda \mathbb{I}_n\Big)^{-1}_{k,i} \Big)\\
&= \frac{\pi(j)}{\pi(i)} \Big(\bbGamma^2 \Pi_j + \lambda \bbGamma \Pi^+_j - (a_j+ \lambda d_j^+) \delta^\prime + \tfrac{1}{2}\sigma^2 \delta^{\prime\prime} + (q_{j,j} + \lambda (q^+_{j,j} - \dagger^+_j)) \delta\Big) \ast {^\lambda\tilde{U}^+_{j,i}}\\
& \quad + \sum_{k \neq j} \frac{\pi(k)}{\pi(i)} \big(q_{k,j} F_{k,j} + \lambda q^+_{k,j} F^+_{k,j}\big) \ast {^\lambda \tilde{U}^+_{k,i}}.
\end{align*}
Observe that the restriction of $\bbGamma \Pi^+_j$ and $\bbGamma^2 \Pi_j$ to the space $\mathcal{D}_+$ of smooth functions on $\R$ with compact support in $(0,\infty)$ is equal to the distributions induced by $\Pi^+_j$ and $\Pi_j$ on this space, see also Propriété 3.9 in \cite{vigon2002}.  Restricting to $(0,\infty)$ and equating both sides therefore yields the equality of distributions on $\mathcal{D}^\prime_+$,
\begin{equation}\label{eq: vigon3}
\begin{split}
&\one_{\{i = j\}} \Pi_i^+ + \one_{\{i \neq j\}} q^+_{i,j} F^+_{i,j}\\
&\quad= \frac{\pi(j)}{\pi(i)} \big(\Pi_j\vert_{(0,\infty)} + \lambda \Pi^+_j\vert_{(0,\infty)}\big) \ast {^\lambda \tilde{U}^+_{j,i}} + \sum_{k \neq j} \frac{\pi(k)}{\pi(i)} \big(q_{k,j} F_{k,j} \vert_{(0,\infty)} + \lambda q^+_{k,j} F^+_{k,j}\vert_{(0,\infty)}\big) \ast {^\lambda \tilde{U}^+_{k,i}}.
\end{split}
\end{equation}
Here, we identified the measures with their restriction to $(0,\infty)$ and used that for a measure $\mu$ on $\R$ such that the distribution $\mu \ast {^\lambda \tilde{U}^+_{k,i}}$, is well-defined it holds that
$$\big(\mu\ast {^\lambda \tilde{U}^+_{k,i}}\big)\big\vert_{(0,\infty)} = \big(\mu\big\vert_{(0,\infty)} \ast {^\lambda \tilde{U}^+_{k,i}}\big)\big\vert_{(0,\infty)},$$
since $^\lambda \tilde{U}^+_{k,i}$ has support $\R_{-}$, see also Propriété 3.8 in \cite{vigon2002}. Denote  $$\tilde{U}^+_{k,i}(\diff{x}) \coloneqq \hat{\E}^{0,k} \Big[\int_0^\infty \one_{\{-H^+_t \in \diff{x}, J^+_t = i\}} \diff{t}\Big], \quad x \in \R,$$
and let $\phi \in \mathcal{B}_b((0,\infty))$ have support $\mathrm{supp}(\phi) \subset (a,b)$, where $0<a<b<\infty$. Utilizing the strong Markov property and spatial homogeneity of the dual ascending ladder height process we can calculate as follows for any measure $\mu$ such that $\mu(z,\infty) < \infty$ for all $z > 0$,
\begin{align*}
\Big\vert \int_0^\infty \phi(y)\, \mu\vert_{(0,\infty)} \ast \tilde{U}^+_{j,i}(\diff{y}) \Big\vert &= \Big\vert \int_{-\infty}^0 \int_{(a,b)} \phi(z)\, \mu(\diff{z}-y) \, \tilde{U}^+_{j,i}(\diff{y}) \Big\vert\\
&= \Big\vert \hat{\E}^{0,j}\Big[\int_0^\infty \int_{(a+H^+_t,b+H^+_t)} \phi(z-H^+_t)\, \mu(\diff{z}) \one_{\{J^+_t = i\}} \diff{t}  \Big] \Big\vert\\
&\leq \lVert \phi\rVert_\infty \int_{(a,\infty)} \hat{\E}^{0,j} \Big[\int_{(T^+_{z-b},T^+_{z-a})} \one_{\{J^+_t = i \}} \diff{t} \Big] \,\mu(\diff{z})\\
&=\lVert \phi\rVert_\infty \int_{(a,\infty)} \hat{\E}^{0,j} \Big[\hat{\E}^{H^+_{T^+_{z-b}},J^+_{T^+_{z-b}}} \Big[\int_{(0,T^+_{z-a})} \one_{\{J^+_t = i\}} \diff{t}\Big] \Big] \,\mu(\diff{z})\\
&\leq \lVert \phi\rVert_\infty \int_{(a,\infty)} \hat{\E}^{0,j} \Big[\hat{\E}^{0,J^+_{T^+_{z-b}}} \Big[\int_{(0,T^+_{b-a})} \one_{\{J^+_t = i\}} \diff{t} \Big]\Big] \,\mu(\diff{z})\\
&\leq \lVert \phi \rVert_\infty \mu((a,\infty))\sum_{k=1}^n \hat{U}^+_{k,i}(b-a)\\
&< \infty.
\end{align*}
This implies that $\mu\vert_{(0,\infty)} \ast \tilde{U}^+_{j,i} \in \mathcal{D}^\prime_+$
and dominated convergence yields for any $\phi \in \mathcal{D}_+$ that $\langle \mu\vert_{(0,\infty)} \ast {^\lambda \tilde{U}^+_{j,i}},\phi\rangle \to \langle \mu\vert_{(0,\infty)} \ast \tilde{U}^+_{j,i},\phi\rangle$ as $\lambda \downarrow 0$, such that  $\mu \vert_{(0,\infty)}\ast {^\lambda \tilde{U}^+_{j,i}} \to \mu \vert_{(0,\infty)}\ast {\tilde{U}^+_{j,i}}$ on $\mathcal{D}_+^\prime$ as $\lambda \downarrow 0$ follows. Consequently, letting $\lambda \downarrow 0$,  \eqref{eq: vigon3} implies that we have
$$\one_{\{i = j\}} \Pi_i^+ + \one_{\{i \neq j\}} q^+_{i,j} F^+_{i,j} = \frac{\pi(j)}{\pi(i)}\, \Pi_j \ast \tilde{U}^+_{j,i} + \sum_{k \neq j} \frac{\pi(k)}{\pi(i)} q_{k,j}\, F_{k,j} \ast \tilde{U}^+_{k,i}$$
in $\mathcal{D}^\prime_+$. The relations \eqref{eq: vigon inverse1} and \eqref{eq: vigon inverse2} follow upon noting that by a monotone class argument measures with support on $(0,\infty)$ are uniquely characterized by their action on $\mathcal{D}_+$ and observing that $\hat{U}^+_{i,j}(\diff{y}) = \tilde{U}^+_{i,j}(-\diff{y})$ for $y \geq 0$. Relations \eqref{eq: vigon inverse3} and \eqref{eq: vigon inverse4} are proved similarly by taking inverse Fourier transforms on both sides of of
$$\hat{\bm{\Psi}}{}^+ = -\bm{\Delta_\pi}^{-1} \Big(\big(\bm{\Psi}{}^+(-\cdot) - \lambda \mathbb{I}_n\big)^{-1}\Big)^\top \big(\bm{\Psi}(-\cdot)^\top + \lambda \bm{\Delta_{\pi}}\hat{\bm{\Psi}}{}^+\bm{\Delta_\pi}^{-1}\big) \bm{\Delta_\pi}, \quad \lambda > 0,$$
which is a rearranged version of \eqref{eq: vigon1}.
\end{proof}

Without loss of  generality, for the remainder of this section we fix a scaling of local time at the supremum such that \eqref{eq: wiener-hopf} is satisfied and hence the formulas given in Theorem \ref{theo: vigon} hold without further multiplicative constants. 
As a first consequence of the équations amicales inversés, we obtain a characterization of $\bm{Q}^+$ in terms of the transitional jumps of $(\xi,J)$, which we made use of in Proposition \ref{prop: irr}.

\begin{lemma} \label{lem: irr}
Suppose that for $i,j \in [n]$ with $i \neq j$, we have $\supp(q_{i,j}F_{i,j}) \cap (0,\infty)  \neq \varnothing.$ Then, $q^+_{i,j} > 0.$ 
\end{lemma}
\begin{proof}
By assumption, there exists $\varepsilon > 0$ such that $q_{i,j} \overline{F}_{i,j}(z) > 0$ for all $z \in (0,\varepsilon)$. Note also that $\hat{U}^+_{i,i}([0,\varepsilon)) > 0$ by increasing and right-continuous paths of $(H^+,J^+)$ under $\hat{\PP}^{0,i}$. Plugging $(0,\infty)$ into \eqref{eq: vigon inverse2} therefore yields 
$$q^+_{i,j}\geq \int_0^\infty q_{i,j}\overbar{F}_{i,j}(z)\, \hat{U}^+_{i,i}(\diff{z}) \geq \int_0^{\varepsilon} q_{i,j}\overbar{F}_{i,j}(z)\, \hat{U}^+_{i,i}(\diff{z}) >  0.$$
\end{proof}

Another simple consequence is the following.
\begin{lemma}
If for some $j \in [n]$, $\xi^{(j)}$ has infinite jump activity on $\R_+$, i.e.\ $\Pi_j(\R_+) = \infty$, then $\hat{U}{}^+_{j,i}$ does not have an atom at $0$ for all $i \neq j$.
\end{lemma}
\begin{proof}
Suppose that there exists $i \neq j$ s.t.\ $\hat{U}{}^+_{j,i}(\{0\}) = \alpha > 0$. Then, again plugging $(0,\infty)$ into \eqref{eq: vigon inverse2}, implies
$$q^+_{i,j} \geq \frac{\pi(j)}{\pi(i)} \alpha \Pi_j(\R_+) = \infty,$$
which is impossible.
\end{proof}

Let us now use the équations amicales inversés to express our assumptions from Section \ref{sec: stability} on the ascending ladder height process $(H^+,J^+)$ that were needed for ergodicity in terms of conditions on $(\xi,J)$. More precisely, we verify the conditions on the smoothness of the L\'evy system required in Proposition \ref{prop: aperiod} and the moment assumptions on the underlying L\'evy processes and the transitional jumps required in Theorem \ref{theo: exp ergodicity} for exponential or polynomial ergodicity of overshoots, in terms of related conditions on the parent MAP $(\xi,J)$. 

\begin{lemma} \label{lemma: cond vigon}
\begin{enumerate}[label=(\roman*),ref=(\roman*)]
\item If there exists $i \in [n]$ and  $0\leq a <b \leq \infty$ such that $\mathrm{Leb}\vert_{(a,b)} \ll \Pi_i\vert_{(a,b)}$, then also $\mathrm{Leb}\vert_{(a,b)} \ll \Pi^+_i\vert_{(a,b)}$.
\item If there exists $i,j \in [n]$ with $i \neq j$ and $0\leq a < b \leq \infty$ such that $\mathrm{Leb}\vert_{(a,b)} \ll q_{i,j} F_{i,j}\vert_{(a,b)},$ then also $\mathrm{Leb}\vert_{(a,b)} \ll q^+_{i,j}F^+_{i,j}\vert_{(a,b)}$.
\item For fixed $i \in [n]$, $\E[\exp(\lambda H_1{}^{+,(i)})] < \infty$  if 
\begin{equation} \label{eq: exp suff1}
\int_1^\infty \mathrm{e}^{\lambda x}\, \Pi_i(\diff{x}) + \sum_{k \neq i} q_{k,i} \int_1^\infty \mathrm{e}^{\lambda x}\, F_{k,i}(\diff{x}) < \infty.
\end{equation}
\item For fixed $i,j \in [n]$ such that $q^+_{j,i} \neq 0$ and $\lambda > 0$, $\E[\exp(\lambda \Delta^+_{j,i})] < \infty$ if \eqref{eq: exp suff1}  holds.
\item If $\lim_{t \to \infty} \xi_t = \infty$ a.s., then for $\lambda > 0$ and $i \in [n]$, $\E[ (H_1^{+,(i)})^\lambda] < \infty$ if
$$\int_1^\infty x^\lambda\, \Pi_i(\diff{x}) + \sum_{k \neq i} q_{k,i} \int_1^\infty x^\lambda \, F_{k,i}(\diff{x}) < \infty,$$
and for $i,j \in [n]$ such that $q^+_{i,j} \neq 0$, $\E[(\Delta^+_{i,j})^\lambda] < \infty$ if
$$\int_1^\infty x^\lambda\, \Pi_j(\diff{x}) + \sum_{k \neq i} q_{k,j} \int_1^\infty x^\lambda \, F_{k,j}(\diff{x}) < \infty.$$ \label{cond vigon5}
\end{enumerate}
\end{lemma}
\begin{proof}\quad
\begin{enumerate}[label=(\roman*),ref=(\roman*)]
\item  Let $B \subset (a,b)$ be a Borel set s.t.\ $\mathrm{Leb}(B) > 0$. We may assume that $ \sup B < b$ and hence $B + z \subset (a,b)$ for all $z \in (0,b-\sup B)$. By translation invariance of the Lebesgue measure, we have $\mathrm{Leb}(B + z) > 0$ and therefore by assumption $\Pi_i(B+z) > 0$ for all $z \in (0,b-\sup B)$. From \eqref{eq: vigon inverse1} it follows
$$\Pi^+_i(B) \geq \int_{0}^\infty \hat{U}{}^+_{i,i}(\diff{z}) \, \Pi_i(B+z) \geq \int_{0}^{b - \sup B} \hat{U}{}^+_{i,i}(\diff{z}) \, \Pi_i(B+z)$$
and since $\hat{U}{}^+_{i,i}([0, b -\sup B))  > 0$ by increasing and right-continuous paths of the ascending ladder process of the dual process of $(\xi,J)$, it follows $\Pi^+_{i}(B) > 0$, implying $\mathrm{Leb} \vert_{(a,b)} \ll \Pi^+_{i} \vert_{(a,b)}$.\label{proof: abs cont}
\item This is immediate from \eqref{eq: vigon inverse2} in Theorem \ref{theo: vigon} and the same arguments as in part \ref{proof: abs cont}.
\item Since $J$ is irreducible, it follows from the proof of the Wiener--Hopf factorization in Theorem 26 of \cite{dereich2017} that $\hat{\bm{\Phi}}{}^+$ is invertible and hence, for any $i,j \in [n]$ we have
$$\int_0^\infty \mathrm{e}^{-\lambda y} \, \hat{U}{}^+_{i,j}(\diff{y}) = \big(\bm{\hat{\Phi}}{}^+(\lambda)^{-1}\big)_{i,j}.$$
Thus, with Fubini and \eqref{eq: vigon inverse1}
\begin{align*} 
&\int_1^\infty \mathrm{e}^{\lambda x} \, \Pi^+_i(\diff{x})\\
&\quad = \int_0^\infty \int_1^\infty \mathrm{e}^{\lambda x}\, \Pi_i(y +\diff{x}) \, \hat{U}{}^+_{i,i}(\diff{y}) + \sum_{k \neq i} \frac{\pi(k)}{\pi(i)} q_{k,i} \int_0^\infty \int_1^\infty \mathrm{e}^{\lambda x}\, F_{k,i}(y +\diff{x}) \, \hat{U}{}^+_{k,i}(\diff{y})\\
&\quad = \int_0^\infty \int_{1+y}^\infty \mathrm{e}^{\lambda x}\, \Pi_i(\diff{x})  \, \mathrm{e}^{-\lambda y} \hat{U}{}^+_{i,i}(\diff{y}) + \sum_{k \neq i} \frac{\pi(k)}{\pi(i)} q_{k,i} \int_0^\infty \int_{1+y}^\infty \mathrm{e}^{\lambda x}\, F_{k,i}(\diff{x}) \, \mathrm{e}^{-\lambda y} \hat{U}{}^+_{k,i}(\diff{y}),\\
&\quad \leq  \int_{1}^\infty \mathrm{e}^{\lambda x}\, \Pi_i(\diff{x}) \int_0^\infty \mathrm{e}^{-\lambda y} \hat{U}{}^+_{i,i}(\diff{y}) + \sum_{k \neq i} \frac{\pi(k)}{\pi(i)} q_{k,i} \int_1^\infty \mathrm{e}^{\lambda x}\, F_{k,i}(\diff{x}) \int_0^\infty \mathrm{e}^{-\lambda y} \hat{U}{}^+_{k,i}(\diff{y})\\
&\quad= \int_1^\infty \mathrm{e}^{\lambda x} \, \Pi_i(\diff{x})\big(\bm{\hat{\Phi}}{}^+(\lambda)^{-1}\big)_{i,i} + \sum_{k \neq i} \frac{\pi(k)}{\pi(i)} q_{k,i} \int_1^\infty \mathrm{e}^{\lambda x} F_{k,i}(\diff{x}) \big(\bm{\hat{\Phi}}{}^+(\lambda)^{-1}\big)_{k,i},
\end{align*}
which is finite given the assumption.
\label{proof: iff exp}
\item Analogously to \ref{proof: iff exp}.
\item Under the assumption $\lim_{t \to \infty} \xi_t = \infty$ a.s., the ascending ladder height process of the dual of $(\xi,J)$ is killed a.s.\ and hence for any $i,j\in [n]$, $\hat{U}^+_{i,j}$ is a finite measure. Thus, again by \eqref{eq: vigon inverse1}, \eqref{eq: vigon inverse2} and a change of variables,
$$\int_1^\infty x^\lambda \Pi^+_i(\diff{x}) \leq \hat{U}^+_{i,i}(\R_+)\int_1^\infty x^\lambda \, \Pi_i(\diff{x}) + \sum_{k \neq i} \frac{\pi(k)}{\pi(i)}q_{k,i} \hat{U}^+_{k,i}(\R_+) \int_1^\infty x^\lambda F_{k,i}(\diff{x}) < \infty$$
and 
$$q^+_{i,j}\int_1^\infty x^\lambda F^+_{i,j}(\diff{x}) \leq \hat{U}^+_{j,i}(\R_+)\frac{\pi(j)}{\pi(i)}\int_1^\infty x^\lambda \, \Pi_j(\diff{x}) + \sum_{k \neq j} \frac{\pi(k)}{\pi(i)}q_{k,j} \hat{U}^+_{k,i}(\R_+) \int_1^\infty x^\lambda F_{k,j}(\diff{x}) < \infty.$$
\end{enumerate}
\end{proof}
\begin{remark}
\begin{enumerate}[label=(\roman*),ref=(\roman*)]
\item Conditions \eqref{eq: exp suff1}  are sufficient but not necessary conditions for exponential moments of the components of the L\'evy system $\bm{\Pi}^+$, since $\hat{U}^+_{k,i}$ is trivial for some $k \neq i$ whenever $J^+$ is not irreducible under $(\hat{\PP}^{0,i})_{i \in [n]}$. It is however true that if $\E[\exp(\lambda H^{+,(i)}_1)] < \infty$, we must necessarily have $\int_1^\infty \mathrm{e}^{\lambda x} \,\Pi_i(\diff{x}) < \infty$ and if $\E[\exp(\lambda \Delta^{+}_{i,j})] < \infty$, it must hold $\int_1^\infty \mathrm{e}^{\lambda x} \,F_{i,j}(\diff{x}) < \infty$, since the on-diagonal potential measures $\hat{U}^+_{i,i}$ are non-trivial.
\item We restrict to the case $\lim_{t \to \infty} \xi_t = \infty$ a.s.\ in \ref{cond vigon5}. The oscillatory case $\limsup_{t \to \infty} \xi_t = -\liminf_{t \to \infty} \xi_t = \infty$ a.s.\ is more difficult to handle since in this case $(H^+,J^+)$ is unkilled under the dual measures $\hat{\PP}^{0,i}$ and we have no control over $\hat{\bm{U}}^+$ solely in terms of the characteristics of $(\xi,J)$. In \cite{dereich2017} the authors establish the necessary and sufficient integral criterion given in \eqref{eq: ladder height finite mean} for finiteness of the first moment of $H^+_1$ in the oscillatory regime by taking a detour via random walk theory, building on the strategy for the related problem for L\'evy processes in \cite{DoneyMaller2002}. Taking into account Theorem 1 of \cite{chow1986}, such an ansatz, even though out of scope of this paper, is a possible strategy to tackle the problem at hand in our case as well.
\end{enumerate}
\end{remark}

\section{Application to real self-similar Markov processes} \label{sec: self similar}
In this section we show how to apply our results on the exponential mixing behavior of Markov additive processes sampled at first hitting times to the class of $\alpha$-self-similar Markov processes and in particular strictly $\alpha$-stable L\'evy processes. Even in the case of $\alpha$-stable processes the application is non-trivial since such  L\'evy processes do not satisfy the fundamental assumption of finite mean of the associated ascending ladder height L\'evy process, since in fact the ascending ladder height process is an $\alpha$-stable subordinator with $\alpha \in (0,1)$ and thus does not have a finite first moment. Because of non-ergodicity of the associated overshoots, we can therefore not expect a strong mixing behavior of the stable process sampled at first hitting times. However, making use of the \textit{Lamperti--Kiu transform} for real self-similar Markov processes, we can give bounds on the $\beta$-mixing coefficient of the $\sigma$-algebras generated by the past and the future of $\alpha$-self-similar process sampled at symmetric first hitting times given appropriate properties of the associated MAP. By considering the Lamperti-stable MAP and its explicit characterization found in \cite{chaumont2013}, we are thus able to bound the $\beta$-mixing coefficient of the above $\sigma$-algebras for transient $\alpha$-stable processes. To this end, let us first recall the precise definitions of real $\alpha$-self-similar Markov processes and $\alpha$-stable L\'evy processes and give a brief overview on the Lamperti--Kiu transform and its implications.

We say that a real-valued Feller process $(\Omega, \cG, (\cG_t)_{t \geq 0}, (Z_t)_{t \geq 0}, (\mathbf{P}^{x})_{x \in \R})$ is an $\alpha$-self-similar Markov process, if it satisfies the \textit{scaling property} that  for any $c> 0$,
\begin{equation}\label{eq: scaling property}
\{Z, \mathbf{P}^{cx}\} \overset{\mathrm{d}}{=} \big\{\big(cZ_{c^{-\alpha}t} \big)_{t\geq 0}, \mathbf{P}^{x}\big), \quad x \in \R.
\end{equation}
An (unkilled) L\'evy process $X =(X_t)_{t \geq 0}$ with associated family of probability measures $(\mathbf{P}^x)_{x \in \R}$ is a strictly $\alpha$-stable L\'evy process (or simply stable L\'evy process for short if there is no room for confusion) for $\alpha \in (0,2]$ if it satisfies \eqref{eq: scaling property}. The case $\alpha = 2$ boils down to Brownian motion, which we exclude from here-on. Since L\'evy processes are Feller, $\alpha$-stable L\'evy processes are therefore particular representatives of $\alpha$-self-similar Markov processes.

Taking the perspective commonly encountered in the literature to parametrize the stable process through its index of self-similarity $\alpha$ and the positivity parameter $\rho \coloneq \mathbf{P}^0(X_t \geq 0)$, the L\'evy measure $\Pi$ of $X$ is absolutely continuous with density $\pi$ satisfying
$$\pi(x) = c_+ x^{-(\alpha +1)} \one_{(0,\infty)}(x) + c_- \lvert x \rvert^{-(\alpha +1)} \one_{(-\infty, 0)}(x), \quad x \in \R,$$
where
$$c_+ = \Gamma(\alpha + 1)\frac{\sin(\pi \alpha \rho)}{\pi}, \quad \text{and} \quad c_- = \Gamma(\alpha + 1)\frac{\sin(\pi \alpha (1-\rho))}{\pi}.$$
The L\'evy--Khintchine exponent $\Psi$ is given by
$$\Psi(\theta) = c \lvert \theta \rvert^{\alpha} \big(1- \mathrm{i} \beta \tan \tfrac{\pi\alpha}{2} \mathrm{sgn}(\theta) \big), \quad \theta \in \R,$$
where $\beta = (c_+-c_-)\slash (c_+ + c_-)$ and our specific parametrization forces $c = \cos(\pi\alpha(\rho - 1\slash 2))$. For all of the above statements we refer to Kyprianou \cite{kyprianou2016}.

We now come to the one-to-one correspondence between self-similar Markov processes on $\R$ and Markov additive processes on $\R \times \{-1,1\}$  expressed through the Lamperti--Kiu transform, which is investigated in \cite{kiu80} and \cite{chaumont2013} for the real valued setting, and, more generally for arbitrary state spaces, in \cite{alili2017}. If we let $Z$ be an $\alpha$-self-similar Markov process on $\R$ absorbed at $0$ with lifetime $\tau_0 = \inf\{t > 0: X_t = 0\}$ and define $\PP^{x,i} = \mathbf{P}^{i \mathrm{e}^{x}}$ for $(x,i) \in \R \times \{-1,1\}$ and $\PP^{-\infty, \varpi} = \mathbf{P}^0$, then the process $(\xi,J)$ defined by
$$\begin{cases} \xi_t = \log \lvert Z_{\tau(t)}\rvert \text{ and } J_t = \mathrm{sgn}(Z_{\tau(t)}), \quad &\text{if } t < \int_0^{\tau_0} \lvert Z_s \rvert^{-\alpha} \diff{s},\\ (\xi_t, J_t) = \vartheta \eqcolon (-\infty, \varpi), \quad &\text{if } t \geq \int_0^{\tau_0} \lvert Z_s \rvert^{-\alpha} \diff{s}, \end{cases}$$
where $t \mapsto \tau(t)$ is the time change given by the right-continuous inverse
$$\tau(t) \coloneq \inf\{s \geq 0 : \int_0^s \lvert Z_u\rvert^{-\alpha} \diff{u} > t\},$$
of the continuous additive functional $(A_t)_{t \geq 0}$ of $Z$, given by
$$A_t \coloneq \int_0^{t \wedge \tau_0} \lvert Z_s \rvert^{-\alpha} \diff{s}, \quad t \geq 0,$$
and $\varpi$ is some isolated state, then $((\xi,J), (\PP^x)_{x \in (\R \times \{-1,1\})_{\vartheta}})$ is a MAP on $\R \times \{-1,1\}$ with lifetime $\zeta = \int_0^{\tau_0} \lvert Z_s \rvert^{-\alpha} \diff{s}$ and underlying filtration $(\cF_t = \cG_{\tau(t)})_{t \geq 0}$. Moreover, we have the following trichotomy characterizing the long-time behavior of the MAPs ordinator in terms of the hitting properties of $Z$ at $0$ (one can indeed verify that self-similarity of $Z$ guarantees that these are the only possible cases):
\begin{enumerate}[label=(\alph*),ref=(\alph*)]
  \item if $\mathbf{P}^x(\tau_0 < \infty) = 0$ for any $x \neq 0$, then $\lim_{t \to \infty} \xi_t = \infty$ almost surely;
  \item if $\mathbf{P}^x(\tau_0 < \infty, Z_{\tau_0-} = 0) = 1$ for any $x \neq 0$, then $\lim_{t \to \infty} \xi_t = -\infty$ almost surely;
  \item if $\mathbf{P}^x(\tau_0 < \infty, Z_{\tau_0-} \neq 0) = 1$ for any $x \neq 0$, then the MAP is almost surely killed and its lifetime $\zeta$ is exponentially distributed with a rate not depending on its initial distribution.
\end{enumerate}

Conversely, for a given MAP $(\xi,J)$ with lifetime $\zeta$,
$$Z_t = J_{\sigma(t)} \mathrm{e}^{\xi_{\sigma(t)}} \one_{\big\{ t < \int_0^\zeta \mathrm{e}^{\alpha \xi_s} \diff{s}\big\}}, \quad t \geq 0,$$
where
$$\sigma(t) = \inf\{s \geq 0: \int_0^{s} \mathrm{e}^{\alpha \xi_u} \diff{u} > t\}, \quad t \geq 0,$$
defines an $\alpha$-self-similar Markov process absorbed in $0$ with lifetime $\tau_0 = \int_0^\zeta \mathrm{e}^{\alpha \xi_s} \diff{s}.$ This is however not the direction we are interested in and we refer the reader to the relevant literature cited above for details. Note also that in case of $Z$ being strictly positive almost surely up to its lifetime, the Lamperti--Kiu transform boils down to the Lamperti transform for positive self-similar Markov processes and the associated MAP can be projected onto a killed L\'evy process.

With the help of the Lamperti--Kiu transform we obtain the following result on the $\beta$-mixing coefficient of the $\sigma$-algebras generated by $\alpha$-self-similar Markov processes sampled at past and future hitting times. While the Lamperti-stable MAP is exponentially $\beta$-mixing under the given assumptions, the $\beta$-mixing coefficient for the $\alpha$-self similar Markov process sampled at symmetric first hitting times shows non-homogeneous decay with polynomial rate as a result of the logarithm present in the Lamperti--Kiu transform.

\begin{proposition} \label{prop: mixing self-similar}
Suppose that $Z$ is $\alpha$-self-similar such that $\mathbf{P}^x(\tau_0 < \infty) = 0$ for all $x \neq 0$ and moreover its associated Lamperti--Kiu MAP satisfies the assumptions from Theorem \ref{theo: exp ergodicity}.\ref{theo: exp ergodicity 1}. If $\eta$ is some distribution on $(\R,\mathcal{B}(\R))$ that is concentrated on $\R\setminus (-1,1)$ such that
$$\int_{\R} \lvert x \rvert^\lambda \,\eta{(\diff{x})} < \infty,$$
for some $\lambda > 0$, then for any $\delta \in (0,1)$ there exists a constant $C(\lambda,\eta,\delta) > 0$ such that for any $t \geq 1$ we have
$$\beta_{\mathbf{P}^{\eta}}(\mathcal{N}_t,\overbar{\mathcal{N}}_{t+s}) \leq C(\lambda,\eta, \delta)\Big(\frac{t+s}{t} \Big)^{-\lambda/(1+\delta)}, \quad s > 0,$$
where we denoted
$$\mathcal{N}_t = \sigma\big(Z_{T_s^{\lvert Z \rvert}}, s \leq t \big), \quad  \overbar{\mathcal{N}}_t = \sigma\big(Z_{T_s^{\lvert Z \rvert}}, s \geq t \big).$$
\end{proposition}
\begin{proof}
First, observe that $Z$ not hitting $0$ almost surely when started away from the origin implies that the time change $(\tau(t))_{t \geq 0}$ is strictly increasing and continuous almost surely. Thus, the overshoot process of $(\log \lvert Z_t \rvert, \mathrm{sgn}(Z_t))_{t \geq 0}$ is indistinguishable from the overshoot process of the associated Lamperti-MAP $(\xi,J)$. Moreover,  the mapping
$$\phi\colon \R \times \{-1,1\} \to \R\setminus\{0\}, \quad (x,i) \mapsto i\mathrm{e}^{x},$$
is a homeomorphism and $Z_t = \phi(\log \lvert Z_t \rvert, \mathrm{sgn}(Z_t))$ for all $t \geq 0$ on the set $\Lambda = \{\omega \in \Omega: Z_t(\omega) \neq 0 \text{ for all } t \geq 0\}$, which is of $\mathbf{P}^\mu$-measure $1$ for any distribution $\mu$ on $(\R,\mathcal{B}(\R))$ not having an atom at $0$. It follows for any $t \geq 1$ with the notation from Corollary \ref{coroll: mixing} that there exists some $\mathbf{P}^\mu$-nullset $N_t^\mu$ such that
$$\mathcal{N}_t\vee N_t^\mu = \big(\big(\xi_{T_s}, J_{T_s}\big), s \leq \log t\big) \vee N_t^\mu = \mathcal{K}_{\log(t)} \vee N_t^\mu,$$
and
$$\overbar{\mathcal{N}}_t\vee N_t^\mu = \big(\big(\xi_{T_{s}}, J_{T_{s}}\big), s \geq \log t\big) \vee N_t^\mu = \overbar{\mathcal{K}}_{\log(t)}\vee N_t^\mu,$$
where for two $\sigma$-algebras $\mathcal{A},\mathcal{B}$ we write $\mathcal{A} \vee \mathcal{B} = \sigma (\mathcal{A} \cup \mathcal{B})$ and we used that for any $1 \leq s \leq t$, $Z_{T^{\lvert Z \rvert}_s} = \phi(\xi_{T_{\log s}}, J_{T_{\log s}})$, $\mathbf{P}^\eta$-almost-surely.
Since moreover $\mathbf{P}^{\eta} = \PP^{\eta \circ \phi}$, it follows from the assumptions on $\eta$ that $\eta \circ \phi$ is concentrated on $\R_+ \times \{-1,1\}$ and that
$$\int_{\R \times \{-1,1\}} \mathrm{e}^{\lambda x} \,\eta \circ \phi(\diff{x},\diff{i}) = \int_{\R \setminus (-1,1)} \mathrm{e}^{\lambda\log \lvert z \rvert} \,\eta(\diff{z}) = \int_{\R \setminus (-1,1)} \lvert z \rvert^\lambda \, \eta(\diff{z}) < \infty.$$
Thus from Corollary \ref{coroll: mixing} and the assumptions on the Lamperti-MAP $(\xi,J)$ it follows that for any $\delta \in (0,1)$ there exists $C(\lambda, \eta, \delta) > 0$ such that
for any $t \geq 1$ and $s > 0$
$$\beta_{\mathbf{P}^{\eta}}(\mathcal{N}_t,\overbar{\mathcal{N}}_{t+s}) = \beta_{\PP^{\eta \circ \phi}}(\mathcal{K}_{\log t},\overbar{\mathcal{K}}_{\log(t+s)}) \leq C(\lambda,\eta,\delta)\mathrm{e}^{-\lambda(\log(t+s)-\log t)/(1+\delta)} = C(\lambda,\eta,\delta) \Big(\frac{t+s}{t} \Big)^{-\lambda/(1+\delta)},$$
as claimed. Note here that the nullsets $N_t^\mu$ and $N_{t+s}^\mu$ from above have no influence on the $\beta$-mixing coefficient by its definition.
\end{proof}

Consider now a scalar $\alpha$-stable process $(X^0_t)_{t \geq 0}$ absorbed upon hitting of the origin, i.e.\ for $\tau_0 = \inf\{s \geq 0 : X_s = 0\}$,
$$X^0_t = X_t \one_{[0,\tau_0)}(t), \quad t \geq 0.$$
We show that $X^0$ satisfies the assumptions from Proposition \ref{prop: mixing  self-similar} that yield $\beta$-mixing of overshoots of the corresponding MAP $(\xi,J)$ obtained through the Lamperti--Kiu transform, which we henceforth will refer to as the \textit{Lamperti-stable MAP}. 

Since the assumptions are couched in form of the ascending ladder height process $(H^+,J^+)$, one direct approach would be to make use of the \textit{deep factorization} of $X^0$ given in \cite{kyprianou2016}, where the MAP exponents $\bm{\Phi}^+$ and $\hat{\bm{\Phi}}{}^+$ of the ascending ladder height processes of $(\xi,J)$ and its dual were explicitly calculated. However, for the sake of exposition, we go another route by making use of the results based on Vigon's équations amicales inversés from Section \ref{sec: vigon} to infer the needed properties of $(H^+,J^+)$ from those of $(\xi,J)$. The characteristics of the latter were calculated in Theorem 10 and Corollary 11 of Chaumont et al.\ \cite{chaumont2013}, giving $\sigma_{\pm 1} = 0$, i.e.\ the underlying L\'evy processes have no Gaussian component,
\begin{align*}
\Pi_{\pm 1}(\diff{x}) &= \mathrm{e}^x \pi(\pm (\mathrm{e}^x-1)) \diff{x}, \quad x \in \R,\\
F_{\pm 1, \mp 1}(\diff{x}) &= \frac{\alpha \mathrm{e}^x}{(1+ \mathrm{e}^x)^{\alpha + 1}} \diff{x}, \quad x \in \R,
\end{align*}
and
$$q_{\pm 1, \mp 1} = \frac{c_{\mp}}{\alpha}.$$
If we assume that $X$ does not have one-sided jumps, then $c_{\pm} > 0$ and hence $J$ is irreducible. Since $\Pi_{\pm1}$ has a strictly positive Lebesgue density on $(0,\infty)$ it follows by Lemma \ref{lemma: cond vigon} that $\mathrm{Leb}\vert_{(0,\infty)} \ll \Pi^+_{\pm 1}\vert_{(0,\infty)}$ as well. Further, we have for $\lambda > 0$
$$\int_1^\infty \mathrm{e}^{\lambda x} \,\Pi_{1}(\diff{x}) = c_+ \int_1^\infty \mathrm{e}^{(\lambda + 1) x} (\mathrm{e}^x -1)^{-(\alpha +1)} \diff{x},$$
and hence
$$\int_1^\infty \mathrm{e}^{\lambda x} \,\Pi_{1}(\diff{x}) < \infty \Longleftrightarrow \lambda \in (0,\alpha).$$
Similarly, we obtain
$$\int_1^\infty \mathrm{e}^{\lambda x} \,\Pi_{-1}(\diff{x}) < \infty \Longleftrightarrow \lambda \in (0,\alpha),$$
and hence $\E[\exp(\lambda \xi^{(\pm 1)}_1)] < \infty$ iff $\lambda \in (0,\alpha)$. Moreover,
$$\int_{\R} \mathrm{e}^{\lambda x} \,F_{\pm1}(\diff{x}) = \alpha \int_{\R} \mathrm{e}^{(\lambda+1)x} (1+\mathrm{e}^x)^{-(\alpha + 1)} \diff{x} < \infty \Longleftrightarrow \lambda \in (0,\alpha).$$
Again by Lemma \ref{lemma: cond vigon} we conclude that $H^{+,(\pm 1)}$ and $\Delta^{+}_{\pm 1, \mp 1}$ all possess an exponential $\lambda$-moment whenever $\lambda \in (0,\alpha)$.

Recall now that $X$ does not hit $0$ if and only if $\alpha \in (0,1)$ and hence the ordinator $\xi$ of the Lamperti-stable MAP satisfies $\limsup_{t \to \infty} \xi_t = \infty$ almost surely if and only if $\alpha \in (0,1)$. Since our asymptotic approach on overshoots of MAPs requires this property, we will restrict to this case and can therefore identify $X = X^0$ almost surely. All that remains to show for exponential $\beta$-mixing of the Lamperti-stable MAP is now upward regularity and irreducibility of $J^+$. Irreducibility of $J^+$ is a direct consequence of Proposition \ref{prop: irr} since $J$ is irreducible and the support of $\Pi_{\pm}$ is unbounded. 
To verify upward regularity, we observe that by Theorem 1 in Kuznetsov and Pardo \cite{KuznetsovPardo2013}, $\xi^{(1)}$ killed at an independent exponential time with rate $c_-\slash \alpha$ belongs to the class of \textit{hypergeometric L\'evy processes} with parameters $(1-\alpha(1-\rho), \alpha \rho, (1-\alpha)(1-\rho), \alpha(1-\rho))$. The ascending ladder height process $H$ of such a hypergeometric L\'evy process is a $\beta$-subordinator with parameters $(\alpha(1-\rho),\alpha(1-\rho),1- \alpha\rho)$, whose L\'evy measure is given by
$$\Pi_H(\diff{x}) = \frac{1-\alpha \rho}{\Gamma(\alpha \rho)}(1-\mathrm{e}^{-x})^{\alpha\rho -2} \mathrm{e}^{-(1+\alpha(1-2\rho))x} \diff{x}, \quad x > 0.$$
Clearly, $\Pi_H((0,1)) = \infty$ and hence $H$ is not compound Poisson, which shows that the associated hypergeometric L\'evy process is upward regular. Since killing has no influence on upward regularity, this now shows that $\xi^{(1)}$ is indeed upward regular. Upward regularity of $\xi^{(-1)}$ can be argued analogously once we observe that $\xi^{(-1)}$ killed at rate $c_+\slash \alpha$ is the hypergeometric process obtained from killing the dual process $\hat{X}$ of $X$ upon entering $(-\infty,0]$. Hence, with the ergodic analysis of overshoots from Section \ref{sec: stability} and Proposition \ref{prop: mixing self-similar}, we have proved the following.

\begin{proposition} \label{prop: mixing alpha}
Let $\alpha \in (0,1)$ and $X$ be strictly $\alpha$-stable. Then the overshoot process of the Lamperti-stable MAP associated to $X$ is $\mathcal{R}_\lambda V_\lambda$-uniformly ergodic and for any starting distribution $\mu$ such that $\mu(\cdot,\{-1,1\})$ has an exponential $\lambda$-moment for any $\lambda \in (0,\alpha)$, the overshoot process is exponentially $\beta$-mixing. Moreover, for any distribution $\eta$ on $(\R,\cB(\R))$ concentrated on $\R \setminus (-1,1)$ such that for some $\lambda \in (0,\alpha)$,
$$\int_{\R} \lvert x \rvert^\lambda \,\eta{(\diff{x})} < \infty,$$
there exists a constant $C(\lambda,\eta,\delta) > 0$ for any $\delta \in (0,1)$ such that for any $t \geq 1$ we have
$$\beta_{\mathbf{P}^{\eta}}(\mathcal{N}_t,\overbar{\mathcal{N}}_{t+s}) \leq C(\lambda,\eta,\delta)\Big(\frac{t+s}{t} \Big)^{-\lambda/(1+\delta)}, \quad s > 0,$$
where we denoted
$$\mathcal{N}_t = \sigma\big(X_{T_s^{\lvert X \rvert}}, s \leq t \big), \quad  \overbar{\mathcal{N}}_t = \sigma\big(X_{T_s^{\lvert X \rvert}}, s \geq t \big).$$
In particular, for any $x \in \R\setminus(-1,1)$ and any $\delta > 0$, there exists a constant $\tilde{C}(\delta) > 0$ such that for any $t \geq 1$
$$\beta_{\mathbf{P}^x}(\mathcal{N}_t,\overbar{\mathcal{N}}_{t+s}) \leq \tilde{C}(\delta)(1+s/t)^{-\alpha/(1+\delta)},\quad s > 0.$$
\end{proposition}

\begin{appendix}
\section{Stability of Markov processes} \label{sec: markov stability}
The theory of stability of time-continuous Markov processes is the essential tool for our analysis of MAPs sampled at first hitting times through studying ergodic properties of its overshoots. Stability of discrete-time Markov chains has a long history dating back at least to the 1930s, whereas the systematic treatment of stability of continuous-time Markov processes is comparatively young. The Meyn and Tweedie approach developed in the 1990s, which we will follow, infers recurrence and ergodic properties of continuous-time processes through techniques developed for the discrete-time case by means of sampling the process on a countable grid of (random) times. In this way accessible criteria are established in terms of the transition semigroup, the generator or the resolvent of the Markov process, at least one of which is available for the specific processes under consideration. Let us therefore introduce the most important terminology and results that we will need in the following.

Let $\mathbf X =(X_t)_{t \geq 0}$ be a continuous-time Borel right Markov process on a locally compact, separable space $(\mathcal{X},\cB(\mathcal{X}))$ with cemetery state $\vartheta$, lifetime $\zeta$ and underlying family of probability measures $(\PP^x)_{x \in \cX_\vartheta}$, where $\cX_\vartheta$ denotes the Alexandrov one-point compactification of $\cX$. Denote its sub-Markov transition semigroup by $(P_t)_{t \geq 0}$, which is induced by $(\PP^x)_{x \in \cX_\vartheta}$ via $P_t(x,B) = \PP^x(X_t \in B, t < \zeta)$ for $(x,B) \in \cX_\vartheta \times \cB(\cX_\vartheta).$ Note that by our convention to extend functions $f \in \cB(\cX)$ to $\cB(\cX_\vartheta)$ by setting $f(\vartheta) = 0$, $(P_t)_{t \geq 0}$ restricted to  $\cB_b(\cX)$ coincides with the  Markov transition semigroup of $\mathbf{X}$. Borel right Markov processes are the most general class of strong Markov processes widely used in the literature and are the cornerstone of the \textit{théorie générale} of Markov processes developed mainly in the 1960s to 1980s based on fundamental works of Dynkin, Feller and others. Their precise potential theoretically motivated definition can be found in Sharpe \cite[Definition 8.1]{sharpe1988}, but for our purposes it will be enough to know that Feller processes as defined below are Borel right, which follows from Theorem II.2.12 in \cite{blumenthal1968}. 

We understand Feller processes as càdlàg Markov processes with right-continuous and complete filtration, whose sub-Markov transition semigroup $(P_t)_{t \geq 0}$ is (i) strongly continuous, i.e.\ $\lVert P_t f - f\rVert_\infty \to 0$ as $t \to 0$ for any $f \in \mathcal{C}_0(\cX)$, and (ii) $P_t \mathcal{C}_0(\cX) \subset \mathcal{C}_0(\cX)$ for all $t \geq0$, where $\mathcal{C}_0(\cX)$ is the space of continuous functions on $\cX$ vanishing at infinity. Note that in presence of the \textit{Feller property} (ii), strong continuity (i) is actually satisfied whenever we have pointwise convergence $P_tf(x) \to f(x)$ as $t \to 0$ for all $x \in \cX$, see e.g.\ Kallenberg \cite[Theorem 19.6]{kallenberg2002}.

The \textit{resolvent} $(U_\lambda)_{\lambda > 0}$ of $\mathbf{X}$ is the operator defined by
$$U_\lambda f(x) = \E^x\Big[\int_0^\infty \mathrm{e}^{-\lambda t} f(X_t) \diff{t} \Big] = \int_0^\infty \mathrm{e}^{-\lambda t} \E^x[f(X_t)] \diff{t}, \quad \lambda > 0, x \in \cX_\vartheta, f \in \cB_b(\cX_\vartheta)\cup \cB_+(\cX_\vartheta),$$
which for $x \in \cX$ and $f \in \cB_b(\cX)$ can be written as
$$U_\lambda f(x) = \int_0^\infty \mathrm{e}^{-\lambda t} P_tf(x)\diff{t}.$$
The family of proability measures induced by the $\lambda$-resolvent $U_\lambda$ via $U_\lambda(x,B) \coloneqq U_\lambda \one_B(x)$ can be interpreted as the potentials $U(x,\cdot)$ of the Markov process $\mathbf{X}$ killed at an independent exponential time with rate $1\slash \lambda$, where for $x\in \cX_\vartheta$, the potential $U(x,\cdot)$ defined by
$$U(x,B) \coloneq \int_0^\infty \PP^x(X_t \in B) \diff{t}, \quad B \in \cB(\cX_\vartheta),$$
is the expected sojourn time of $\mathbf{X}$ in $B$ when started in $\mathbf{X}$.

Suppose from here on that $\zeta$ is almost surely infinite and thus $\mathbf{X}$ is an unkilled Borel right Markov process, which is the setting in which Meyn and  Tweedie's stability theory is embedded in. We say that a $\sigma$-finite measure $\chi$ on $(\cX,\cB(\cX))$ is an \textit{invariant measure} for $\mathbf{X}$, if
$$\forall B \in \cB(\cX)\colon \, \PP^\chi(B) \coloneq \int_0^\infty \PP^x(X_t \in B)\, \chi(\diff{x}) = \chi(B).$$
Note that an invariant measure is never unique, since any scaling of the measure is again invariant. We therefore say that an invariant measure $\chi$ is \textit{essentially} unique if it is unique up to constant multiples. If $\chi(\cX) = 1$, we call $\chi$ an \textit{invariant distribution} (which is unique under Harris recurrence, which we define below). The following proposition gives a criterion in terms of the resolvent, which is helpful for detecting an invariant measure provided the resolvent can be determined analytically. The statement has a very classical flavor, but we were not able to find it in the literature. We will use it in combination with an analytical treatment of the resolvent of the overshoot process to determine an essentially unique measure for this process.

\begin{proposition}\label{invariant measure resolvent}
Suppose that $\cH \subset \cB_b(\cX) \cap \cB_+(\cX)$ such that $P_t \cH \subset \cH$ for any $t \geq 0$ and there is a non-trivial measure $\chi$ on $(\cX,\cB(\cX))$ and a family $(\alpha_\lambda)_{\lambda > 0}$ of finite measures on $(\cX, \cB(\cX))$ satisfying $\lim_{\lambda \downarrow 0} \alpha_\lambda(\cX) = 0$ such that for any $f \in \cH$
\begin{equation}\label{resolvent scaling convergence}
\lim_{\lambda \downarrow 0} \int_{\cX} U_\lambda f(x) \, \alpha_\lambda(\diff{x}) = \chi(f) \coloneq \int_\cX f(y) \,\chi(\diff{y}).
\end{equation}
Then, for any $t \geq 0$ and $f \in \cH$,
$$\int_\cX P_tf(y) \, \chi(\diff{y}) = \chi(f).$$
In particular, if $\cH = \cB_b(\cX) \cap \cB_+(\cX)$ (i.e.\ $U_\lambda^{\alpha_\lambda} \coloneq \int_{\cX} U_\lambda(x, \diff{y}) \, \alpha_\lambda(\diff{x})$ converges strongly to $\chi$ as $\lambda \downarrow 0$), then $\chi$ is an invariant measure of $\mathbf{X}$.
\end{proposition}
\begin{proof}
Let $f \in \cH$ such that \eqref{resolvent scaling convergence} holds and $t \geq 0$. We have for any $\lambda > 0$ by the semigroup property of $(P_t)_{t \geq 0}$
\begin{align*}
  U_\lambda^{\alpha_\lambda} (P_tf) &= \int_{\cX} \int_0^\infty \mathrm{e}^{-\lambda s} P_s P_tf(x) \diff{s} \,\alpha_\lambda(\diff{x})\\
  &= \int_{\cX} \int_0^\infty \mathrm{e}^{-\lambda s} P_{s+t}f(x) \diff{s} \, \alpha_\lambda(\diff{x})\\
  &= \mathrm{e}^{\lambda t}\int_{\cX} \int_t^\infty \mathrm{e}^{-\lambda s} P_sf(x) \diff{s} \, \alpha_\lambda(\diff{x})\\
  &= \mathrm{e}^{\lambda t} \left(U_\lambda^{\alpha_\lambda} f - \int_{\cX} \int_0^t \mathrm{e}^{-\lambda s} P_sf(x) \diff{s} \, \alpha_\lambda(\diff{x})\right).
\end{align*}
Since $\lvert \int_{\cX} \int_0^t \mathrm{e}^{-\lambda s} P_sf(x) \diff{s} \,\alpha_\lambda(\diff{x}) \rvert \leq t \lVert f \rVert_\infty \alpha_\lambda(\cX)$ it therefore follows by our assumption that $\alpha_\lambda(\cX) \rightarrow 0$ and
$$U_\lambda^{\alpha_\lambda}(f) \rightarrow \chi(f)$$
as $\lambda \downarrow 0$ that
$$\lim_{\lambda \downarrow 0} U_\lambda^{\alpha_\lambda}(P_tf)  = \chi(f).$$
On the other hand, our assumptions and $P_tf \in \cH$ yield that
$$\lim_{\lambda \downarrow 0} U_\lambda^{\alpha_\lambda}(P_tf)  = \int_\cX P_t f(y) \, \chi(\diff{y})$$
and hence
$$\int_\cX P_tf(y) \, \chi(\diff{y}) = \chi(f)$$
follows. If $\cH = \cB_b(\cX) \cap \cB_+(\cX)$, then for any $B \in \cB(\cX)$ the choice $f = \one_B$  shows that
$$P_t(\chi, B) = \chi(B), \quad \forall t\geq 0,$$
i.e.\ $\chi$ is an invariant measure.
\end{proof}

A $\sigma$-finite measure $\psi$ is called \textit{irreducibility measure} of $\mathbf{X}$, if for any Borel set $B$, $\psi(B) > 0$ implies $U(x,B) > 0$ for any $x \in \cX.$ Whenever such a measure exists, we say that $\mathbf{X}$ is $\psi$-\textit{irreducible} or simply irreducible when the specific measure does not matter. If $\mathbf{X}$ is irreducible, there exists a maximal irreducibility measure $\psi$ in the sense that for any irreducibility measure $\nu$ of $\mathbf{X}$ it holds that $\nu \ll \psi$, see Tweedie \cite[Theorem 2.1]{Tweedie1994}. We define  $\cB^+(\cX) \coloneq \{B \in \cB(\cX): \psi(B) > 0\}$ and call sets in $\cB^+(\cX)$ accessible. Note that maximal irreducibility measures are clearly non-unique. Moreover, if $\mathbf{X}$ is $\psi$-irreducible and admits an invariant measure $\chi$, then $\chi$ is a maximal irreducibility measure. To see this, let $\psi(B) > 0$, then
$$t\chi(B) = \int_0^t \Big( \int_{\cX} \PP^x(X_s \in B)\, \chi(\diff{x})\Big) \diff{s} = \int_{\cX} \Big( \int_0^t \PP^x(X_s \in B)\diff{s}\Big) \,\chi(\diff{x})$$
and by monotone convergence the right hand side converges to $U(\chi,B) \coloneq \int_{\cX} U(x,B) \, \chi(\diff{x}) > 0$ since $U(x,B) > 0$ for all $x \in \cX$ by our choice of $B$. Hence, $\psi \ll \chi$. The next important concept, \textit{Harris recurrence}, is an even stronger property than irreducibility. We say that $\mathbf{X}$ is $\mu$-\textit{Harris recurrent} if there exists a $\sigma$-finite measure $\mu$ on the state space s.t.\
\begin{equation} \label{eq: harris}
\forall B \in \cB(\cX)\colon \, \mu(B) > 0 \implies \PP^x\Big(\int_0^\infty \one_B(X_t) \diff{t} = \infty\Big) =1, \quad \forall x \in \cX,
\end{equation}
i.e.\ if $\mu(B) > 0$, the process almost surely spends infinitely much time in the set $B$. A powerful implication of Harris recurrence is that an invariant measure of a Markov process having this property (we call such processes \textit{positive Harris recurrent}) is essentially unique, see \cite[Théorème 2.5]{AzemaDufloRevuz1969}. Moreover, by the remark succeding this theorem in \cite{AzemaDufloRevuz1969}, an invariant measure $\chi$ of a Harris recurrent process is a Harris measure. Thus, it is \textit{maximal Harris} in the sense that it dominates any other Harris measure, since any Harris measure is in particular an irreducibility measure and $\chi$ is a maximal irreducibility measure, as discussed above.  The defining condition for Harris recurrence is often hard to check directly, however, Kaspi and Mandelbaum \cite[Theorem 1]{KaspiMandelbaum1994} provide us with a simpler equivalent criterion for Borel right Markov processes: suppose that there exists a $\sigma$-finite measure $\nu$ such that for any Borel set $B$ we have the implication
\begin{equation} \label{eq: kaspi}
\nu(B) > 0 \implies \PP^x(T_B < \infty) = 1, \quad \forall x \in \cX,
\end{equation}
where $T_B \coloneq \inf\{t \geq 0: X_t \in B\}$ is the first hitting time of $B$. Then, $\mathbf{X}$ is Harris recurrent and a Harris recurrence measure $\mu$ is given by
\begin{equation}\label{eq: harris2}
\mu(B) = \E^{\nu} \Big[\int_0^\infty \mathrm{e}^{- t} \one_B(X_t) \diff{t} \Big] = U_1(\nu,B), \quad B \in \cB(\cX).
\end{equation}
Let us now recall the notion of \textit{petite} and \textit{small} sets, with the former concept being a generalization of the latter. We say that a non-empty set $C \in \cB(\cX)$ is petite, if there exists a sampling distribution $a$ on $((0,\infty), \cB(0,\infty))$ and a non-trivial measure $\nu_a$ on the state space such that for the sampled kernel
$$K_a(x, \diff{y}) \coloneq \int_{0+}^\infty P_t(x,\diff{y}) \, a(\diff{t}), \quad x,y \in \cX,$$
it holds that
$$K_a(x,\cdot) \geq \nu_a(\cdot), \quad x \in C.$$
The sampled kernel corresponds to the transition kernel of the discrete-time Markov chain obtained from $\mathbf{X}$ by sampling at renewal times of an independent renewal process with increment distribution $a$. An important special case is the $\lambda$-resolvent kernel
$$R_\lambda(x,\diff{y}) \coloneq \int_{0+}^\infty \lambda\mathrm{e}^{-\lambda t} P_t(x,\diff{y}) \diff{t} = \lambda U_\lambda(x,\diff{y}), \quad x,y \in \cX,$$
obtained for the sampling distribution $a = \mathrm{Exp}(\lambda)$, $\lambda > 0$. If $a = \delta_\Delta$ for some $\Delta > 0$, then $C$ is called a \textit{small} set and we refer to the sampled chain $\mathbf{X}^\Delta \coloneq (X_{n\Delta})_{n \in \N_0}$ as the $\Delta$-skeleton of $\mathbf{X}$. The importance of petite sets comes from the fact, that petite sets are small for the sampled chain and small sets in discrete time Markov chain theory allow to construct a related Markov chain possessing an atom via the technique of Nummelin splitting, which then makes renewal arguments that are well-known for Markov chains on countable state spaces transferrable to the general state space situation. We refer to Meyn and Tweedie \cite{MeynTweedie2009} for a comprehensive account. Let us also remark that petite sets are by no means rare. E.g.\ consider the case that $\mathbf{X}$ is a $T$-process, that is there exists a non-trivial continuous component $T$ for the sampled kernel $K_a$ for some sampling distribution $a$, meaning that
\begin{enumerate}[label=(\alph*),ref=(\alph*)]
  \item $x \mapsto T(x,B)$ is lower semicontinuous for all $B \in \mathcal{B}(\cX);$
  \item $K_a(x,B)  \geq T(x,B)$ for all $x \in \cX$ and $B \in \mathcal{B}(\cX).$
\end{enumerate}
Then every compact subset of $\cX$ is petite, provided that $\mathbf{X}$ is irreducible, see Theorem 5.1 in \cite{Tweedie1994}.

The final concept that we need is \textit{aperiodicity}. We say that $\mathbf{X}$ is aperiodic, if there exists a petite set $C \in \cB^+(\cX)$ (i.e.\, $C$ must be accessible) and some $T \geq 0$ s.t.\
$$\forall t \geq T, x \in C\colon \, P_t(x,C) > 0.$$
Alternatively, $\mathbf{X}$ is called aperiodic in \cite{MeynTweedie2009} if there exists some $\Delta > 0$ such that the $\Delta$-skeleton $\mathbf{X}^\Delta$ is irreducible, i.e.\ there exists a $\sigma$-finite measure $\mu$ on $(\cX,\cB(\cX))$ such that
$$\mu(B) > 0 \implies \forall x \in \cX\colon \, \sum_{n =1}^\infty \PP^x(X_{n\Delta} \in B) > 0.$$
It seems to be well-known in the literature that the existence of an irreducible skeleton chain for a Harris recurrent Markov process implies aperiodicity, but there is no concrete statement to be found. Proposition 6.1 in \cite{MeynTweedie1993}, which  \cite{douc2009} refers to, does not quite state that irreducibility of skeletons implies aperiodicity, but indeed provides the right tool to prove it. For completeness we give the short proof and make the additional simple observation that if the petite set $C$ in the definition of aperiodicity is a singleton set, then aperiodicity also implies the existence of an irreducible skeleton chain, which will be useful later on.

\begin{lemma} \label{lemma: aperiod}
Suppose that the $\psi$-irreducible Markov process $\mathbf{X}$ is positive Harris recurrent, Borel right and its state space is locally compact and separable. Then, if there exists some irreducible skeleton chain, $\mathbf{X}$ is aperiodic. Conversely, if $\mathbf{X}$ is aperiodic and the defining set $C$ is a singleton set, then any $\Delta$-skeleton is irreducible.
\end{lemma}
\begin{proof}
Suppose first that there exists some irreducible $\Delta$-skeleton. Then, the assumptions on the process allow to use Proposition 6.1 from \cite{MeynTweedie1992}, which states that for any petite set $C$ there exists some non-trivial measure $\mu$ and and a $T > 0$ such that for all $t \geq T$ we have
\begin{equation} \label{eq1 skeleton aperiodic}
\PP^x(X_t \in \cdot) \geq \mu(\cdot), \quad \forall t \geq T, x \in C,
\end{equation}
which implies in particular that $C$ is even a small set. By the Markov property it thus follows for $s \geq 0$ that
\begin{equation} \label{eq2 skeleton aperiodic}
\PP^x(X_{t +s} \in \cdot) = \int_\cX \PP^x(X_t \in \diff{y})\,\PP^y(X_s \in \cdot)\geq \int_\cX \mu(\diff{y})\,\PP^y(X_s \in \cdot) = \PP^{\mu}(X_s \in \cdot), \quad \forall t \geq T, x \in C.
\end{equation}
By Proposition 3.4 of Meyn and Tweedie \cite{MeynTweedie1993b} the state space $\cX$ can be covered by countably many petite sets ($=$ small sets in our case), hence we may assume that $\psi(C) > 0$, i.e.\ $C \in \cB^+(\cX).$ Note that $U(x,C) > 0$ for all $x \in \cX$ and non-triviality of $\mu$ then yield that $U(\mu,C) = \int_\cX U(x,C) \,\mu(\diff{x}) > 0$ and since with Fubini $U(\mu,C) = \int_0^\infty \PP^\mu(X_t \in C) \diff{t}$ it follows that there exists $s > 0$ such that $\PP^\mu(X_s \in C) > 0$. From \eqref{eq2 skeleton aperiodic} it thus follows that for such $s$ and all $t \geq T+s$ and $x \in C$ it holds that
$$\PP^x(X_{t} \in C) \geq \PP^\mu(X_s \in C) > 0,$$
which proves aperiodicity of $\mathbf{X}$.

Suppose now that $\mathbf{X}$ is aperiodic with defining small singleton set $C = \{c\} \in \cB^+(\cX)$ for some $c \in \cX$. Then, there exists $T> 0$ such that
$$\PP^{c}(X_t = c) > 0, \quad \forall t\geq T,$$
and $\delta_c$ is an irreducibility measure. Then, for given $x \in \cX$, there exist $t_x$ such that $\PP^x(X_{t_x} = c) > 0$ and the Markov property yields for any $t \geq T$
\begin{align*}
  \PP^x(X_{t_x+t} =c) &\geq \PP^x(X_{t_x+t} = c, X_{t_x} = c) = \PP^x(X_{t_x}=c)\PP^{c}(X_t =c) > 0.
\end{align*}
Hence, for given $\Delta > 0$, if we choose $n \in \N$ such that $n\Delta \geq t_x + T$, it follows that $\PP^x(X_{n\Delta} = c) > 0$ and thus $\mathbf{X}^\Delta$ is $\delta_c$-irreducible.
\end{proof}

We are now well-suited to discuss ergodicity of a Markov process. Let $\lVert \cdot \rVert_{\mathrm{TV}}$ denote the total variation norm on the space of signed finite measures $\mathcal{M}_b^s(\cX,\cB(\cX))$ on $(\cX,\cB(\cX))$, defined by
$$\lVert \nu\rVert_{\mathrm{TV}} \coloneq \sup_{\lvert g \rvert \leq 1} \lvert \nu(g) \rvert, \quad \nu \in \mathcal{M}_b^s(\cX,\cB(\cX)).$$
We say that $\mathbf{X}$ having a stationary distribution $\rho$ is \textit{ergodic} if
$$\forall x \in \cX\colon \quad \lim_{t \to \infty} \lVert \PP^x(X_t \in \cdot) - \rho\rVert_{\mathrm{TV}} = 0.$$
Clearly, ergodicity implies weak convergence of the marginal distributions of $\mathbf{X}$ to its invariant distribution. If $\mathbf{X}$ is positive Harris recurrent, Theorem 6.1 in Meyn and Tweedie \cite{MeynTweedie1993} provides us with a necessary and sufficient criterion for ergodicity in terms of skeletons of the process:
\begin{equation}\label{eq: ergodicity}
\mathbf X \text{ is ergodic} \iff \exists \Delta > 0 \text{ s.t. } \mathbf{X}^\Delta \text{ is irreducible}.
\end{equation}
Once we know that $\mathbf{X}$ is ergodic, a natural question is the rate of convergence of the marginals to the invariant distribution. To this end, \cite{DownMeynTweedie1995} investigate convergence in the so called $f$-norm. For  a strictly positive, measurable function $f \in \cB(\cX)$ satisfying $f \geq 1$, the $f$-norm on $\mathcal{M}_b^s(\cX,\cB(\cX))$ on $(\cX,\cB(\cX))$ is given by
$$\lVert \nu \rVert_f \coloneq \sup_{\lvert g \rvert \leq f} \lvert \nu(g) \rvert, \quad \nu \in \mathcal{M}(\cX,\cB(\cX)),$$
where the supremum is taken over all measurable functions $g$ bounded by $f$. Note that for $f \equiv 1$, the $f$-norm reduces to the total variation norm. We say that the Markov process $\mathbf{X}$ with stationary distribution $\rho$ is $f$-\textit{uniformly ergodic} if there exist constants $D,\kappa > 0$ such that
\begin{equation} \label{eq: uniform ergodicity}
\lVert P_t(x,\cdot) - \mu \rVert_f \leq Df(x) \mathrm{e}^{-\kappa t}, \quad x \in \cX,
\end{equation}
which in particular implies that the marginal distributions of $\mathbf{X}$ converge to the stationary distribution at an exponential rate in total variation. For the latter, we also refer to the process as being \textit{exponentially} or \textit{geometrically ergodic}.

\cite{DownMeynTweedie1995} give conditions in terms of drift criteria for the generator, semigroup and resolvent kernel for $f$-uniform ergodicity. For our treatment of overshoots based on their resolvent, we will choose the resolvent drift criterion for determining the convergence speed of overshoots. More precisely, if $\mathbf{X}$ is irreducible and aperiodic and for some $\lambda > 0$ there exist constants $b \in \R_+, \beta \in (0,1)$, a petite set $C$ and a measurable function $V \geq 1$ such that
\begin{equation} \label{eq: exp ergodic1}
R_\lambda V \leq \beta V + b\one_C,
\end{equation}
Theorem 5.2 in \cite{DownMeynTweedie1995} tells us that $\mathbf{X}$ is $R_\lambda V$-uniformly ergodic. If $V$ is \textit{unbounded off petite sets}, that is $\{x \in \cX: V(x) \leq z\}$ is petite for any $z > 0$, \eqref{eq: exp ergodic1} is equivalent to demanding that there exists $\beta_0 \in (0,1)$ such that
\begin{equation} \label{eq: exp ergodic2}
R_\lambda V\leq \beta_0 V + b.
\end{equation}
To see this, for $\gamma > 1$ define the petite set $C(\gamma) \coloneq \{x \in \cX: V(x) \leq \gamma b\slash(1-\beta_0)\}$, then
\begin{equation}\label{eq: erg equiv}
\begin{split}
\beta_0 V + b & \leq \beta_0 V + b \one_{C(\gamma)} + \frac{1}{\gamma}(1-\beta_0)V \one_{C(\gamma)^{\mathsf{c}}}\\
&\leq \frac{1}{\gamma}(1+ (\gamma - 1)\beta_0)V + b \one_{C(\gamma)},
\end{split}
\end{equation}
showing that for any choice of $\gamma > 1$, \eqref{eq: exp ergodic2} implies \eqref{eq: exp ergodic1} with $ C = C(\gamma)$ and $\beta = (1+ (\gamma - 1)\beta_0)\slash \gamma \in (0,1)$. The converse relation is obvious.

General drift to petite sets criteria for the speed of convergence to the invariant distribution were extended in \cite{douc2009} to the case of \textit{subgeometric} rates. The combined conclusions of Theorem 3.2 and Theorem 4.9 in \cite{douc2009} read that if $\mathbf{X}$ is ergodic and for some $\lambda > 0$ there exists 
\begin{itemize}
\item a closed, petite set $C$ and a constant $b < \infty$,
\item a function $\tilde{V}\colon \mathcal{X} \to [1,\infty)$, 
\item an increasing, differentiable and concave function $\phi\colon [1,\infty) \to (0,\infty)$,
\end{itemize}
such that
\begin{equation}\label{eq:drift subgeo}
R_\lambda \tilde V \leq \tilde V - \phi\circ \tilde V + b \one_C,
\end{equation}
then, provided $R_\lambda \tilde V$ is continuous, there exists some constant $c > 0$ such that 
\begin{equation}\label{eq:subgeo rate}
\lVert P_t(x,\cdot) - \mu \rVert_{\mathrm{TV}} \leq c\mathcal{R}_\lambda \tilde V(x) \Xi(t), \quad t \geq 0, x \in \mathcal{X},
\end{equation}
where $\Xi(t) = 1/(\lambda\phi \circ H_{\lambda\phi}^{-1})(t)$ for $H_{\lambda\phi}(t) = \int_1^t 1/(\lambda \phi(s)) \diff{s}.$ Note that \eqref{eq: exp ergodic1} can be recovered for linear $\phi$, in which case $\Xi(t) = \mathrm{e}^{-\kappa t}$ for some $\kappa > 0$, and hence exponential ergodicity can be regarded as a special case of this general result.

Studying exponential and subgeometric convergence is not only interesting in its own right, but does have direct implications on the mixing behavior of the Markov process, which we are ultimately going after in this article. For two $\sigma$-algebras $\cG$ and $\cH$ and a given probability measure $\mathbf{P}$, introduce the $\beta$-mixing coefficient
\begin{equation} \label{eq: beta mix coeff}
\beta_{\mathbf{P}}(\cG,\cH) \coloneq  \sup_{C \in \mathcal{G} \otimes \mathcal{H}} \big\lvert\mathbf{P}\vert_{\mathcal{G} \otimes \mathcal{H}}(C) - \mathbf{P}\vert_{\mathcal{G}} \otimes \mathbf{P}\vert_{\mathcal{H}}(C)\big \rvert,
\end{equation}
where $\mathbf{P}\vert_{\mathcal{G}\otimes \mathcal{H}}$ is the restriction to $(\Omega \times \Omega, \mathcal{G} \otimes \mathcal{H})$ of the image measure of $\mathbf{P}$ under the canonical injection $\iota(\omega) = (\omega,\omega).$ Noting that for $A \times B \in \mathcal{G} \otimes \mathcal{H}$, it holds that $\mathbf{P}\vert_{\mathcal{G}\otimes \mathcal{H}}(A\times B) = \mathbf{P}(A \cap B)$,
it is clear that the $\beta$-mixing coefficient should be interpreted as a measure of independence of the $\sigma$-algebras.
For the Markov process $\mathbf{X}$ with natural filtration $\F^0 = (\cF^0_t)_{t \geq 0}$ and a given initial distribution $\eta$ let us now define
\begin{equation} \label{eq: beta mix coeff2}
\beta(\eta, t) = \sup_{s \geq 0} \beta_{\PP^\eta}(\cF^0_s,\overbar{\cF}{}^0_{s+t}), \quad t > 0,
\end{equation}
where we denoted by $\overbar{\cF}{}^0_t = \sigma(X_s,s\geq t)$ the $\sigma$-algebra of the future after time $t$. We then say that $\mathbf{X}$ is $\beta$-mixing when started in $\eta$, if $\lim_{t \to \infty}\beta(\eta,t) = 0$ 
Hence, if $\mathbf{X}$ is $\beta$-mixing we can roughly state that there is an asymptotic independence between the past and the future of the Markov process. If there even exist constants $C,\kappa > 0$ such that $\beta(\eta,t) \leq C\mathrm{e}^{-\kappa t}$, we call $\mathbf{X}$ \textit{exponentially} $\beta$-mixing.

\cite[Lemma 1.4]{volkonskii1961} gives
$$\beta_{\PP^\eta}(\cF^0_s, \overbar{\cF}{}^0_{t+s}) = \E^\eta\Big[\sup_{B \in \overbar{\cF}^{0}_{t+s}}\lvert \PP^\eta(B \vert \cF^0_s) - \PP^\eta(B) \rvert \Big].$$
Proposition 1 in \cite{davydov1973} therefore demonstrates that
$$\beta(\eta,t) = \sup_{s \geq 0} \int_{\cX} \lVert \PP^x(X_t \in \cdot) - \PP^\eta(X_{t+s} \in \cdot)\rVert_{\mathrm{TV}} \, \PP^\eta(X_s \in \diff{x}), \quad t > 0.$$
Masuda \cite[Lemma 3.9]{masuda2007} uses this characterization to establish that if we have (sub)geometric decay as in \eqref{eq:subgeo rate} for $\mathbf{X}$ and moreover
\begin{equation}\label{eq: masuda criterion}
\varrho(\eta)\coloneq \sup_{t \geq 0} c \E^\eta[R_\lambda\tilde{V}(X_t)] < \infty,
\end{equation}
then $\mathbf{X}$ started in $\eta$ is $\beta$-mixing at rate $\Xi(t)$ with
$$\beta(\eta,t) \leq 2 \varrho(\eta) \Xi(t), \quad t > 0.$$

\section{Proof of the resolvent formula} \label{app: resolvent}
\begin{proof}[Proof of Theorem \ref{resolvent overshoot}]
Note first that by assumed irreducibility  of $J^+$, it follows as a consequence of the Perron--Frobenius theorem that $\bm{\Phi}^+(\lambda)$ is invertible for any $\lambda > 0$, see Corollary 2.4 in Stephenson \cite{stephenson2018} or Remark 2.2 in Ivanovs et al. \cite{ivanovs2010}, and hence the statement of the theorem makes formally sense. Fix $(x,i) \in \R_+ \times [n]$. Let $\tau_0 \coloneq \inf\{t \geq 0: \cO_t = 0\}$, which is clearly finite and a stopping time for $(\cO_t,\cJ_t)$ since the process is Feller by Proposition \ref{overshoot feller}. By the sawtooth structure of $(\cO,\cJ)$, see also Figure \ref{fig: overshoot plot} for an illustration, we have  $\tau_0 = x$ and $(\cO_t,\cJ_t) = (x- t,i)$ for $t \in [0,x]$, $\PP^{x,i}$-a.s.. Together with the strong Markov property of $(\cO,\cJ)$, we therefore obtain for $f \in \cB_b(\R_+ \times [n])$
$$\mathcal{U}_\lambda f(x,i) = Q_\lambda f(x,i) + \mathrm{e}^{-\lambda x} \mathcal{U}_\lambda f(0,i).$$
Hence, we only need to calculate $\mathcal{U}_\lambda f(0,i)$.

We start with the case that the L\'evy measures $\Pi^+_i$, $i \in [n]$ are finite and then proceed by an approximation argument to the general case. Our assumption of upward regularity of $(\xi,J)$ then forces $d^+_i > 0$ for all $i \in [n]$, that is the processes $H^+_i$ are compound Poisson processes with drift. Denote for $i \in [n]$ by $Y^{(i)}$ random variables independent of $(\xi,J)$ corresponding to the jumps of $H^{+,(i)}$, whose distribution is given by $\Pi_i^+(\diff{x})\slash \Pi_i^+(\R_+)$. Moreover, denote by $\sigma \coloneq \inf\{t \geq 0: J^+_t \neq J^+_0\}$ the first jump time of $J^+$ and by $\tau = \inf\{t \geq 0: \Delta H{}^{+,0,J_0^+}_t > 0\}$ the first jump time of the L\'evy process driving the ascending ladder height process before the first phase transition. Then, from Proposition \ref{char map levy} and indistinguishability of $(\cO^+,\cJ^+)$ and $(\cO,\cJ)$ we can infer that under $\PP^{0,i}$, it holds that $T \coloneq \inf\{t \geq 0 : \Delta(\cO_t, \cJ_t) \neq 0\} = H{}^{+,0,i}_{(\sigma \wedge \tau)- }  = d^+_i(\sigma \wedge \tau)$ almost surely (consult again Figure \ref{fig: overshoot plot} for an illustration).
\begin{equation}\label{split integral}
\mathcal{U}_\lambda f(0,i) = \E^{0,i}\Big[\int_0^T + \int_T^{T + \tau_0 \circ \theta_T} + \int_{T + \tau_0 \circ \theta_T}^\infty \mathrm{e}^{-\lambda t} f(\cO_t,\cJ_t) \diff{t}\Big] \eqcolon I_1 + I_2 + I_3,
\end{equation}
where $(\theta_t)_{t \geq 0}$ denotes the transition operator of $(\cO,\cJ)$. Since under $\PP^{0,i}$, $\tau \overset{\mathrm{d}}{=} \mathrm{Exp}(\Pi_i^+(\R_+))$ is independent of $\sigma \overset{\mathrm{d}}{=} \mathrm{Exp}(-q_{i,i}^+)$ by Proposition \ref{char map levy}, it follows that $T \overset{\mathrm{d}}{=}\mathrm{Exp}((\Pi_i^+(\R_+) -q^+_{i,i})\slash d_i^+)$  and hence
\begin{align*}
I_1 = \E^{0,i}\Big[\int_0^T \mathrm{e}^{-\lambda t} f(0,i) \diff{t}\Big] = f(0,i)\frac{1}{\lambda}\big(1-\E^{0,i}[\mathrm{e}^{-\lambda T}]\big) &= f(0,i)\frac{1}{\lambda}\Big(1- \frac{\Pi_i^+(\R_+) - q_{i,i}^+}{d_i^+\lambda + \Pi_i^+(\R_+) - q_{i,i}^+}\Big)\\
&= f(0,i)\frac{d_i^+}{d_i^+\lambda + \Pi_i^+(\R_+) - q_{i,i}^+}.
\end{align*}
For the second integral, we use that $\PP^{0,i}(J^+_\sigma = j) =  - q^+_{i,j}\slash q^+_{i,i}$, independence of $\sigma, J^+_\sigma$ and $Y^{(i)}$ in combination with Proposition \ref{char map levy} and  the strong Markov property to obtain
\begin{align*}
  I_2 &= \E^{0,i} \Big[\mathrm{e}^{-\lambda T} \E^{0,i}\Big[\int_0^{\tau_0} \mathrm{e}^{-\lambda t} f(\cO_t,\cJ_t) \diff{t} \circ \theta_T \Big\vert \cG_T\Big]\Big] \\
  &= \E^{0,i}\Big[\mathrm{e}^{-\lambda  T} \E^{\cO_T,\cJ_T}\Big[\int_0^{\tau_0} \mathrm{e}^{-\lambda t} f(\cO_t,\cJ_t) \diff{t}\Big]\Big]\\
  &= \E^{0,i}\big[\mathrm{e}^{-\lambda d_i^+ \tau} Q_\lambda(Y^{(i)},i)\,;\, \tau < \sigma\big] + \E^{0,i}\big[\mathrm{e}^{-\lambda d_i^+ \sigma} Q_\lambda\big(\Delta_{i,J^+_\sigma}^{+,1},J^+_\sigma\big)\,;\, \sigma < \tau\big]\\
  &= \E^{0,i}[\mathrm{e}^{-\lambda d_i^+ \tau}\,;\, \tau < \sigma]\, \E^{0,i}[Q_\lambda(Y^{(i)},i)] + \E^{0,i}\big[\mathrm{e}^{-\lambda d_i^+ \sigma}\,;\, \sigma < \tau] \,\E^{0,i}\big[Q_\lambda\big(\Delta_{i,J^+_\sigma}^{+,1},J^+_\sigma\big)\big]\\
  &= \frac{\Pi_i^+(\R_+)}{\lambda d_i^+ + \Pi_i^+(\R_+) - q_{i,i}^+} \int_0^\infty Q_\lambda f(y,i) \, \Pi_i^+(\diff{y}) \slash \Pi_i^+(\R_+) \\
  &\quad + \frac{-q_{i,i}^+}{\lambda d_i^+ + \Pi_i^+(\R_+) - q_{i,i}^+} \sum_{j \neq i} \frac{q_{i,j}^+}{-q_{i,i}^+}  \int_0^\infty Q_\lambda f(y,j) \, F^+_{i,j}(\diff{y)}\\
  &= \frac{1}{\lambda d_i^+ + \Pi_i^+(\R_+) - q_{i,i}^+}\Big(\int_0^\infty Q_\lambda f(y,i) \, \Pi_i^+(\diff{y}) + \sum_{j \neq i}  q_{i,j}^+\int_0^\infty Q_\lambda f(y,j) \, F^+_{i,j}(\diff{y)} \Big)
\end{align*}
With the same arguments as above we also obtain
\begin{align*}
  I_3 &= \E^{0,i}\Big[\mathrm{e}^{-\lambda T} \E^{0,i}\Big[\int_{\tau_0}^\infty \mathrm{e}^{-\lambda t} f(\cO_t,\cJ_t) \diff{t}\Big \vert \cG_T \Big] \Big]\\
  &= \E^{0,i}\Big[\mathrm{e}^{-\lambda T} \E^{\cO_T,\cJ_T}\Big[\int_{\tau_0}^\infty \mathrm{e}^{-\lambda t} f(\cO_t,\cJ_t) \diff{t}\Big]\Big]\\
  &= \E^{0,i}[\mathrm{e}^{-\lambda d^+_i\tau}\,;\, \tau < \sigma]\, \E^{0,i}\Big[\E^{y,i}\Big[\mathrm{e}^{-\lambda \tau_0}\E^{y,i}\Big[\int_0^\infty \mathrm{e}^{-\lambda t} f(\cO_t,\cJ_t) \diff{t} \circ \theta_{\tau_0} \Big \vert \cG_{\tau_0}\Big]\Big]\Big\vert_{y = Y^{(i)}}\Big]\\
  &\quad+ \E^{0,i}[\mathrm{e}^{-\lambda d^+_i\sigma}\,;\, \sigma < \tau] \sum_{j \neq i}\frac{q^+_{i,j}}{-q_{i,i}^+}\E^{0,i}\Big[\E^{y,j}\Big[\mathrm{e}^{-\lambda \tau_0}\E^{y,j}\Big[\int_0^\infty \mathrm{e}^{-\lambda t} f(\cO_t,\cJ_t) \diff{t} \circ \theta_{\tau_0} \Big \vert \cG_{\tau_0}\Big]\Big]\Big\vert_{y = \Delta_{i,j}^{+,1}}\Big]\\
  &= \frac{\Pi_i^+(\R_+)}{\lambda d_i^+ + \Pi_i^+(\R_+) - q_{i,i}^+} \mathcal{U}_\lambda f(0,i)\E^{0,i}\big[\mathrm{e}^{-\lambda Y^{(i)}}\big]\\
  &\quad+ \frac{1}{\lambda d^+_i + \Pi_i^+(\R_+) - q_{i,i}^+} \sum_{j \neq i}q_{i,j}^+ \mathcal{U}_\lambda f(0,j)\E^{0,i}\big[\mathrm{e}^{-\lambda \Delta_{i,j}^{+,1}}\big]\\
  &= \frac{1}{\lambda d_i^+ + \Pi_i^+(\R_+) - q_{i,i}^+}\Big(\mathcal{U}_\lambda f(0,i) \int_0^\infty \mathrm{e}^{-\lambda y} \, \Pi_i^+(\diff{y}) + \sum_{j \neq i} q^+_{i,j} \mathcal{U}_\lambda f(0,j) \int_0^\infty \mathrm{e}^{-\lambda y} \, F^+_{i,j}(\diff{y})  \Big).
\end{align*}
Plugging into \eqref{split integral}, using $G^+_{ij}(\lambda) = \int_0^\infty \exp(-\lambda y)\, F_{ij}^+(\diff{y})$ and rearranging now yields
\begin{align*}
  &\mathcal{U}_\lambda f(0,i)\Big(d_i^+ \lambda + \int_0^\infty (1 - \mathrm{e}^{-\lambda y}) \, \Pi_i^+(\diff{y}) - q_{i,i}^+\Big) - \sum_{j \neq i} q_{i,j}^+ G^+_{ij}(\lambda) \mathcal{U}_\lambda f(0,j)\\
  &\qquad= d_i^+ f(0,i) + \int_0^\infty Q_\lambda f(y,i) \, \Pi_i^+(\diff{y}) + \sum_{j \neq i} q_{i,j}^+ \int_0^\infty Q_\lambda f(y,i) \, F^+_{i,j}(\diff{y}).
\end{align*}
By \eqref{char exp map} the left hand side is equal to
$$\mathcal{U}_\lambda f(0,i) (\Phi_i^+(\lambda) -q_{i,i}^+) - \sum_{j \neq i} q_{i,j}^+ G^+_{ij}(\lambda) \mathcal{U}_\lambda f(0,j) = \big(\bm{\Phi}^+(\lambda) \cdot(\mathcal{U}_\lambda f(0,j))_{j=1,\ldots,n}^\top\big)_i$$
and hence we conclude that
\begin{equation} \label{resolvent compound poisson}
(\mathcal{U}_\lambda f(0,i))_{i=1,\ldots,n}^\top = \bm{\Phi}^+(\lambda)^{-1} \cdot \Big(d^+_i f(0,i) + \int_0^\infty Q_\lambda f(x,i) \,\Pi_i^+(\diff{x}) + \sum_{j \neq i} q^+_{i,j}\E[Q_\lambda f(\Delta^+_{i,j},j)]\Big)^\top_{i=1,\ldots,n},
\end{equation}
which proves the assertion in case that $(H^+, J^+)$ is a compound Poisson Markov additive subordinator. For the general case, suppose that $(\xi,J)$ is an upward regular MAP and let for $\varepsilon > 0$, $(^\varepsilon \! H^{+},J^+)$ be the MAP subordinator corresponding to the ordinator constructed from the L\'evy subordinators $^\varepsilon\! H^{+,(i)}$ defined by
$$^\varepsilon\! H^{+,(i)}_t \coloneq (d^+_i + \varepsilon)t + \sum_{s \leq t} \Delta H^{+,(i)}_s \one_{(\varepsilon,\infty)}(\Delta H^{+,(i)}_s), \quad t \geq 0,$$
i.e.\ $^\varepsilon\! H^{+,(i)}$ is obtained from $H^+$ by deleting jumps smaller than $\varepsilon$ and adding an additional drift $\varepsilon$. This ensures that $^\varepsilon\! H^{+,(i)}$ is a compound Poisson subordinator with drift $d_i^+ + \varepsilon$ and L\'evy measure $\Pi_i^{+,\varepsilon} = \Pi_i^+(\cdot \cap (\varepsilon,\infty))$ and hence we may apply \eqref{resolvent compound poisson} for the $\lambda$-resolvent of the overshoot process
$$(^\varepsilon\! \cO^+_t, {^\varepsilon\! \cJ_t^+})_{t \geq 0} \coloneq (^\varepsilon\! H^+_{T^{+,\varepsilon}_t}-t, J^+_{T^{+,\varepsilon}_t})_{t \geq 0},$$
where $T^{+,\varepsilon}_t \coloneq \inf\{s \geq 0 : {^\varepsilon\! H^+_s} > t\},$ $t \geq 0.$ We first observe that for any $t > 0$ we obtain from Proposition \ref{char map levy}
$$\sup_{s\leq t} \lvert {^\varepsilon\! H^{+}_s} - H_s^+ \rvert \leq \varepsilon t + \sum_{s \leq t} \Delta H_s^+ \one_{\{\Delta H^+_s < \varepsilon\}},$$
and since $\sum_{s \leq t} \Delta H_s^+$ converges we obtain by dominated convergence that almost surely
$$\sup_{s\leq t} \lvert {^\varepsilon\! H^{+}_s} - H_s^+ \rvert \to  0, \quad \text{as } \varepsilon \downarrow 0$$
i.e.\ $^\varepsilon \! H^+$ converges to $H^+$ uniformly on compact sets almost surely as $\varepsilon \downarrow 0$. Let $\Xi$ be the set of $\PP$-measure $1$ on which $^\varepsilon\! H^+$ and $(H^+,J^+)$ have càdlàg paths and on which the above convergence holds. Let $\omega \in \Xi$. Then $^\varepsilon\! H^+_\cdot(\omega), H^+_\cdot(\omega) \in \mathcal{D}(\R_+)$, the space of càdlàg functions mapping from $\R_+$ to $\R_+$, which we endow with Skorokhods $J_1$-topology. Since $^\varepsilon\! H^+_\cdot(\omega)$ converges uniformly on compact time sets to $H^+_\cdot(\omega)$, Proposition VI.1.17 in \cite{jacod2003} tells us that $^\varepsilon\! H^+_\cdot(\omega)$ also converges with respect to the metric inducing the Skorokhod topology to $H^+_\cdot(\omega)$. For $t \geq 0$ let
$$ S_t\colon \mathcal{D}(\R_+) \to [0,\infty], \quad \alpha \mapsto \inf\{s \geq 0: \lvert \alpha(s) \rvert \geq t \text{ or } \lvert \alpha(s-)\rvert \geq t\}.$$
Since $^\varepsilon \! H^+_\cdot(\omega)$ and $H^+_\cdot(\omega)$ are strictly increasing it follows that $T^{+,\varepsilon}_t(\omega) = S_t(^\varepsilon \! H^+_\cdot(\omega))$ and $T^+_t(\omega) = S_t(H^+_\cdot(\omega))$. Moreover, the set $\{ t > 0: S_t(H^+_\cdot (\omega)) \neq S_{t+}(H^+_\cdot (\omega))\}$ is empty by strictly increasing paths of $H^+_\cdot(\omega)$. Hence, we obtain from Proposition 2.11 and the proof of part c) of Proposition VI.2.12 in \cite{jacod2003} that
\begin{equation} \label{conv hitting}
T_t^{+,\varepsilon}(\omega) = S_t(^\varepsilon \! H^+_\cdot(\omega)) \to S_t(H^+_\cdot(\omega)) = T_t^+(\omega), \quad \text{ as } \varepsilon \downarrow 0,
\end{equation}
and that for $t \notin \Lambda(\omega) = \{t > 0: \Delta H_{T_{t}^+}^+(\omega) > 0 \text{ and } H_{T_{t}^+ -}^+(\omega) = t\}$ we have
\begin{equation}\label{conv overshoot}
^\varepsilon\! H^+_{T^{+,\varepsilon}_t}(\omega) \to H^+_{T^+_t}(\omega), \quad \text{ as } \varepsilon \downarrow 0.
\end{equation}
But from the sawtooth structure of the paths of $\cO$ it is easy to see that $\Lambda(\omega) = \{t > 0: \Delta \cO^+_t(\omega) > 0\}$, which is countable (alternatively, see Lemma VI.2.10.(d) in \cite{jacod2003} for the same conclusion), hence non-convergence of $^\varepsilon \!\cO_t^{+}(\omega)$ to $\cO^+_t(\omega)$ only takes place on a set of Lebesgue measure $0$. Furthermore, from \eqref{conv hitting} it follows that $^\varepsilon \! \cJ^+_t(\omega)$ converges to $\cJ^+_t(\omega)$ as $\varepsilon \downarrow 0$ except possibly on the set
\begin{align*}
\Lambda^\prime(\omega) &\coloneq \{t > 0: J^+_{T^+_t}(\omega) \neq J^+_{T^+_t -}(\omega)\}\\
 &= \{t > 0: \Delta J^+_{T^+_t}(\omega) \neq 0, \Delta H^+_{T^+_t}(\omega) > 0 \} \cup \{t > 0: \Delta J^+_{T^+_t}(\omega) \neq 0, \Delta H^+_{T^+_t}(\omega) = 0 \}\\
&\eqcolon \Lambda^\prime_1(\omega) \cup \Lambda^\prime_2(\omega).
\end{align*}
For $t \in \Lambda^\prime_1(\omega)$ we have that in case $H^+_{T^+_{t}-}(\omega) < t  \leq H^+_{T^+_t}(\omega)$ it holds that $T^+_s(\omega) = T^+_t(\omega)$ for $s \in [H_{T^+_t -}(\omega), t]$. Right-continuity of $s \mapsto T^+_s(\omega)$ and $s \mapsto J^+_s(\omega)$ therefore imply that
for such $t$ we also have $^\varepsilon \! \cJ^+_t(\omega) \to \mathcal{J}_t^+(\omega)$ as $\varepsilon \downarrow 0$. Further, since $t \mapsto H_t^+(\omega)$ is continuous in $T^+_t(\omega)$ if $\Delta H^+_{T^+_t}(\omega) = 0$, it follows from strictly increasing paths that for $s,t \in \Lambda^\prime_2(\omega)$ we have $T_s^+(\omega) \neq T^+_t(\omega)$. Hence, $t \mapsto T^+_t(\omega)$ is injective on $\Lambda^\prime_2(\omega)$. Since
$$T^+_\cdot(\omega)(\Lambda^\prime_2(\omega)) = \{t > 0: \Delta J^+_t(\omega) \neq 0, \Delta H^+_t(\omega) = 0\} \subset \{t > 0: \Delta J^+_t(\omega) \neq 0\},$$
and the set on the right-hand side is countable thanks to $J^+_\cdot(\omega)$ being càdlàg, it follows that $\Lambda^\prime_2(\omega)$ is countable as well. The above discussion therefore yields that the set of times $t > 0$ for which $J^+_{T^{+,\varepsilon}_t}(\omega)$ does not converge to $\mathcal{J}^+_t(\omega)$ is given by
$$\Lambda^{\prime \prime}(\omega) \coloneq \{t > 0: \Delta J^+_{T^+_t}(\omega) \neq 0, H^+_{T^+_t-}(\omega) = t < H^+_{T^+_t}(\omega)\} \cup \Lambda^\prime_2(\omega) \subset \{t > 0: \Delta \cO^+_t(\omega) > 0\} \cup \Lambda^\prime_2(\omega)$$
is countable and therefore has Lebesgue measure $0$ as well. It follows that for any $\omega \in \Xi$ we have for $f \in \mathcal{C}_b(\R_+ \times [n])$ by dominated convergence
\begin{align*}
\lim_{\varepsilon \downarrow 0} \int_0^\infty f(^\varepsilon \! \cO^+_t(\omega), {^\varepsilon \! \cJ^+_t}(\omega)) \diff{t} &= \int_{(\Lambda(\omega) \cup \Lambda^{\prime \prime}(\omega))^\mathrm{c}} \lim_{\varepsilon \downarrow 0} f(^\varepsilon \! \cO^+_t(\omega), ^\varepsilon \! \cJ^+_t(\omega)) \diff{t}\\
&= \int_{(\Lambda(\omega) \cup \Lambda^{\prime \prime}(\omega))^\mathrm{c}}  f(\cO^+_t(\omega), \cJ^+_t(\omega)) \diff{t}\\
&= \int_0^\infty f(\cO^+_t(\omega), \cJ^+_t(\omega))\diff{t}.
\end{align*}
Consequently, if we denote by $U^\varepsilon_\lambda$ the $\lambda$-resolvent for $(^\varepsilon \! \cO^+,^\varepsilon \! \cJ^+)$, the set $\Xi$ having $\PP$-measure $1$ implies that for any $f \in \mathcal{C}_b(\R_+ \times [n])$
\begin{align*}
&(\mathcal{U}_\lambda f(0,i))_{i =1,\ldots,n}\\
&\quad= \Big(\int_\Xi \lim_{\varepsilon \downarrow 0} \int_0^\infty f(^\varepsilon \! \cO^+_t(\omega), {^\varepsilon \! \cJ^+_t}(\omega)) \diff{t} \, \PP^{0,i}(\diff{\omega})\Big)_{i =1,\ldots,n} \\
&\quad= \lim_{\varepsilon \downarrow 0 }\Big(\int_\Xi \int_0^\infty f(^\varepsilon \! \cO^+_t(\omega), {^\varepsilon \! \cJ^+_t}(\omega)) \diff{t}\, \PP^{0,i}(\diff{\omega}) \Big)_{i =1,\ldots,n}\\
&\quad= \lim_{\varepsilon \downarrow 0} (U^\varepsilon_\lambda f(0,i))_{i =1,\ldots,n}\\
&\quad= \lim_{\varepsilon \downarrow 0} {^\varepsilon \! \bm{\Phi}^+(\lambda)^{-1}} \cdot \Big((d^+_i + \varepsilon) f(0,i) + \int_\varepsilon^\infty Q_\lambda f(x,i) \,\Pi_i^+(\diff{x}) + \sum_{j \neq i} q^+_{i,j}\E[Q_\lambda f(\Delta^+_{i,j},j)]\Big)^\top_{i=1,\ldots,n}\\
&\quad= \bm{\Phi}^+(\lambda)^{-1} \cdot \Big(d^+_i f(0,i) + \int_0^\infty Q_\lambda f(x,i) \,\Pi_i^+(\diff{x}) + \sum_{j \neq i} q^+_{i,j}\E[Q_\lambda f(\Delta^+_{i,j},j)]\Big)^\top_{i=1,\ldots,n}\\
&\quad \eqcolon \bm{\Upsilon}(\lambda),
\end{align*}
where we used dominated convergence for the second and \eqref{resolvent compound poisson} for the fourth equality. It remains to extend this result to any $f \in \cB_+(\R_+ \times [n]) \cup \cB_b(\R_+ \times [n])$. To this end, let
$$\mathcal{M} \coloneq \big\{f \in \cB_b(\R_+ \times [n]): (\mathcal{U}_\lambda f(0,i))_{i =1,\ldots,n} = \bm{\Upsilon}(\lambda)\big\}.$$
Clearly, $\mathcal{M}$ is a vector space over $\R_+$ by linearity of the Lebesgue integral and since $\cC_b(\R_+ \times [n]) \subset \mathcal{M}$, the constant function $\one_{\R_+ \times [n]}$ is contained in $\mathcal{M}$. Moreover, dominated convergence shows that $\mathcal{M}$ is closed under convergence of an increasing family of functions $f_n$ converging to some $f \in \cB_b(\R_+ \times [n])$. Hence, $\mathcal{M}$ is a monotone vector space and since $\cC_b(\R_+ \times [n])$ is closed under multiplication and contained in $\mathcal{M}$, the functional Monotone Class Theorem A.0.6 from \cite{sharpe1988} implies that $\sigma(\mathcal{C}_b(\R_+ \times [n])) \subset \mathcal{M}$. But since $\R_+ \times [n]$ is a locally compact metric space with countable base, $\mathcal{C}_b(\R_+ \times [n])$ generates $\mathcal{B}_b(\R_+ \times [n])$ and hence $\mathcal{M} = \cB_b(\R_+ \times [n])$ follows.
For general $f \in \cB_+(\R_+ \times [n])$ let $f_n \coloneq f \one_{\{f \in [0,n]\}} \in \cB_b(\R_+ \times [n])$ and apply monotone convergence to deduce that \eqref{resolvent compound poisson} also holds for $f \in \cB_+(\R_+ \times [n]).$ This finishes the proof.
\end{proof}
\end{appendix}

\printbibliography

\end{document}